\documentclass[letter,11pt]{amsart}
\usepackage{amsfonts,textcomp,amssymb,amsmath,amsthm}
\usepackage{color,ulem}
\usepackage{fullpage}
\usepackage{graphicx}
\usepackage[active]{srcltx}
\usepackage{mathabx}
\usepackage{dsfont,mathrsfs,stmaryrd,wasysym,amsbsy}
\usepackage{enumerate}
\usepackage{IEEEtrantools}
\usepackage{resmes}

\usepackage{tikz}
\usetikzlibrary{positioning}

\newcommand{\cE}{\mathcal{E}}

\newcommand{\cH}{\mathcal{H}}

\newcommand{\cY}{\mathcal{Y}}

\newcommand{\EE}{\mathbb{E}}
\newcommand{\E}{\mathbb{E}}

\newcommand{\NN}{\mathbb{N}}

\newcommand{\PP}{\mathbb{P}}
\newcommand{\QQ}{\mathbb{Q}}
\newcommand{\RR}{\mathbb{R}}
\newcommand{\rr}{\mathbb{R}}

\newcommand{\bone}{\mathbf{1}}

\theoremstyle{plain}
\newtheorem{theorem}{Theorem}[section]
\newtheorem{corollary}[theorem]{Corollary}
\newtheorem{lemma}[theorem]{Lemma}
\newtheorem{proposition}[theorem]{Proposition}

\theoremstyle{definition}
\newtheorem{remark}[theorem]{Remark}
\newtheorem{definition}[theorem]{Definition}

\numberwithin{equation}{section}

\begin{document}
\title[Cascade PDE]{Cascade equation for the discontinuities in the Stefan problem with surface tension} 
\author{Yucheng Guo, Sergey Nadtochiy and Mykhaylo Shkolnikov}
\address{ORFE Department, Princeton University, Princeton, NJ 08544.} 
\email{yg7348@princeton.edu}
\address{Department of Applied Mathematics, Illinois Institute of Technology, Chicago, IL 60616.}
\email{snadtochiy@iit.edu}
\address{Department of Mathematical Sciences and Center for Nonlinear Analysis, Carnegie Mellon University, PA 15232.}
\email{mshkolni@gmail.com}
\footnotetext[1]{S.~Nadtochiy is partially supported by the NSF CAREER grant DMS-1651294.}
%\footnotetext[2]{M.~Shkolnikov is partially supported by the NSF grant DMS-2108680.}

\begin{abstract}
The Stefan problem with surface tension is well known to exhibit discontinuities in the associated moving aggregate (i.e., in the domain occupied by the solid), whose structure has only been understood under translational or radial symmetry so far.~In this paper, we derive an auxiliary partial differential equation of second-order hyperbolic type, referred to as the cascade equation, that captures said discontinuities in the absence of any symmetry assumptions.~Specializing to the one-phase setting, we introduce a novel (global) notion of weak solution to the cascade equation, which is defined as a limit of mean-field game equilibria.~For the spatial dimension two, we show the existence of such a weak solution and prove a natural perimeter estimate on the associated moving aggregate.
\end{abstract}

\maketitle

%!TEX root =Main.tex

\section{Introduction} \label{se:intro}

The classical formulation of the Stefan problem with surface tension can be stated as follows: Given $\Gamma_0\subset\rr^d$ and $u(0,\cdot)\!:\rr^d\to\rr$, find $\{\Gamma_s\subset\rr^d\}_{s>0}$ and $\{u(s,\cdot)\!\!:\rr^d\to\rr\}_{s>0}$ such that
\begin{equation}\label{SGT}
	\begin{cases}
		& \partial_s u(s,x) = \frac{1}{2}\Delta_x u(s,x),\quad x\in\rr^d\backslash\partial\Gamma_s,\;\; s>0, \\
		& u(s,x) = -\gamma H_s(x),\quad x\in\partial\Gamma_s,\;\; s>0, \\
		& V_s(x) = -\frac{1}{2}\partial_{\mathbf{n}_s(x)} u(s,x) - \frac{1}{2}\partial_{-\mathbf{n}_s(x)} u(s,x),\quad x\in\partial\Gamma_s,\;\; s\ge0.
	\end{cases}
\end{equation}
Here, $\{\Gamma_s\}_{s\ge0}$ is interpreted as the evolution of a solid (e.g., ice) crystal that is surrounded by the corresponding liquid (e.g., water); $\{u(s,\cdot)\}_{s\ge0}$ describes the temperature distribution relative to the equilibrium freezing point, which evolves according to the standard heat equation; $H_s(x)$ is the mean curvature of $\Gamma_s$ at the point $x\in\partial\Gamma_s$; the ``surface tension coefficient'' $\gamma>0$ quantifies the ``Gibbs-Thomson relation'' $u(s,x)=-\gamma H_s(x)$ reflecting the shift of the equilibrium freezing point by $-\gamma H_s(x)$ at a curved interface; $V_s(x)$ is the outward (relative to $\Gamma_s$) normal speed of the point $x\in\partial\Gamma_s$ at the time~$s$; and $\partial_{\mathbf{n}_s(x)} u(s,x)$ (resp., $\partial_{-\mathbf{n}_s(x)} u(s,x)$) are the outward (resp., inward) normal derivatives of $u(s,\cdot)$ at the point $x\in\partial\Gamma_s$. 

\medskip

While the mathematical study of Stefan problems goes as far back as \cite{LC} and \cite{Stefan1,Stefan2,Stefan3,Stefan4}, the well-posedness of the problem \eqref{SGT} is still open, despite the regularizing effect of the Gibbs-Thomson relation on $\partial\Gamma_s$.~Thereby, a crucial challenge is posed by the (generic) discontinuities of $\{\Gamma_s\}_{s\ge0}$ in $s$, see, e.g., \cite[Section 1]{GNS}, \cite[Section 1]{dns}, and \cite[Section~5]{Lu} for discussions in the radially symmetric, the translationally symmetric, and the general case, respectively.~In fact, the behavior of $(\Gamma,u)$ at the times of such discontinuities, a.k.a.~the times of jumps, is not uniquely determined by the system \eqref{SGT}, and the appropriate choice of such behavior turns out to be crucial for obtaining a notion of solution that enjoys uniqueness. Assuming, for simplicity, that a jump occurs at time $s$ and that $\Gamma_s \supset \Gamma_{s-}:=\bigcup_{0\le r<s} \Gamma_r$, one can describe the \textit{energy solution} introduced in \cite{Lu} as corresponding to the assumption that each point in the jump region $\Gamma_s\setminus\Gamma_{s-}$ instantly changes its phase -- from liquid to solid -- and adjusts its temperature so that the overall phase-temperature energy at this point does not change while the boundary condition (i.e., the second line of \eqref{SGT}) may be violated. This choice of jump time behavior leads to obvious non-uniqueness of a solution, even in the radially symmetric case.~In contrast, the \textit{physical solution} introduced in \cite{DIRT2} and \cite{dns} is obtained by viewing each jump as a limit of a family of aggregates (i.e., a \textit{cascade}) evolving continuously but on a ``fast" time scale. This leads to a relaxation of the energy preservation condition -- it is only preserved globally, over the entire jump region, as opposed to locally, at each point -- but enforces the boundary condition (i.e., the second line of \eqref{SGT}) at each point of the jump region. In the radially symmetric case, the physical solution enjoys both existence and uniqueness, see \cite{NaShsurface}, \cite{GNS}.
However, to date, the notion of physical solution has not been extended beyond the radially symmetric case, and the main obstacle is the lack of an appropriate definition of a \textit{physical jump} without the assumption of radial symmetry.

\medskip

The main subject of this paper is the extension of the notion of a physical jump for solutions to the Stefan problem with surface tension to the general multidimensional setting in the absence of any symmetry assumptions. More specifically, we analyze the discontinuities of $\{\Gamma_s\}_{s\ge0}$ in $s$ when $\Gamma_s$ is non-decreasing in $s$ (``one-phase setting''). We argue in Section \ref{se:form} that, given $\Gamma_{s-}:=\bigcup_{0\le r<s} \Gamma_r$ and $u(s-,\cdot):=\lim_{r\uparrow s} u(r,\cdot)$, the set $\Gamma_s$ is obtained by, first, (globally) solving the auxiliary partial differential equation (PDE)
\begin{equation}\label{PDE:wave}
	\mathrm{div}\bigg(\frac{\nabla w(x)}{|\nabla w(x)|}\,\bigg(\frac{1}{|\nabla w(x)|}+\gamma\bigg)\!\bigg)
	=-1-u(s-,x),
\end{equation}
for $x\in\rr^d\backslash\Gamma_{s-}$, subject to the boundary conditions $w(x)=0$ and $\partial_{\mathbf{n}_{s-}(x)} w(x)=1/V_0$ along $\partial\Gamma_{s-}$, for any given initial speed function $V_0\geq0$, then, considering the limiting $w$ as $V_0\downarrow0$ and letting $\Gamma_s:=\{x\!:w(x)<\infty\}$ (where we extend $w$ to $\RR^d$ by setting $w\equiv0$ on $\Gamma_{s-}$). In the present paper, we focus on the first part of this program -- i.e., on constructing a global solution to \eqref{PDE:wave} for a given~$V_0$. As alluded to in the previous paragraph, the formulation \eqref{PDE:wave} is inspired by the main idea of \cite[Subsection 3.1.2]{DIRT2} in the translationally symmetric case:~to continuously interpolate between $\Gamma_{s-}$ and $\Gamma_s$ using an auxiliary \textit{cascade} of aggregates evolving on a ``fast'' time scale. We interpret each $w(x)$ in \eqref{PDE:wave} as the first arrival time, on the fast time scale, of the cascade at $x\in\rr^d$.

\medskip

A straightforward computation reveals that \eqref{PDE:wave} is of hyperbolic type and that, at any point $x$, the direction $\nabla w(x)$ is time-like, whereas all directions orthogonal to $\nabla w(x)$ are space-like.~In other words, the level sets of $w$ propagate towards higher values of $w$ according to a quasi-linear wave equation encapsulated by \eqref{PDE:wave}. Such equations are routinely encountered in the theory of general relativity (at least, when $d=4$), and their local solvability for smooth initial data is well-known (see, e.g., \cite[Chapter 16, Proposition 3.2]{TaylorIII}). In contrast, the construction $\Gamma_s=\{x\!:w(x)<\infty\}$ necessitates a \textit{global} solution $w$ of \eqref{PDE:wave}. Moreover, for solutions of \eqref{SGT}, higher regularity of $\Gamma_s$ beyond being a set of finite perimeter in the sense of the geometric measure theory (see, e.g., \cite{EvGa}) may not hold (think, e.g., of two growing solid balls at the time they merge). 

In some sense, it is more instructive to recast \eqref{PDE:wave} as 
\begin{equation}\label{PDE accel}
	-\frac{\nabla w(x)^\top\,\mathrm{Hess}\, w(x)\,\nabla w(x)}{|\nabla w(x)|^5}
	+ \frac{1+u(s-,x)}{|\nabla w(x)|} = -\mathrm{div}\bigg(\frac{\nabla w(x)}{|\nabla w(x)|}\bigg)\,\bigg(\frac{1}{|\nabla w(x)|}+\gamma\bigg)\,\frac{1}{|\nabla w(x)|}
\end{equation}
and to contrast it with $\frac{1}{|\nabla w(x)|}=-\mathrm{div}\big(\frac{\nabla w(x)}{|\nabla w(x)|}\big)$, the arrival time PDE of the mean curvature flow (see \cite[equation (2.4)]{EvSp} and take $u(t,x)=t-w(x)$ therein).~When $u(s-,\cdot)\equiv-1$, the PDE~\eqref{PDE accel} states that the \textit{acceleration} $-\frac{\nabla w(x)^\top\,\mathrm{Hess}\, w(x)\,\nabla w(x)}{|\nabla w(x)|^5}$ (rather than the speed $\frac{1}{|\nabla w(x)|}$ in the mean curvature flow) is proportional to the negative mean curvature $-\mathrm{div}\big(\frac{\nabla w(x)}{|\nabla w(x)|}\big)$, resulting in a hyperbolic (rather than an elliptic) problem. The factors $\frac{1}{|\nabla w(x)|}+\gamma$ and $\frac{1}{|\nabla w(x)|}$ have the physical meanings of interfacial energy and speed, respectively (see Subsection \ref{subse:arrival} for more details).

\medskip

Simple examples (see Propositions \ref{prop:Example1Minimal}, \ref{prop:Example2Minimal}, \ref{prop:Example3Minimal}, \ref{prop:Example4Minimal}, and Remarks \ref{rem:min.lim.equil.2}(b), \ref{rem:min.lim.equil.4}(c)) reveal that:~(i) the arrival time function $w$ may take infinite values (if a point is never reached by the cascade), and (ii) even in the domain $\{x:\,w(x)<\infty\}$, the PDE \eqref{PDE:wave} can only be solved with ``$=$'' replaced by ``$\leq$", or, equivalently, we must relax \eqref{PDE:wave} by allowing for an additional non-positive measure-valued summand in its right-hand side.~As elaborated in Remarks \ref{rem:min.lim.equil.2}(b), \ref{rem:min.lim.equil.4}(c), the appearance of such a summand corresponds to the ``excessive loss of boundary energy" that ends up being ``locked inside" the solid domain. Naturally, we would like to select a solution to the relaxed version of \eqref{PDE:wave} that ``loses as little extra energy as possible for as long as possible", meaning that a boundary point keeps moving as long as its velocity is non-zero and there is vacant space for it. This concept is made precise in Definition~\ref{def:min}, leading to the notion of a \textit{minimal solution} to \eqref{PDE:wave}.

\medskip

In order to define and construct solutions to the relaxed version of \eqref{PDE:wave}, it turns out to be mathematically convenient to obtain a ``particle-level" representation for such solutions.~In the spirit of the control representations for curvature flows in \cite{BuCaQu}, \cite{SoTo}, \cite{KoSe}, \cite{LaRu}, we represent solutions to (the relaxed version of) \eqref{PDE:wave} via equilibria in an associated family of mean-field games. Such equilibria are given by triplets $(\kappa,{\mathbb Q},L)$, where $\kappa\in(0,\infty)$ is a normalizing constant (turning finite positive measures into probability measures), ${\mathbb Q}$ is a Borel probability measure on the space $\mathcal{Y}\times\RR_+=\{(y,\beta)\}$, in which we equip 
\begin{equation*}
	{\mathcal Y}:=\big\{y\!:[0,\infty)\to\rr^d:\;\|\dot{y}\|_\infty\le 1,\,\theta:=\inf\{t\ge0\!:y_t\in\overline{\Gamma}_{s-}\}<\infty,\,y_t=y_{\theta}\;\text{for}\; t\ge\theta\big\}
\end{equation*}
with the metric $d_{\mathcal Y}(y,\hat{y}):=\|y-\hat{y}\|_\infty+|\theta-\hat{\theta}|$, and $L\!:\rr^d\to\rr\backslash\{0\}$ is upper semicontinuous.

\begin{definition}\label{def:equi}
	We call $(\kappa,\mathbb{Q},L)$ an \textit{equilibrium capped at $T\in(0,\infty)$}, in short $(\kappa,\mathbb{Q},L)\in\mathcal{E}_T$, if 
	\begin{align}
		& \mathbb{Q}\in\text{argmin}_{\hat{\mathbb Q}:\;\hat{\mathbb Q}\circ y_0^{-1}=\mathbb Q\circ y_0^{-1}}\; \E^{\hat{\mathbb Q}}\bigg[\int_0^\theta L^+(y_t)\,\mathrm{d}t \bigg],\nonumber\\
		&\mathbb{Q}\bigg(\int_0^\theta L^+(y_t)\,\mathrm{d}t \le T,\,\beta\leq\theta\bigg)=1,
		\quad \mathbb{Q}\big(y_\theta\in\partial_*\Gamma_{s-}\,\big\vert\,\beta=\theta\big)=1, \label{equi}\\
		& \forall\,\phi\in C_c^\infty(\rr^d)\!: \nonumber \\
		& \kappa\,\E^{\mathbb Q}\bigg[\int_0^\beta \phi(y_t)\, L^+(y_t)\,\mathrm{d}t\bigg]
		=\int_{\rr^d} \phi(x)\,(1+\gamma L(x))\,\mathrm{d}x,\quad
		\E^{\mathbb Q}\bigg[\int_0^\beta \phi(y_t)\,L^-(y_t)\,\mathrm{d}t\bigg]=0, \nonumber
	\end{align}
	where $\partial_*\Gamma_{s-}$ is the measure theoretic boundary of $\Gamma_{s-}\subset\RR^d$, as defined, e.g., in \cite[Section~5.8]{EvGa}. 
\end{definition}

For the set $\mathcal{E}_T$ to be well-defined and non-trivial, we need to assume that $\Gamma_{s-}$ is a set of locally finite perimeter (a.k.a.~a Caccioppoli set) and that its perimeter is non-zero. These are standing assumptions throughout the paper.

\smallskip

The described setup constitutes a mean-field game equilibrium in which the particles, whose initial states are distributed according to $\mathbb{Q}\circ y_0^{-1}$, minimize the cost $\int_0^\theta L^+(y_t)\,\mathrm{d}t$ over $y\in\mathcal{Y}$ and are coupled via the ``fixed-point'' constraint $\kappa\,\E^{\mathbb Q}\big[\int_0^\beta \phi(y_t)\, L^+(y_t)\,\mathrm{d}t\big] =\int_{\rr^d} \phi(x)\,(1+\gamma L(x))\,\mathrm{d}x$, $\phi\in C_c^\infty(\rr^d)$, where $\beta$ is the ``killing time" of a particle. The latter constraint can be interpreted as follows: The running cost in the objective of each particle is given by a function of the cumulative occupation density of all alive particles' trajectories.~We refer to Subsection \ref{subse:equi} for more details.

\medskip

To represent the cascade and the arrival time function $w$ via an equilibrium $(\kappa,{\mathbb Q},L)$, let us write $w^{L}$ for the value function of a representative particle:
\begin{align}\label{eq:Eikonal.OC.1}
	w^L(x):=\inf_{y\in\mathcal{Y}:\,y_0=x} \int_0^\theta L^+(y_t)\,\mathrm{d}t,\quad x\in\RR^d.
\end{align}
%It follows trivially from the definition that $w^L=0$ in $\overline{\Gamma}_{s-}$.~The next section shows that, after redefining $w^L$ to take the infinite value outside of the support of $\overline{Q}$, we indeed obtain the desired arrival time function $w^{\QQ,L}$.
The HJB equation of the above control problem is $|\nabla w^{L}|\!=\! L^+$, and the optimal trajectory can be written (heuristically) as $\dot y_t = -\nabla w^{L}(y_t)/|\nabla w^{L}(y_t)|$. Consequently, along any equilibrium trajectory $y$, we have $\mathrm{d}w^{L}(y_t)=-L^+(y_t)\,\mathrm{d}t$, which allows us to regard the integrals $\int_0^\cdot L^+(y_t)\,\mathrm{d}t$ as the ``physical time" clock on which the cascade evolves along a particle trajectory $y$. We then define the time-$t$ cascade $D^{\QQ,L}_t$ as the set of initial points from which a particle following an equilibrium trajectory reaches $\Gamma_{s-}$ within $t$ units of physical time:
\begin{align*}
	& D^{\mathbb Q,L}_t
	:=\Gamma_{s-}\cup\mathrm{supp}\,\overline\QQ_t,\quad\overline\QQ_t(\mathrm{d}x)
	:=\,\E^{{\mathbb Q}}\bigg[\int_{\iota(L,t)}^{\beta\vee\iota(L,t)} \mathbf{1}_{\{y_r\in\, \mathrm{d}x\}}\,\mathrm{d}r\bigg],\\
	&\iota(L,t):=\inf\bigg\{r\in[0,\theta]\!:\,\int_r^\theta L^+(y_z)\,\mathrm{d}z<t\bigg\}.
\end{align*}
Lemma \ref{lem:w=v} shows that, in the region visited by the particles, the value function $w^{L}$ coincides with the arrival time function $w^{\QQ,L}$ of the cascade $(D^{\QQ,L}_t)_{t\geq0}$:
\begin{align*}
	&w^{L}(x)=w^{{\mathbb Q},L}(x):=\inf\{t\ge0\!:\,x\,\in D^{{\mathbb Q},L}_t\}.
\end{align*}

\smallskip

The above definition of the equilibrium arrival time function $w^{{\mathbb Q},L}$ ensures that $w^{\QQ,L}\vert_{\partial\Gamma_{s-}}=0$. It turns out (see Proposition \ref{prop:equi} for more details) that the second boundary condition, $\partial_{\mathbf{n}_{s-}} w^{\QQ,L}\,\vert_{\partial\Gamma_{s-}}=1/V_0$, is equivalent to
\begin{align}
	&\kappa\,\mathbb{Q}(y_\theta\in \mathrm{d}x,\beta=\theta)=\big(\gamma+V_0(x)\big)\,\bone_{\partial_*\Gamma_{s-}}(x)\,\mathcal{H}^{d-1}(\mathrm{d}x), \label{eq.intro.defEquil.boundaryCond}
\end{align}
assuming $\partial\Gamma_{s-}=\partial_*\Gamma_{s-}$ (otherwise, the boundary condition for $\partial_{\mathbf{n}_{s-}} w^{\QQ,L}$ is relaxed to hold only at $\partial_*\Gamma_{s-}$), where $\mathcal{H}^{d-1}$ is the $(d-1)$-dimensional Hausdorff measure. In addition, to take into account $u(s-,\cdot)$, we require $w^{\QQ,L}$ to satisfy \eqref{PDE:wave} with the inequality ``$\leq$'', where the inequality is due to the fact that a cascade may exhibit an excess loss of boundary energy, as mentioned above and illustrated by the examples in Subsections \ref{ex:1D2} and \ref{ex:Radial2}. (In the absence of such an excess loss of energy, the inequality turns into an equality -- see Proposition \ref{prop:equi}.) Corollary \ref{cor:fromEquil.toPDE} shows that, for any $t\geq0$, the PDE 
\begin{equation*}
	\mathrm{div}\bigg(\frac{\nabla w^{\QQ,L}(x)}{|\nabla w^{\QQ,L}(x)|}\,\bigg(\frac{1}{|\nabla w^{\QQ,L}(x)|}+\gamma\bigg)\!\bigg)
	=-\kappa\big(\QQ(y_0\in \mathrm{d}x) - \QQ(y_\beta\in \mathrm{d}x)\big)
\end{equation*}
holds in any open subset of $D^{\QQ,L}_t\setminus\Gamma_{s-}=\{x:\,0<w^{\QQ,L}(x)\leq t\}$.
Comparing the above to \eqref{PDE:wave}, we realize that the excess loss of energy is measured by
\begin{equation}\label{eq.intro.def.e}
	\begin{split}
		\mathbf{e}^{\kappa,\QQ,L}(\mathrm{d}x) := 
		& \;\kappa\,\mathbb{Q}(y_0\in \mathrm{d}x,\theta=0) \\ 
		&+\bone_{\{w^{\QQ,L}<\infty\}\setminus\overline{\Gamma}_{s-}}(x)\,\big(\kappa\,\QQ(y_0\in\,\mathrm{d}x) 
		-\kappa\,\QQ(y_\beta\in\,\mathrm{d}x) 
		-(1+u(s-,x))\,\mathrm{d}x \big), 
	\end{split}
\end{equation}
and that any \textit{admissible equilibrium} must have
\begin{align}
	&\mathbf{e}^{\kappa,\QQ,L}\geq0, \label{eq.intro.defEquil.u}
\end{align}
i.e., $\mathbf{e}^{\kappa,\QQ,L}$ must be a non-negative measure.

\medskip

Thus, finding a minimal solution to the relaxed version of \eqref{PDE:wave}, with the boundary conditions $w\vert_{\partial\Gamma_{s-}}=0$ and $\partial_{\mathbf{n}_{s-}} w\,\vert_{\partial\Gamma_{s-}}=1/V_0$, amounts to selecting a minimal element in the set of admissible equilibria $(\kappa,\QQ,L)$ satisfying \eqref{eq.intro.defEquil.boundaryCond}, parameterized by $\mathbb{Q}\circ y_0^{-1}$ and $\mathbb{Q}(\beta|y)$.~To make this selection (or, planning) problem precise it remains to define the notion of minimality.~Recalling the above discussion of minimality for solutions to the relaxed version of \eqref{PDE:wave}, we conclude that an equilibrium is minimal if its cumulative excess energy loss function 
\begin{align*}
	t\mapsto\mathbf{E}^{\kappa,\QQ,L}_t:=\mathbf{e}^{\kappa,\QQ,L}(\{w^{\QQ,L}\leq t\})
\end{align*}
cannot be bounded from below by the corresponding function of another equilibrium, on any interval $t\in[0,T]$ (see Subsection \ref{subse:minSol} for more details).

\smallskip

\begin{remark}
	The structure of \eqref{equi}, \eqref{eq.intro.defEquil.boundaryCond} resembles that of the Lions planning problem (see \cite[Section 1]{RTTY} for a summary). However, in our problem the planner chooses the distribution $\mathbb{Q}\circ y_0^{-1}$ of the players' initial states, as well as the conditional distribution $\mathbb{Q}(\beta|y)$ of the killing time~$\beta$ (note that the objective in the first line of \eqref{equi} depends only on the $y$-marginal of $\hat\QQ$, hence, the agents only control the latter), rather than the players' terminal reward function.~In contrast to the Lions planning problem, the planner's choice is highly non-unique, leading to the additional selection problem for the planner.
\end{remark}

\smallskip

Simple examples (see, e.g., Remarks \ref{rem:min.lim.equil.1}(b) and \ref{rem:min.lim.equil.2}(a)) show that, in order to capture natural solutions~$w$ to the cascade equation, one needs to consider a closure of the set of equilibrium arrival time functions $\{w^{\QQ,L}\}$, with the approximating sequences being asymptotically minimal and having strictly larger initial velocity.
All in all, we arrive at the following definition of a minimal solution (see Subsection \ref{subse:minSol} for a more detailed interpretation of this definition).

\begin{definition} \label{def:min}
	We say that $w:\,\RR^d\rightarrow[0,\infty]$ is a \textit{minimal solution} to \eqref{PDE:wave}, with initial domain $\Gamma_{s-}$ and initial velocity~$V_0$, if there exist equilibria $(\kappa^n,\mathbb{Q}^n,L^n)_{n\in\NN}$ such that, for any $T>0$:
	\begin{enumerate}[(i)]
		\item $\kappa^n\,\mathbb{Q}^n(y_\theta\in \mathrm{d}x,\beta=\theta)>\big(\gamma+V_0(x)\big)\,\bone_{\partial_*\Gamma_{s-}}(x)\,\mathcal{H}^{d-1}(\mathrm{d}x)$, $\mathbf{e}^{\kappa^n,\QQ^n,L^n}\geq0$, $w^{\mathbb{Q}^n,L^n}\wedge T\stackrel{L^1_{\mathrm{loc}}}{\longrightarrow} w\wedge T$ and 
		\begin{align*}
			&\kappa^n\,\QQ^n(\beta=\theta)\rightarrow \int_{\partial_*\Gamma_{s-}}\big(\gamma+V_0(x)\big)\,\mathcal{H}^{d-1}(\mathrm{d}x),
		\end{align*} 
		\item and for any sequence $(\hat{\kappa}^n,\hat{\QQ}^n,\hat{L}^n)_{n\in\mathbb N}$ satisfying $\hat{\kappa}^n\,\hat{\QQ}^n(y_\theta\in \mathrm{d}x,\beta=\theta)=\kappa^n\,\QQ(y_\theta\in \mathrm{d}x,\beta=\theta)$, $\mathbf{e}^{\hat{\kappa}^n,\hat{\QQ}^n,\hat{L}^n}\geq0$, and $\mathbf{E}^{\hat{\kappa}^n,\hat{\QQ}^n,\hat{L}^n}_t\leq\mathbf{E}^{\kappa^n,\QQ^n,L^n}_t$ for all $t\in[0,T]$, we have:
		\begin{align*}
			&\liminf_{n\rightarrow\infty}\Big(\mathbf{E}^{\hat{\kappa}^n,\hat{\QQ}^n,\hat{L}^n}_t-\mathbf{E}^{\kappa^n,\QQ^n,L^n}_r\Big)\geq0\quad \text{for any}\,\,0\leq r<t<T.
		\end{align*}
	\end{enumerate}
\end{definition}

We are now ready to state our first main result. 

\begin{theorem}\label{thm:exist}
	Suppose that $d=2$, the set $\Gamma_{s-}$ has finite perimeter, $\int_{\partial_*\Gamma_{s-}} (\gamma+V_0)\,\mathrm{d}\mathcal{H}^{1} < \infty$, and  $\int_{\RR^d\setminus\overline{\Gamma}_{s-}} (1+u(s-,x))^-\,\mathrm{d}x<\infty$. Then, there exists a minimal solution to \eqref{PDE:wave}.
	%{\color{red}The above assumption $|u(s-,\cdot)|\leq C$ is not efficient/natural, which is a consequences of the inefficiency of the current equilibrium formulation with general $u$. Ideally, it needs to be replaced by an integrability assumption.}
\end{theorem}

\begin{remark}
	Our proof of Theorem \ref{thm:exist} goes by a compactness argument in $\mathrm{BV}_{\mathrm{loc}}(\rr^d)$, the space  of functions of locally bounded variation equipped with the topology of $L^1_{\mathrm{loc}}(\rr^d)$-convergence. For this purpose, it suffices to show a uniform bound on the $\mathrm{BV}$-norm (see \cite[Subsection~5.2.3]{EvGa}), which in turn reduces to a \textit{perimeter estimate} on the cascade aggregates $D^{{\mathbb Q},L}_t$, at Lebesgue-a.e. time $t\geq0$. To prove such a perimeter estimate, in Lemma \ref{le:2dim.key} and Corollary \ref{cor:2dim.key}, we use a characterization of the essential boundary of a \textit{planar} finite-perimeter set in terms of rectifiable Jordan curves (\cite[Corollary~1]{ACMM}). The latter restricts Theorem \ref{thm:exist} to $d=2$, while all other arguments used to establish the desired perimeter estimates, and to prove Theorem \ref{thm:exist}, are valid for any $d\geq1$. 
\end{remark}

In addition to playing a crucial role in the proof of Theorem \ref{thm:exist}, a perimeter estimate on the cascade aggregates~$D^{{\mathbb Q},L}_t$ is of great interest in its own right, as it implies a perimeter estimate on the ``post-discontinuity'' aggregate $\Gamma_s=\{x:\,w(x)<\infty\}$ in the Stefan problem with surface tension \eqref{SGT}. The latter constitutes a major step in controlling the perimeter of the aggregate $\Gamma_s$ in \eqref{SGT} \textit{for all times $s\geq0$}.~While expected, such a result for \eqref{SGT} would stand in stark contrast to the findings for the Stefan problem without surface tension, in which an instantaneous explosion of the perimeter is possible (see \cite[Figure 1]{NSZ} for an illustration).~Thus, the perimeter estimate asserted in the next theorem is our second main result. 

\begin{theorem} \label{thm:peri}
	Under the assumptions of Theorem \ref{thm:exist}, for any minimal solution $w$ to \eqref{PDE:wave},
	\begin{equation}\label{eq: peri main}
		\big\|\partial\{w<\infty\}\big\|(\RR^2)
		\le \frac{1}{\gamma} \int_{\partial_*\Gamma_{s-}} (\gamma+V_0)\,\mathrm{d}\mathcal{H}^1 - \frac{1}{\gamma}\int_{\{w<\infty\}\setminus\overline{\Gamma}_{s-}} (1+u(s-,x))\,\mathrm{d}x,
	\end{equation}
	where $\big\|\partial\{w<\infty\}\big\|(\RR^2)$ is the perimeter of the set $\{w<\infty\}$, i.e., the total variation of $\bone_{\{w<\infty\}}$ (cf., e.g.,  \cite[Section 5.1]{EvGa}).
	%{\color{red}The current equilibrium formulation with general $u$ does not yield the above result. Instead, it seems to produce a multiplicative term in the right hand side that is exponential in $C$.}
\end{theorem}

\begin{remark}
	\!\! To gain intuition for the perimeter estimate~\eqref{eq: peri main}, suppose a continuous $w\!\!:\!\rr^2\!\to\![0,\infty]$ solves the relaxed version of \eqref{PDE:wave},
	\begin{equation}\label{PDE:wave.relaxed}
		\mathrm{div}\bigg(\frac{\nabla w(x)}{|\nabla w(x)|}\,\bigg(\frac{1}{|\nabla w(x)|}+\gamma\bigg)\!\bigg)\leq -1-u(s-,x),
	\end{equation}
	\textit{classically} in $\{0<w<\infty\}$, and is such that $w(x)=0\,\Leftrightarrow\,x\in\Gamma_{s-}$ and that
	\begin{align*}
		\underset{\varepsilon\downarrow0}{\liminf}\, \int_{\partial_*\{\varepsilon<w\}} \bigg(\gamma+\frac{1}{|\nabla w|}\bigg)\,\mathrm{d}\mathcal{H}^1 \le \int_{\partial_*\{w=0\}} (\gamma+V_0)\,\mathrm{d}\mathcal{H}^1.
	\end{align*} 
	%{\color{purple}Assume, for simplicity that $\kappa=\int_{\partial\Gamma_{s-}} (\gamma+V_0)\,\mathrm{d}\mathcal{H}^1=1$.}
	Then, using \eqref{PDE:wave.relaxed} and the divergence theorem in the form of \cite[Section 5.8, Theorem 1]{EvGa}:
	\begin{equation*}
		\begin{split}
			&-\int_{\{\varepsilon<w<T\}} (1+u(s-,x))\,\mathrm{d}x
			\geq \int_{\{\varepsilon<w<T\}} \mathrm{div}\bigg(\frac{\nabla w(x)}{|\nabla w(x)|^2}+\gamma\frac{\nabla w(x)}{|\nabla w(x)|}\bigg)\,\mathrm{d}x \\
			& = \int_{\partial_*\{w<T\}} \left(\frac{1}{|\nabla w|}+\gamma\right)\,\mathrm{d}\mathcal{H}^1
			- \int_{\partial_*\{\varepsilon<w\}} \left(\frac{1}{|\nabla w|}+\gamma\right)\,\mathrm{d}\mathcal{H}^1 \\
			&\ge \gamma\,\|\partial\{w<T\}\|(\RR^2)
			- \int_{\partial_*\{\varepsilon<w\}} 
			\left(\gamma+\frac{1}{|\nabla w|}\right)\,\mathrm{d}\mathcal{H}^1,
		\end{split}
	\end{equation*}
	for all $0<\varepsilon<T<\infty$, where we assume that $1+u(s-,\cdot)$ is absolutely integrable on $\{0<w<\infty\}$. Taking $T\uparrow\infty$ via \cite[Subsection 5.2.1, Theorem 1]{EvGa} and then $\varepsilon\downarrow0$ gives
	\begin{equation}
		\|\partial\{w<\infty\}\|(\RR^2) \le 
		\frac{1}{\gamma}\int_{\partial_*\{w=0\}} \left(\gamma+V_0\right)\,\mathrm{d}\mathcal{H}^1 
		- \frac{1}{\gamma}\int_{\{0<w<\infty\}} (1+u(s-,x))\,\mathrm{d}x,
	\end{equation} 
	which is precisely \eqref{eq: peri main}.~The actual proof of \eqref{eq: peri main} for a \textit{minimal} solution $w$ is much more involved, because the various regularity assumptions made in this example are not satisfied a priori.~In particular, the arrival time function $w^{\QQ,L}$ approximating $w$ may not be continuous, which means that the boundary of $\{\varepsilon<w^{\QQ,L}<T\}$ may not be contained in $\{w^{\QQ,L}=\varepsilon\}\cup\{w^{\QQ,L}=T\}$.
	Remarkably, our proof of the perimeter estimate for $w^{\QQ,L}$, given in Subsection \ref{subse:perim.equil}, shows that the discontinuities of $w^{\QQ,L}$ may only occur when $w^{\QQ,L}$ ``jumps to infinity". We conjecture that a minimal solution $w$ cannot jump to infinity, which would imply that it is continuous, but we are not yet able to prove it.
\end{remark}

The rest of the paper is structured as follows.~In Section \ref{se:form}, we derive the cascade PDE \eqref{PDE:wave} (Subsection \ref{subse:arrival}) and relate its solutions to the arrival time functions of a  family of mean-field games (Subsection \ref{subse:equi}). In Subsection \ref{subse:arrivalRig}, we establish various useful properties of the aforementioned mean-field equilibria and derive a rigorous connection between their arrival time functions and the solutions to the relaxed version of \eqref{PDE:wave}, in Corollary \ref{cor:fromEquil.toPDE}.~Then, in Subsection \ref{subse:perim.equil}, we estimate the perimeter of equilibrium aggregates.~In Subsection \ref{subse:minSol}, we discuss the concept of a minimal solution and prove its existence (i.e., prove Theorem \ref{thm:exist}).~Section \ref{subse:perim.minSol} is devoted to the proof of Theorem \ref{thm:peri}. In the concluding Section \ref{se:example}, we consider examples that illustrate the proposed theory and justify the choices we made in developing it.

\medskip

\noindent\textbf{Acknowledgement.} The authors are grateful to Alpar Meszaros and Nizar Touzi for enlightening discussions during the preparation of this paper. 

%!TEX root =Main.tex

\section{Derivation of the cascade equation} \label{se:form}

%%%%%%%%%%%%%%%%%%%%
\subsection{PDE formulation} \label{subse:arrival}
%%%%%%%%%%%%%%%%%%%%

To motivate the arrival time formulation of the cascade equation, let us start with the translationally invariant setting: $\Gamma_{0-}=(-\infty,0]\times\RR^{d-1}$, $u(0-,x)=\overline{u}(0-,x_1)$. One is then naturally led to look for solutions of the Stefan problem in the form $\Gamma_s=(-\infty,\Lambda_s]\times\RR^{d-1}$, $u(s,x)=\overline{u}(s,x_1)$. In such solutions, ``physical'' discontinuities are characterized by the condition 
\begin{equation}\label{eq:1D casc}
	\Lambda_s-\Lambda_{s-} = \inf\bigg\{z>0:\;-\int_{\Lambda_{s-}}^{\Lambda_{s-}+z} \overline{u}(s-,x_1)\,\mathrm{d}x_1<z\bigg\}
\end{equation}
(cf.~\cite[equation (1.6)]{dns}). One can also imagine the solid traversing the strip $[\Lambda_{s-},\Lambda_s]\times\RR^{d-1}$ continuously on an auxiliary ``fast'' time scale $t\ge0$. The jump size in \eqref{eq:1D casc} can then be modeled equivalently through the ordinary differential equation (ODE) problem
\begin{equation} \label{eq:1D ODE} 
	\begin{cases}
		& x_1'(t)=-\int_{\Lambda_{s-}}^{x_1(t)} (1+\overline{u}(s-,z)) \,\mathrm{d}z,\quad x_1(0)=\Lambda_{s-}, \\
		& \Lambda_s:=\lim_{t\to\infty} x_1(t).
	\end{cases}
\end{equation}
The arrival time formulation of the problem \eqref{eq:1D ODE} goes as follows: Define the arrival time function
\begin{equation}
	w(z)=\inf\{t\ge0:\,x_1(t)\ge z\}, 
\end{equation}
with the convention $\inf\,\emptyset:=\infty$. Then, differentiating once with respect to $t$ in \eqref{eq:1D ODE}, and using the change of variable formulas $x_1'(t)\, w'(x_1(t))=1$, $x_1''(t)\, w'(x_1(t)) + x_1'(t)^2\, w''(x_1(t)) = 0$ and $w(\Lambda_{s-})=0$, we end up with 
\begin{eqnarray}  
	&& w''(z)=(1+\overline{u}(s-,z))\, w'(z)^2,\quad w(\Lambda_{s-})=0, \label{eq:1D w PDE} \\
	&& \Lambda_s:=\sup\{z>\Lambda_{s-}:\, w(z)<\infty\}. 
\end{eqnarray}

\smallskip

In the absence of translational invariance, the naive analogue of \eqref{eq:1D w PDE} reads
\begin{equation}
	\partial_{\mathbf{n}(x)\mathbf{n}(x)} w(x) = (1+u(s-,x))\,(\partial_{\mathbf{n}(x)} w(x))^2, \quad w|_{\partial\Gamma_{s-}}\equiv0, \label{eq: w PDE very naive}
\end{equation}
where $\mathbf{n}(x)$ is the outward unit normal to the sublevel set $\{\widetilde{x}\in\RR^d\!: w(\widetilde{x})\!\le\! w(x)\}$. The equation~\eqref{eq: w PDE very naive} simply restates the ODE in \eqref{eq:1D w PDE} at every $x\in\{\widetilde{x}\in\RR^d:\, w(\widetilde{x})<\infty\}$ in the direction of $\mathbf{n}(x)$. Since
\begin{eqnarray}
	&& \partial_{\mathbf{n}(x)} w(x) = \nabla w(x)\cdot \frac{\nabla w(x)}{|\nabla w(x)|} = |\nabla w(x)|, \\
	&& \partial_{\mathbf{n}(x)\mathbf{n}(x)} w(x) = \nabla |\nabla w(x)|\cdot \frac{\nabla w(x)}{|\nabla w(x)|}
	=\frac{\nabla w(x)^\top\,\mathrm{Hess}\, w(x)\,\nabla w(x)}{|\nabla w(x)|^2},
\end{eqnarray}
the problem \eqref{eq: w PDE very naive} amounts to 
\begin{equation}\label{eq: w PDE very naive'}
	\frac{\nabla w(x)^\top\,\mathrm{Hess}\, w(x)\,\nabla w(x)}{|\nabla w(x)|^4}
	=  1+u(s-,x),\quad w|_{\partial\Gamma_{s-}}\equiv0.
\end{equation}
However, \eqref{eq: w PDE very naive'} (and \eqref{eq: w PDE very naive}) ignores the change in the surface area of the sublevel sets $\{\widetilde{x}\in\RR^d\!: w(\widetilde{x})\le t\}$ as $t$ increases, which leads to an energy imbalance in \eqref{eq: w PDE very naive'} as larger surface areas require more energy to move.~The local change in the surface area is given by $H(x)\,\frac{1}{|\nabla w(x)|}$, where $H(x)$ is the mean curvature of $\partial\{\widetilde{x}\in\RR^d\!: w(\widetilde{x})\!\le\! w(x)\}$ at $x$, and $\frac{1}{|\nabla w(x)|}$ is the speed at which $x$ moves (see, e.g., \cite[equation (3)]{Gar}). Thus, we amend \eqref{eq: w PDE very naive'} to
\begin{equation}\label{eq: w PDE naive}
	-\frac{\Delta w(x)}{|\nabla w(x)|^2}
	+2\,\frac{\nabla w(x)^\top\,\mathrm{Hess}\, w(x)\,\nabla w(x)}{|\nabla w(x)|^4}
	=  1+u(s-,x),\quad w|_{\partial\Gamma_{s-}}\equiv0,
\end{equation}
where we have plugged in $H(x)=\mathrm{div}\Big(\frac{\nabla w(x)}{|\nabla w(x)|}\Big)$ and expanded.

\medskip

It is instructive to rewrite \eqref{eq: w PDE naive} as 
\begin{equation}\label{eq: w PDE naive'}
	-\mathrm{div}\bigg(\frac{\nabla w(x)}{|\nabla w(x)|^2}\bigg) = 1+u(s-,x),
	\quad w|_{\partial\Gamma_{s-}}\equiv0.
\end{equation}
In particular, this renders the divergence theorem applicable:
\begin{equation*}
	\begin{split}
		& \;\int_{\{0<w<t\}} (1+u(s-,x))\,\mathrm{d}x 
		= \int_{\{0<w<t\}} -\mathrm{div}\bigg(\frac{\nabla w(x)}{|\nabla w(x)|^2}\bigg)\,\mathrm{d}x \\
		& = \int_{\partial\{0<w<t\}\cap\{w=0\}} \frac{\nabla w(x)}{|\nabla w(x)|^2}
		\cdot\frac{\nabla w(x)}{|\nabla w(x)|}\,\mathcal{H}^{d-1}(\mathrm{d}x) 
		- \int_{\partial\{0<w<t\}\cap\{w=t\}} \frac{\nabla w(x)}{|\nabla w(x)|^2}
		\cdot\frac{\nabla w(x)}{|\nabla w(x)|}\,\mathcal{H}^{d-1}(\mathrm{d}x)  \\
		& =  \int_{\partial\{0<w<t\}\cap\{w=0\}} \frac{1}{|\nabla w(x)|}\,\,\mathcal{H}^{d-1}(\mathrm{d}x) 
		- \int_{\partial\{0<w<t\}\cap\{w=t\}} \frac{1}{|\nabla w(x)|}\,\mathcal{H}^{d-1}(\mathrm{d}x).
	\end{split}
\end{equation*}
The latter energy balance equation assumes zero surface tension, i.e., $\gamma=0$. For $\gamma>0$, one accounts for the surface tension by following \cite[Example 5, displays (2.42), (2.44)]{RoSa} and including the increment of the surface energy into the energy balance equation:
\begin{equation}\label{eq.cascadePDE.byParts}
	\begin{split}
		& \;\int_{\{0<w<t\}} (1\!+\! u(s-,x))\,\mathrm{d}x + \gamma\mathcal{H}^{d-1}\big(\partial\{0\!<\! w \!<\! t\}\!\cap\!\{w \!=\! t\}\big) 
		- \gamma\mathcal{H}^{d-1}\big(\partial\{0\!<\! w \!<\! t\}\!\cap\!\{w \!=\!0\}\big) \\
		& = \int_{\partial\{0<w<t\}\cap\{w=0\}} \frac{1}{|\nabla w(x)|}\,\,\mathcal{H}^{d-1}(\mathrm{d}x) 
		- \int_{\partial\{0<w<t\}\cap\{w=t\}} \frac{1}{|\nabla w(x)|}\,\mathcal{H}^{d-1}(\mathrm{d}x).
	\end{split}
\end{equation}
The corresponding amendment of the problem \eqref{eq: w PDE naive'} is then
\begin{equation*}
	\mathrm{div}\bigg(\frac{\nabla w(x)}{|\nabla w(x)|^2}+\gamma\frac{\nabla w(x)}{|\nabla w(x)|}\bigg) = -1-u(s-,x),
	\quad w|_{\partial\Gamma_{s-}}\equiv0,
\end{equation*}
which yields precisely the PDE \eqref{PDE:wave}.
%\begin{equation}
%-\Delta w(x)\bigg(\frac{1}{|\nabla w(x)|^2}+\frac{\gamma}{|\nabla w(x)|}\bigg)
%+\nabla w(x)^\top\,\mathrm{Hess}\, w(x)\,\nabla w(x)
%\bigg(\frac{2}{|\nabla w(x)|^4}+\frac{\gamma}{|\nabla w(x)|^3}\bigg)
%=  1-u(s-,x).
%\end{equation}

%%%%%%%%%%%%%%%%%%%%%%%
\subsection{Equilibrium formulation} \label{subse:equi}
%%%%%%%%%%%%%%%%%%%%%%%

In the first part of this subsection, we show that a sufficiently regular weak solution to the PDE \eqref{PDE:wave} (understood in its integrated form \eqref{eq: Green I}) is an arrival time function of an equilibrium. Note that the regularity of a solution to \eqref{PDE:wave} includes the requirement $|\nabla w|\neq0$, which excludes points of local maxima. Hence, the existence of such a regular solution can only be assumed to hold locally. The desired result is described in the following proposition and in its proof.
%We begin with the derivation of the ``\textit{mean-field game equilibrium} capped at $T\in(0,\infty)$'' given in \eqref{equi}.

\begin{proposition}[From PDE to equilibrium]\label{prop:equi}
Consider an arbitrary $T>0$, and measurable functions $V_0,w:\,\RR^d\rightarrow\RR_+$ and $u(s-,\cdot):\,\RR^d\to\RR$, with the properties that
\begin{itemize}
\item[(i)] $w\in C^1(\{0<w< T\})\cap C(\overline{\{0<w< T\}})$,  $\{w=0\}=\overline{\Gamma}_{s-}$, $\int_{\{0<w<T\}} |u(s-,x)|\,\mathrm{d}x<\infty$, and the functions $V_0$, $\bone_{\{0<w<T\}}/|\nabla w|$ are bounded.
\item[(ii)] For any $t\in(0,T)$, the sets $\Gamma_{s-}$ and $\{w<t\}$ have finite perimeters, with $\partial\Gamma_{s-}=\partial_*\Gamma_{s-}$ and $\partial\{w<t\}=\partial_*\{w<t\}$. Moreover, $\mathrm{Leb}(\partial\{w<T\})=0$.
\item[(iii)] $w$ is a weak solution to \eqref{PDE:wave}, \eqref{eq.intro.defEquil.boundaryCond} in $\{0<w<T\}$, in the sense that
\begin{equation}\label{eq: Green I}
\begin{split}	
& \;\int_{\{0<w<t\}} -\phi(x)\,(u(s-,x)+1) + \nabla\phi(x)\cdot\bigg(\frac{\nabla w(x)}{|\nabla w(x)|^2}+\gamma\frac{\nabla w(x)}{|\nabla w(x)|}\bigg)\,\mathrm{d}x \\
& = \int_{\partial\{w<t\}} \phi(x)\,\bigg(\frac{1}{|\nabla w|}+\gamma\bigg) \,\mathcal{H}^{d-1}(\mathrm{d}x) 
- \int_{\partial\Gamma_{s-}} \phi(x)\,\big(V_0(x)+\gamma\big) \,\mathcal{H}^{d-1}(\mathrm{d}x)
\end{split}
\end{equation}
for any compactly supported $\phi\in\mathrm{Lip}(\overline{\{0<w<t\}})$ and any $t\in(0,T)$.
\item[(iv)] For any $z\in\partial\Gamma_{s-}\cup\{0<w<T\}$, there exists a unique solution $Y^z$ to the ODE 
\begin{equation}\label{eq: ODE}
\mathrm{d}Y^z_r = \frac{\nabla w}{|\nabla w|}(Y^z_r)\,\mathrm{d}r,\quad \text{a.e.}\quad r\in[0,\tau(T)),
\end{equation}
with the initial condition $Y^z_0=z$ and with $\tau(t):=\inf\{r\geq0:\,w(Y^z_r)\geq t\}$. Moreover, there exists a continuous map $[0,T)\times\partial\Gamma_{s-}\ni(t,z)\mapsto \mathbf{Y}(t,z)\in\overline{\{0<w<T\}}$ with a (locally) Lipschitz-continuous inverse on $\partial\Gamma_{s-}\cup\{0<w<T\}$, denoted by $x\mapsto(\mathbf{t}(x),\mathbf{z}(x))$, such that $Y^z_{\tau(t)}=\mathbf{Y}(t,z)$.
\end{itemize}

Then, there exists $(\kappa,\QQ,L)\in\cE_T$ with the properties that
\begin{itemize}
\item[(1)] $w=w^{\QQ,L}$ in $\overline{\{w<T\}}$.
\item[(2)] $\mathbf{e}^{\kappa,\QQ,L}\geq0$ and $\mathbf{E}^{\kappa,\QQ,L}_t=\mathbf{e}^{\kappa,\QQ,L}(\{w^{\QQ,L}\leq t\})=0$ for any $t\in(0,T)$.
\item[(3)] $\kappa\QQ(y_\theta\in\,\mathrm{d}x,\beta=\theta)=(V_0(x)+\gamma)\,\bone_{\partial\Gamma_{s-}}(x)\,\cH^{d-1}(\mathrm{d}x)$.
\end{itemize}
\end{proposition}
\begin{remark}
Assuming (i) holds, a sufficient condition for (iv) is as follows: $\nabla w$ is Lipschitz-continuous in $\overline{\{0<w<t\}}$ for any $t\in(0,T)$ and $\Gamma_{s-}=\{x\in\RR^d:\,\Phi(x)\leq 0\}$ with a function $\Phi\in C^1(\RR^d)$ such that $|\nabla\Phi|>0$ on $\partial\Gamma_{s-}$.
\end{remark}
\begin{proof}
Let us construct $(\kappa,\QQ,L)\in\cE_T$ satisfying the desired properties.
First, we denote
\begin{align*}
\kappa:=\int_{\{0<w<T\}} (1+u(s-,x))^-\,\mathrm{d}x + \int_{\partial\Gamma_{s-}} (V_0+\gamma) \,\mathrm{d}\mathcal{H}^{d-1}>0
\end{align*}
%and denote the left hand side of the above by $\kappa>0$ (note we can make $\kappa$ arbitrarily large by representing $1+u=f-g$ with any appropriate $f,g\geq0$ -- the derivations that follow would not change).
%and note that we aim to choose $\QQ$ so that $\QQ(\beta<\theta) = \kappa^{-1}\int_{\{0<w<t\}} (-(1+u(s-,x)))^+\,\mathrm{d}x$.
%Next, we 
and define a probability measure 
\begin{align*}
\alpha(\mathrm{d}x):= \kappa^{-1}\,\bone_{\partial\Gamma_{s-}}(x)\,\big(V_0(x)+\gamma\big)\,\mathcal{H}^{d-1}(\mathrm{d}x) + \kappa^{-1}\,\bone_{\{0< w<T\}}(x)\,(1+u(s-,x))^-\,\mathrm{d}x.
\end{align*}
Next, we denote by $Y$ the solution to \eqref{eq: ODE} with a random initial value $\xi$, constructed on the time interval $[0,\tau]$, where $0\leq\tau\leq\tau(T)$ is a random time.
We choose the joint distribution of $(\xi,\tau)$ so that $\xi\sim\alpha$ and the distribution of $Y_\tau$ restricted to $\{0\leq w<T\}$ equals
\begin{align*}
& \lambda(\mathrm{d}x):=\kappa^{-1}\,\bone_{\{0<w<T\}}(x)\,(1+u(s-,x))^+\,\mathrm{d}x.
\end{align*}

\smallskip

Let us show that such a coupling of $(\xi,\tau)$ exists under the assumptions of the proposition.
We consider an arbitrary $\psi\in C^1_c(\RR^d)$, define $\phi(x)=\psi(\mathbf{z}(x))$, and notice that $\phi(Y^z_r)=\psi(z)$ remains constant over $r$, for any $z\in\partial\Gamma_{s-}$, which implies $\nabla\phi\cdot\nabla w=0$ in $\{0<w<t\}$. Then, for such $\phi$, the equation \eqref{eq: Green I} becomes
\begin{equation}\label{eq: Green I.projected}
\begin{split}	
& \int_{\{0<w<t\}} \psi(\mathbf{z}(x))\,(1+u(s-,x))^+\,\mathrm{d}x = \int_{\{0<w<t\}} \psi(\mathbf{z}(x))\,(1+u(s-,x))^-\,\mathrm{d}x\\
& + \int_{\partial\Gamma_{s-}} \psi(x)\,\big(V_0(x)+\gamma\big) \,\mathcal{H}^{d-1}(\mathrm{d}x)
-\int_{\partial\{w<t\}} \psi(\mathbf{z}(x))\,\bigg(\frac{1}{|\nabla w|}+\gamma\bigg) \,\mathcal{H}^{d-1}(\mathrm{d}x).
\end{split}
\end{equation}
Choosing $\psi\uparrow1$ and $t\uparrow T$, we deduce from \eqref{eq: Green I.projected} that $\lambda(\RR^d)\leq 1$.
Next, we treat $(\mathbf{t},\mathbf{z})$ as a random vector defined on $\RR^d$ and disintegrate its distribution under the measures $\alpha$ and $\lambda$:
\begin{align*}
& \alpha(\mathbf{t}\in \mathrm{d}t,\,\mathbf{z}\in \mathrm{d}z) = \alpha_1(z,\mathrm{d}t)\,\alpha_2(\mathrm{d}z),\quad \lambda(\mathbf{t}\in \mathrm{d}t,\,\mathbf{z}\in \mathrm{d}z) = \lambda_1(z,\mathrm{d}t)\,\lambda_2(\mathrm{d}z),
\end{align*}
where $\alpha_1$, $\lambda_1$ are probability kernels, $\alpha_2$ is a probability measure, and $\lambda_2$ is a sub-probability measure.
Equation \eqref{eq: Green I.projected} yields
\begin{align}
& \EE^\alpha \left[\psi(\mathbf{z})\,\bone_{[0,t)}(\mathbf{t})\right] \geq  \EE^\lambda\left[\psi(\mathbf{z})\,\bone_{[0,t)}(\mathbf{t})\right],\quad 0\leq\psi\in C^1_c(\RR^d),\label{eq.alpha.lambda.exp}
\end{align}
which implies $\lambda_2(\mathrm{d}z) = l(z)\,\alpha_2(\mathrm{d}z)$ with $0\leq l\leq 1$.
Next, we consider a probability measure $\bar\lambda_1(z,\mathrm{d}t)\,\alpha_2(\mathrm{d}z)$ on $[0,T]\times\partial\Gamma_{s-}$, with the property that $\bar\lambda_1(z,\cdot)=l(z)\,\lambda_1(z,\cdot)$ on $[0,T)$, constructed as follows:
\begin{align*}
& \bar\lambda_1(z,\mathrm{d}t) = l(z)\,\bone_{[0,T)}(t)\,\lambda_1(z,\mathrm{d}t) + (1-l(z))\,\delta_T(\mathrm{d}t).
\end{align*}
Then, \eqref{eq.alpha.lambda.exp} implies $\alpha_1(z,[0,t])\geq\bar\lambda_1(z,[0,t])$ for all $t\in[0,T)$. Since $\alpha_1(z,\{T\})=0$, we extend the latter property of stochastic dominance to all $t\in[0,T]$, which in turn implies the existence of random elements $\sigma_0$, $\sigma_1$, $Z$ such that $(\sigma_0,Z)\sim \alpha_1(z,\mathrm{d}t)\,\alpha_2(\mathrm{d}z)$, $(\sigma_1,Z)\sim \bar\lambda_1(z,\mathrm{d}t)\,\alpha_2(\mathrm{d}z)$, $\sigma_0\leq \sigma_1$. Since the law of $(\sigma_0,Z)$ coincides with the law of $(\mathbf{t},\mathbf{z})$ under $\alpha$, we conclude that $\mathbf{Y}(\sigma_0,Z)\sim\alpha$, and hence we set $\xi:=\mathbf{Y}(\sigma_0,Z)$. Similarly, since the restriction of the law of $(\sigma_1,Z)$ to $[0,T)\times\partial\Gamma_{s-}$ coincides with the law of $(\mathbf{t},\mathbf{z})$ under $\lambda$, we conclude that the distribution of $\mathbf{Y}(\sigma_1,Z)$ restricted to $\{0\leq w<T\}$ equals $\lambda$, and hence we set $\tau:=\inf\{r\geq0:\,w(Y_r)\geq \sigma_1\}-\inf\{r\geq0:\,w(Y_r)\geq \sigma_0\}\geq0$. It is easy to see that the joint distribution of $(\xi,\tau)$ satisfies the desired properties.

%indeed, we can construct a time-reversed version of $Y$, by starting the process from the desired distribution at $\partial_*\Gamma_{s-}$ and moving it along the gradient ascent curve of $w$ until it hits $\{w\geq t\}$.

\medskip

Finally, we define $L(x)=|\nabla w(x)|\,\bone_{\overline{\{0< w< T\}}}(x) - (1/\gamma)\,\bone_{\RR^d\setminus\overline{\{0< w< T\}}}(x)$ and $\QQ=\mathrm{Law}(Y_{\tau-\cdot},\tau)$, where we extend $|\nabla w|$ to the boundary of $\{0< w< T\}$ via its upper semicontinuous envelope.
Note that $L^+=|\nabla w|$ is the HJB equation expected to be satisfied by the value function of the control problem $\inf_{\|\dot{y}\|_\infty\le 1} \int_0^\theta L^+(y_t)\,\mathrm{d}t$ (see, e.g., \cite[Section I.8]{FlSo}). Then, the solution of \eqref{eq: ODE} gives an optimal trajectory for this problem, after time-reversal, as can be shown directly by verifying the equality $w(Y_{\tau})=\int_0^\tau L^+(Y_{\tau-r})\,\mathrm{d}r$ and the inequality $w(y_0)\leq\int_0^\theta L^+(y_{r})\,\mathrm{d}r$ for any $y\in\mathcal{Y}$. This yields the optimality stated in the first line of \eqref{equi}. The second line of \eqref{equi} holds by construction.

\medskip

It is easy to see that $w^{\QQ,L}=w\,\bone_{\overline{\{w< T\}}} + \infty\,\bone_{\RR^d\backslash\overline{\{w< T\}}}$.
In addition,
\begin{align*}
& \bone_{\{0<w^{\QQ,L}<\infty\}}(x)\,\QQ(y_0\in \mathrm{d}x) = \bone_{\{0<w\leq T\}}(x)\,\QQ(y_0\in \mathrm{d}x) \\
&\phantom{????????????????????????}\,= \kappa^{-1}\,\bone_{\{0<w<T\}}(x)\,(1+u(s-,x))^+\,\mathrm{d}x
+ \bone_{\{w= T\}}(x)\,\QQ(y_0\in \mathrm{d}x),\\
& \bone_{\{0<w^{\QQ,L}<\infty\}}(x)\,\QQ(y_\beta\in \mathrm{d}x) = \bone_{\{0<w\leq T\}}(x)\,\QQ(y_\beta\in \mathrm{d}x)\\
&\phantom{????????????????????????}\;= \bone_{\{0<w\leq T\}}(x)\,\alpha(\mathrm{d}x) = \kappa^{-1}\,(1+u(s-,x))^-\, \bone_{\{0<w\leq T\}}(x)\,\mathrm{d}x,\\
&\QQ(y_\theta\in \mathrm{d}x\,\vert\,\beta=\theta) = \bigg( \int_{\partial\Gamma_{s-}} (V_0+\gamma)\,\mathrm{d}\mathcal{H}^{d-1}\bigg)^{-1}
\,\big(V_0(x)+\gamma\big)\,\bone_{\partial\Gamma_{s-}}(x) \,\mathcal{H}^{d-1}(\mathrm{d}x),\\
&\QQ(\beta=\theta) = \kappa^{-1}\,\int_{\partial\Gamma_{s-}} (V_0+\gamma)\,\mathrm{d}\mathcal{H}^{d-1}.
\end{align*}
The last two lines of the above yield the property (3) stated in the proposition, while the first two lines imply the property (2). %\eqref{eq.intro.defEquil.u}. Note also that the latter equation is satisfied with equality everywhere except $\{w=T\}$ (i.e., the constructed equilibrium $(\kappa,\QQ,L)$ does not have any excess loss of energy strictly before the physical time $T$), hence proving the desired property $\mathbf{E}^{\kappa,\QQ,L}_t=\mathbf{e}^{\kappa,\QQ,L}(\{w^{\QQ,L}\leq t\})=0$ for any $t\in(0,T)$.

\medskip

To verify the last line of \eqref{equi}, for any $\phi\in C^\infty_c(\RR^d)$ supported in $\{0<w<T\}$, we deduce
\begin{equation}\label{eq: FK}
\begin{split}
&-\kappa^{-1}\int_{\{0<w<T\}} \phi(x)(u(s-,x)+1)\,\mathrm{d}x=\E^\QQ[\phi(y_\beta)-\phi(y_0)]\\
&=\E^\QQ\bigg[\int_0^\beta - \nabla\phi(y_r)\cdot\frac{\nabla w}{|\nabla w|}(y_r)\,\mathrm{d}r\bigg] 
= -\int_{\{0<w<T\}} \nabla\phi(x)\cdot\frac{\nabla w(x)}{|\nabla w(x)|}\, \overline{\mathbb Q}(\mathrm{d}x),
\end{split}
\end{equation}
where we introduce the occupation measure of $Y$:
\begin{equation}\label{eq.barQ.def.1}
\overline{\mathbb Q}(\mathrm{d}x):=\E\bigg[\int_0^\beta \mathbf{1}_{\{Y_r\in\,\mathrm{d}x\}}\,\mathrm{d}r\bigg].
\end{equation}
Matching \eqref{eq: FK} and \eqref{eq: Green I} we conclude: 
\begin{equation}
\frac{L^+(x)}{1+\gamma L^+(x)}\,\overline{\mathbb Q}(\mathrm{d}x) = \frac{\mathrm{d}x}{\kappa}
\quad\Longleftrightarrow\quad 
\kappa L^+(x)\,\overline{\mathbb Q}(\mathrm{d}x) = (1+\gamma L^+(x))\,\mathrm{d}x, 
\end{equation}
which gives the last line in \eqref{equi}.
\end{proof}
	
\begin{remark}
Notice that the above proof relies on several simplifying assumptions.~Among them are the assumptions that $w$ solves the cascade PDE \eqref{PDE:wave}, as opposed to its relaxed version with ``$=$'' replaced by ``$\leq$'', and that $|\nabla w|>0$ in $\{0<w<T\}$, which typically does not hold for all $T>0$, even if $w$ is smooth.~In fact, the latter assumption is directly related to the existence of a solution $Y$ to the ODE $\dot Y_t = \nabla w(Y_t)/|\nabla w(Y_t)|$ that is well-defined up until $Y$ exits $\{w\leq T\}$, which was used to show that the equilibrium constructed above satisfies \eqref{eq.intro.defEquil.u} in $\{w<T\}$ with ``$=$'' rather than ``$\ge$''.~As illustrated by the examples of Subsections \ref{ex:1D2} and \ref{ex:Radial2}, the arrival time functions generically violate these assumptions (e.g., these properties fail to hold at any strict local maximum of $w$), which leads to the existence of gradient ascent curves that cannot be extended to arbitrary level sets, resulting in an excess loss of boundary energy that manifests itself in turning ``$=$'' in \eqref{PDE:wave} into ``$\leq$''.
\end{remark}
	
	%To avoid infinite values of the objective (which is possible, as $w$ may take infinite values), we cap the total cost at $T$ in the first line of \eqref{equi}.
	
	The above proposition shows that any sufficiently regular solution to the cascade PDE \eqref{PDE:wave} can be represented as the arrival time function of an equilibrium.~Since the latter functions can be defined without strong regularity assumptions, we choose to define solutions to \eqref{PDE:wave} via equilibria. 
	
	\medskip
	
Let us comment on the specific assumptions made in the definition of the set of admissible equilibria (recall Definition \ref{def:equi}), which are not explained by the proof of Proposition \ref{prop:equi}.  
First, we note that the moving interface starts at $\Gamma_{s-}$ and moves into the liquid domain, whereas the particles in the optimization problem given by the first line of \eqref{equi} evolve along the time-reversed trajectories -- from inside the liquid domain and towards $\Gamma_{s-}$ -- minimizing the total cost of travel, measured by $\int_0^\theta L^+(y_t)\,\mathrm{d}t$. Such trajectories are expected to be orthogonal to $\partial\Gamma_{s-}$, which is consistent with the intuition that $\nabla w$ should be orthogonal to $\partial\Gamma_{s-}$ at the points of the latter. In order to ensure that the notion of orthogonal direction is well-defined at the points where the trajectories reach $\partial\Gamma_{s-}$, we require that $y_\theta\in\partial_*\Gamma_{s-}$ in the second line of \eqref{equi}.
Second, in order to make the connection between the HJB equation $|\nabla w|=L^+$ and the associated control problem rigorous (i.e., to ensure that the dynamic programming principle holds), we assume that $L$ is upper semicontinuous.

	\begin{remark}\label{rem:nonempty}
		We note that, for any $T\in(0,\infty)$, $\kappa\geq0$ and $\nu\in \mathcal{P}(\partial_*\Gamma_{s-})$, the tuple $(\kappa,\QQ^0,L^0)$ is always in $\mathcal{E}_T$, where $\QQ^0$ is the image of $\nu$ under the map that takes any $x\in\RR^d$ to the constant path $y_t\equiv x$, and $L^0\equiv-\frac1\gamma$. Hence, $\mathcal{E}_T$ is always non-empty.
	\end{remark}
	
	\begin{remark}
		An equilibrium capped at $T$ corresponds to a cascade that is stopped at the physical time $T$. We let $\mathcal{E}:=\bigcup_{T>0}\mathcal{E}_T$.
	\end{remark}
	
	\begin{remark}
		Note that neither the boundary condition \eqref{eq.intro.defEquil.boundaryCond}, nor the condition \eqref{eq.intro.defEquil.u}, nor the minimality of an equilibrium are ensured by the definition of $\mathcal{E}$. These additional conditions lead to the selection (or, planning) problem analyzed in the next section.
	\end{remark}

\section{Properties of equilibria}
\label{se:fullyInt}

%%%%%%%%%%%%%%%%%%%%%%
\subsection{Arrival time function in the equilibrium formulation}
\label{subse:arrivalRig}
%%%%%%%%%%%%%%%%%%%%%%
In this subsection, we introduce the equilibrium arrival time function $w^{\QQ,L}$ and show rigorously its connection to the relaxed version of the cascade PDE.
We begin with two auxiliary results.

\begin{lemma}\label{lem:UsefulProperties}
Let $(\kappa,\QQ,L)\in\mathcal{E}_T$ and $w^L$ be the corresponding value function. Then the following statements hold true:
\begin{enumerate}
    \item For any Borel-measurable  $\phi$ that is either Lebesgue-integrable and compactly supported or non-negative,
    \begin{align*}
    \int_{\RR^d} \phi(x)\,(L(x)+1/\gamma)\,\mathrm{d}x = \frac{\kappa}{\gamma}\,\EE^{\QQ}\left[ \int_0^\beta \phi(y_t)\,L^+(y_t)\,\mathrm{d}t\right],
\quad \EE^{\QQ}\left[\int_0^\beta \phi(y_t)\,L^-(y_t)\,\mathrm{d}t\right]=0.
    \end{align*}
    \item Recall the occupation measure of $\QQ\circ y_{[0,\beta]}^{-1}$: $\overline\QQ(\mathrm{d}x)=\EE^\QQ\left[\int_0^\beta \bone_{\{y_t\in\,\mathrm{d}x\}}\,\mathrm{d}t\right]$. This measure is $\sigma$-finite, absolutely continuous w.r.t. the Lebesgue measure, and its density $q$ satisfies
    \begin{align*}
    & L(x) = \frac{\kappa}{\gamma}\, q(x)\, L^+(x) - \frac{1}{\gamma},
    \quad q(x)\,L^-(x)=0,\quad\text{a.e. }x\in\RR^d.
    \end{align*}
    In particular, $q(x)>\frac\gamma\kappa$ a.e. on the set $\{x:\,q(x)>0\}$, and $L$ is determined a.e. from $q$ by 
    \begin{align*}
    L(x)=\frac{1}{\kappa q(x)-\gamma}\bone_{\{q(x)>0\}}-\frac1\gamma\bone_{\{q(x)=0\}},    
    \end{align*}
    and conversely, $q$ is determined a.e. by $L$ via
    \begin{align*}
    q(x)=\left(\frac\gamma\kappa+\frac{1}{\kappa L}\right)\bone_{\{L(x)>0\}}.    
    \end{align*}
    \item $L>0$ in $\mathrm{supp}\,\overline\QQ$, and $L=-1/\gamma$ a.e. in $\RR^d\setminus\mathrm{supp}\,\overline\QQ$.
    \item For $\QQ$-a.e. $y\in\mathcal{Y}$, 
    \begin{align*}
    w^L(y_t)=\int_t^\theta L^+(y_r)\,\mathrm{d}r,\quad t\in[0,\theta].
    \end{align*}
    \item $w^L$ is locally Lipschitz continuous and, thus, is differentiable a.e. In addition, $|\nabla w^L|=L^+$ a.e. and $|\nabla w^L|\in L^1(\RR^d)$, with
    \begin{align*}
    \int_{\RR^d}|\nabla w^L(x)|\,\mathrm{d}x\leq\frac{\kappa}{\gamma}T.
    \end{align*}
    Finally, for $\QQ$-a.e. $(y,\beta)$ and a.e. $t\in[0,\beta]$, $t\mapsto y_t$ is differentiable with
    \begin{align*}
    |\nabla w^L(y_t)|=L^+(y_t)>0\quad\text{and}\quad\dot y_t=-\frac{\nabla w^L}{|\nabla w^L|}(y_t). 
    \end{align*}
\end{enumerate}
\end{lemma}
\begin{proof}
(1) follows from standard measure-theoretic results. (2) is deduced similarly, making use of the fixed-point conditions in the last line of equations \eqref{equi} and recalling that $L$ is upper semicontinuous and does not take value zero. (3) follows from (2) and the upper semicontinuity of~$L$.

\smallskip

To prove (4), we ease the notation by identifying $\QQ$ with its $y$-marginal. Let us first show that 
\begin{align*}
w^L(y_0)=\int_0^\theta L^+(y_r)\,\mathrm{d}r
\end{align*}
for $\QQ$-a.e. $y$. Apparently, $w^L(y_0)\leq\int_0^\theta L^+(y_r)\,\mathrm{d}r$ as $y\in\cY$ is a path connecting $y_0$ to $\overline{\Gamma}_{s-}$, and we will prove the opposite inequality by contradiction. Suppose 
\begin{align*}
\QQ\bigg(\bigg\{y:w^L(y_0)<\int_0^\theta L^+(y_r)\,\mathrm{d}r\bigg\}\bigg)>0.    
\end{align*}
Then, setting $\mu:=\QQ\circ y_0^{-1}$,
\begin{align*}
\varepsilon:=\EE^\QQ\left[\int_0^\theta L^+(y_r)\,\mathrm{d}r - w^L(y_0)\right]=\EE^\QQ\left[\int_0^\theta L^+(y_r)\,\mathrm{d}r\right]-\int_{\RR^d} w^L(x)\,\mu(\mathrm{d}x)>0.    
\end{align*}
We would like to apply \cite[Theorem 4.8]{rieder1978measurable} (a version of the measurable selection theorem) to $X:=\RR^d$, $Y:=\cY$,
\begin{align*}
&\Omega:=\{(x,y)\in\RR^d\times\cY:y_0=x\},\\
&u(x,y)=u(y):=\int_0^{\theta(y)}L^+(y_r)\,\mathrm{d}r\quad\text{for }\,(x,y)\in\RR^d\times\cY.
\end{align*}
It is obvious that $\Omega$ is a closed (thus trivially $F_\sigma$) subset of $X\times Y$. We claim that $y\mapsto\int_0^{\theta(y)}L^+(y_t)\,\mathrm{d}t$ is upper semicontinuous. Indeed, let $y_n\to y$ in $\cY$. Then $\infty>\theta(y)=\lim_{n\to\infty}\theta(y^n)$ and by the upper semicontinuity of $L^+$, $L^+(y_t)\,\bone_{[0,\theta(y)]}(t)\ge\limsup_{n\to\infty} L^+(y_t^n)\,\bone_{[0,\theta(y^n)]}(t)$ for all $t\neq\theta(y)$. Moreover, there exists a compact set $E\subset\RR^d$ containing the images of $y$ and all the $y^n$'s. As $L^+$ is upper bounded on $E$ thanks to the upper semicontinuity of $L^+$, we can now apply Fatou's Lemma to obtain
\begin{align*}
\int_0^{\theta(y)}L^+(y_t)\,\mathrm{d}t&=\int_0^\infty L^+(y_t)\,\bone_{[0,\theta(y)]}(t)\,\mathrm{d}t\\&\ge\int_0^\infty\limsup_{n\to\infty}L^+(y_t^n)\,\bone_{[0,\theta(y^n)]}(t)\,\mathrm{d}t\\
&\ge\limsup_{n\to\infty}\int_0^\infty L^+(y_t^n)\,\bone_{[0,\theta(y^n)]}(t)\,\mathrm{d}t=\limsup_{n\to\infty}\int_0^{\theta(y^n)}L^+(y_t^n)\,\mathrm{d}t,
\end{align*}
which justifies the claim. Therefore,
\begin{align*}
\{(x,y)\in \Omega:u(x,y)<c\}=\bigg(\RR^d\times\bigg\{y\in\cY:\int_0^{\theta(y)}L^+(y_t)\,\mathrm{d}t<c\bigg\}\bigg)\cap \Omega    
\end{align*}
is the intersection of an open (thus $F_\sigma$) set and a closed set, hence an $F_\sigma$ set. As a result, \cite[Theorem 4.8]{rieder1978measurable} (with a change of sign) implies the existence of an $\frac{\varepsilon}{2}$-minimizer, that is a measurable map $\varphi:\RR^d\to\cY$ such that $(x,\varphi(x))\in \Omega$ and that
\begin{align*}
u(x,\varphi(x))\leq \inf_{y\in\cY,y_0=x}\int_0^\theta L^+(y_t)\,\mathrm{d}t+\frac\varepsilon2=w^L(x)+\frac\varepsilon2.  
\end{align*}
It is now clear that $\QQ^\varepsilon:=\mu\circ\varphi^{-1}$ satisfies $\QQ^\varepsilon\circ y_0^{-1}=\mu$ and that
\begin{align*}
\EE^{\QQ^\varepsilon}\left[\int_0^\theta L^+(y_t)\,\mathrm{d}t\right]=\int u(x,\varphi(x))\,\mu(\mathrm{d}x)\leq\int w^L(x)\,\mu(\mathrm{d}x)+\frac{\varepsilon}{2}<\EE^\QQ\left[\int_0^\theta L^+(y_t)\,\mathrm{d}t\right],    
\end{align*}
which is a contradiction to the first line of equations \eqref{equi}. Therefore, it must hold that $w^L(y_0)=\int_0^\theta L^+(y_t)\,\mathrm{d}t$ for $\QQ$-a.e. $y$.
 
For a general $t\in[0,\theta]$, since $y_{[t,\theta]}$ is a path connecting $y_t$ to $\overline{\Gamma}_{s-}$, it is obvious that $w^L(y_t)\leq\int_t^\theta L^+(y_r)\,\mathrm{d}r$. On the other hand, as $y_{[0,t]}$ is a path connecting $y_0$ to $y_t$ we must have $w^L(y_0)\leq w^L(y_t)+\int_0^tL^+(y_r)\,\mathrm{d}r$, which together with $w^L(y_0)=\int_0^\theta L^+(y_r)\,\mathrm{d}r$ implies that $w^L(y_t)\ge\int_t^\theta L^+(y_r)\,\mathrm{d}r$.\\

\smallskip

To prove (5), we first take any $x_1\in\RR^d\setminus\overline{\Gamma}_{s-}$, $r>0$ such that $B(x_1,r)\subset\RR^d\setminus\overline{\Gamma}_{s-}$ and any $x_2\in B(x_1,r)\setminus\{x_1\}\subset\RR^d\setminus\overline{\Gamma}_{s-}$. As $L$ is a real-valued upper semicontinuous function, it must be upper bounded on any compact set. For any $y\in\cY$ such that $y_0=x_2$, we can construct a corresponding $\tilde y\in\cY$ by defining $\tilde y_t:=(x_1+t\frac{x_2-x_1}{|x_2-x_1|})\,\bone_{[0,|x_2-x_1|]}(t)+y_{t-|x_2-x_1|}\,\bone_{[|x_2-x_1|,\infty)}(t)$, which satisfies $\tilde y_0=x_1$ and $\theta(\tilde y)=|x_2-x_1|+\theta(y)$. This $\tilde y$ allows us to upper bound
\begin{align}\label{eq:Lip0}
w^L(x_1)\leq\int_0^{\theta(\tilde y)}L^+(\tilde y_r)\,\mathrm{d}r\leq|x_2-x_1|\cdot\sup_{\overline B(x_1,r)}L^++\int_0^{\theta(y)}L^+(y_r)\,\mathrm{d}r.    
\end{align}
As $y$ ranges over $\{y\in\cY:y_0=x_2\}$, we get
$w^L(x_1)\leq|x_2-x_1|\,\sup_{\overline B(x_1,r)}L^++w^L(x_2)$. Similarly, we can prove $w^L(x_2)\leq|x_2-x_1|\,\sup_{\overline B(x_1,r)}L^++w^L(x_1)$, and thus $|w^L(x_2)-w^L(x_1)|\;\leq\;|x_2-x_1|\,\\ \sup_{\overline B(x_1,r)}L^+$. This establishes the local Lipschitz continuity of $w^L$ in $\RR^d\setminus\overline{\Gamma}_{s-}$. We then take any $x_1\in\overline\Gamma_{s-}$ and any $x_2\in\RR^d$. If $x_2\in\overline\Gamma_{s-}$, then $w^L(x_1)=w^L(x_2)=0$, and thus $|w^L(x_2)-w^L(x_1)|\leq|x_2-x_1|$ trivially holds. If however $x_2\in\RR^d\setminus\overline{\Gamma}_{s-}$, then, using $w^L(x_1)=0$ and the shortest line segment joining $x_2$ and $\overline{\Gamma}_{s-}$, we obtain
\begin{align*}
|w^L(x_2)-w^L(x_1)|=w^L(x_2)\leq d(x_2,\overline{\Gamma}_{s-})\,\sup_{\overline{B}(x_2,d(x_2,\overline{\Gamma}_{s-}))}L^+\leq|x_2-x_1|\,\sup_{\overline{B}(x_2,d(x_2,\overline{\Gamma}_{s-}))}L^+.
\end{align*}
In summary, we have proved that $w^L$ is locally Lipschitz continuous in the whole space $\RR^d$. It then follows from Rademacher's Theorem that $w^L$ is a.e. differentiable. Moreover, given a direction $\alpha\in\QQ^d$, $|\alpha|=1$ and an $x\in\RR^d\backslash\overline{\Gamma}_{s-}$, an argument similar to the one that led to \eqref{eq:Lip0} yields
\begin{equation}
w^L(x+t\alpha) \le w^L(x) + \int_0^t L^+(x+z\alpha)\,\mathrm{d}z,
\end{equation}
as long as $x+t\alpha\in\RR^d\backslash\overline{\Gamma}_{s-}$, $t\in[0,r]$. By the Lebesgue Differentiation Theorem, it follows that $|\partial_\alpha w^L(x+t\alpha)|\le L^+(x+t\alpha)$ for Lebesgue-a.e. $t$ in the open set $\{t\in\RR\!:x+t\alpha\in\RR^d\backslash\overline{\Gamma}_{s-}\}$. In view of Fubini's Theorem, we may conclude that $|\partial_\alpha w^L|\le L^+$ Lebesgue a.e.~in $\RR^d\backslash\overline{\Gamma}_{s-}$. Taking the supremum over $\alpha\in\QQ^d$, we infer that $|\nabla w^L|\leq L^+$ Lebesgue a.e.~in $\RR^d\backslash\overline{\Gamma}_{s-}$. On the set $\overline{\Gamma}_{s-}$, $w^L=0$ and thus $|\nabla w^L|=0$ a.e. on $\overline{\Gamma}_{s-}$ by \cite[Theorem 3.3]{EvGa}. In summary, we have obtained the comparison $|\nabla w^L|\leq L^+$ a.e. in $\RR^d$.

Now, taking $\phi\equiv1$ in the fixed point equation (the last line of equations \eqref{equi}) gives
\begin{align*}
\int_{\RR^d}|\nabla w^L(x)|\,\mathrm{d}x\leq\int_{\RR^d}L^+(x)\,\mathrm{d}x&\leq\int_{\RR^d}(L(x)+1/\gamma)\,\mathrm{d}x
\\&=\frac\kappa\gamma\EE^\QQ\left[\int_0^\beta L^+(y_t)\,\mathrm{d}t\right]
\leq\frac\kappa\gamma\EE^\QQ\left[\int_0^\theta L^+(y_t)\,\mathrm{d}t\right]\leq\frac\kappa\gamma T. 
\end{align*}

To obtain the equality between $|\nabla w^L|$ and $L^+$, we first note that $\nabla w^L(y_t)$ exists and $L^+(y_t)>0$ for $\QQ$-a.e. $(y,\beta)$ and a.e. $t\in[0,\beta]$ as the occupation measure $\overline\QQ$ is absolutely continuous with respect to the Lebesgue measure. As a result, another application of Rademacher's Theorem and the Lebesgue Differentiation Theorem, combined with item (4), shows that 
\begin{align*}
&\dot y_t:=\frac{\mathrm{d}}{\mathrm{d}t}y_t\,\,\text{exists and }|\dot y_t|\leq 1,\\&-\nabla w^L(y_t)\cdot\dot y_t=\frac{\mathrm{d}}{\mathrm{d}t}(-w^L(y_t))=\frac{\mathrm{d}}{\mathrm{d}t}\left(-\int_t^\theta L^+(y_r)\,\mathrm{d}r\right)=L^+(y_t)\begin{cases}>0\\\ge|\nabla w^L(y_t)|\end{cases}    
\end{align*}
for $\QQ$-a.e. $(y,\beta)$ and a.e. $t\in[0,\theta]$. For such $(y,\beta)$ and $t$, the Cauchy-Schwarz inequality forces  
\begin{align*}
|\nabla w^L(y_t)|=L^+(y_t)>0\quad\text{and}\quad\dot y_t=-\frac{\nabla w^L}{|\nabla w^L|}(y_t).
\end{align*}
The above implies that $|\nabla w^L(x)|=L^+(x)$ for $\overline\QQ$-a.e. (and hence Lebesgue-a.e.) $x\in\{q>0\}$. On the set $\{q=0\}$, we know from part (2) that $L^+=0$ a.e., which, combined with the above estimate $|\nabla w^L|\leq L^+$ a.e. in $\RR^d$, forces that $|\nabla w^L|=L^+$ a.e. on $\{q=0\}$. This finishes the proof of the desired result that $|\nabla w^L|=L^+$ a.e. in $\RR^d$.
\end{proof}

The next lemma introduces the notion of a $T$-level projection, which transforms any equilibrium into its $T$-capped version, and shows that this projection preserves the equilibrium property.

\begin{lemma}\label{lem:EquilibriumTimeConsistency}
For any equilibrium $(\kappa,\QQ,L)\in\mathcal{E}_{T'}$ and any $T\in(0,T')$, we have $(\kappa,\QQ_T,L_T)\in \mathcal{E}_T$, where the latter is the $T$-level projection defined by:
\begin{align*}
& \QQ_T:= \QQ\circ (y_{[\iota(L,T),\infty)},(\beta-\iota(L,T))^+)^{-1},
\quad L_T := L\,\bone_{\{w^L\leq T\}} - \frac{1}{\gamma}\,\bone_{\{w^L>T\}},\\
& \iota(L,T):=\inf\bigg\{r\in[0,\theta]:\, \int_{r}^\theta L^+(y_t)\,\mathrm{d}t<T\bigg\}.
\end{align*}
Moreover, the value functions are related by
\begin{align*}
w^{L_T}=w^L\wedge T.    
\end{align*}
\end{lemma}
\begin{proof}
We first prove the identity relating the two value functions, which is a basic fact about the optimal control problem \eqref{eq:Eikonal.OC.1}. Indeed, it follows immediately from the definition of $L_T$ that $L_T^+\leq L^+$, and hence $w^{L_T}\leq w^L$. If $x\in\RR^d$ is such that $w^L(x)>T$, then we can connect $x$ with the set $\{w^L=T\}$ via a line segment with zero $L_T^+$-cost. This shows $w^{L_T}\leq T$ and hence $w^{L_T}\leq w^L\wedge T$. We then take any $x\in\RR^d\setminus\overline{\Gamma}_{s-}$ and any $y\in\mathcal{Y}$ that connects $x$ to $\overline{\Gamma}_{s-}$, i.e., $y_0=x$ and $\theta=\inf\{t\geq0:\,y_t\in\overline{\Gamma}_{s-}\}<\infty$. If $y_{[0,\theta]}\subset\{w^L\leq T\}$ then $\int_0^\theta L_T^+(y_t)\,\mathrm{d}t=\int_0^\theta L^+(y_t)\,\mathrm{d}t\ge w^L(x)$. Otherwise $\xi:=\sup\{t\in[0,\theta]:w^L(y_t)>T\}$ and necessarily $\int_0^\theta L_T^+(y_t)\,\mathrm{d}t\ge\int_\xi^\theta L_T^+(y_t)\,\mathrm{d}t=\int_\xi^\theta L^+(y_t)\,\mathrm{d}t\ge w^L(y_\xi)=T$. This proves $w^{L_T}\ge w^L\wedge T$.

Next, we prove that $(\kappa,\QQ_T,L_T)\in \mathcal{E}_T$. Clearly, $\QQ_T\circ y_0^{-1}=\QQ\circ y^{-1}_{\iota(L,T)}=:\mu_T$.
Thanks to the definition of $\iota(L,T)$, it holds that $y_{[\iota(L,T),\theta]}\subset\{w^L\leq T\}$ for $\QQ$-a.e. $y$, that is $y_{[0,\theta]}\subset\{w^L\leq T\}$ for $\QQ_T$-a.e. $y$. It is also clear that
\begin{align}\label{eq:CappedbyT1}
&\QQ_T\bigg(\int_0^\theta L_T^+(y_t)\,\mathrm{d}t\leq T\bigg)
=\QQ\bigg(\int_{\iota(L,T)}^\theta L^+(y_t)\,\mathrm{d}t\leq T\bigg)=1.
\end{align}
For the optimality condition $\QQ_T\in\mathrm{argmin}_{\hat\QQ\circ y_0^{-1}=\QQ_T\circ y_0^{-1}} \EE^{\hat\QQ}\left[\int_{0}^\theta L_T^+(y_t)\,\mathrm{d}t\right]$ (the first line of \eqref{equi}), it suffices to notice that
\begin{align*}
\int_0^\theta L_T^+(y_t)\,\mathrm{d}t=\int_0^\theta L^+(y_t)\,\mathrm{d}t=w^L(y_0)=w^{L_T}(y_0)
\end{align*}
for $\QQ_T$-a.e. $y$ as a result of Lemma \ref{lem:UsefulProperties} and equation \eqref{eq:CappedbyT1}.

To verify the fixed point equation (last line of \eqref{equi}), for any $\phi\in C_c^\infty(\RR^d)$, we obtain:
\begin{align*}
&\int_{\RR^d} \phi(x)\,(L_T(x)+1/\gamma)\,\mathrm{d}x=\int_{\RR^d} \phi(x)\,\bone_{\{w^L\leq T\}}(x)\,(L(x)+1/\gamma)\,\mathrm{d}x\\
&=\frac{\kappa}{\gamma}\,\EE^{\QQ}\left[ \int_0^\beta \phi(y_t)\,\bone_{\{w^L\leq T\}}(y_t)\,L^+(y_t)\,\mathrm{d}t\right]\\
&=\frac{\kappa}{\gamma}\,\EE^{\QQ}\left[\int_{\iota(L,T)}^{\beta\vee\iota(L,T)} \phi(y_t)\,L_T^+(y_t)\,\mathrm{d}t\right]
=\frac{\kappa}{\gamma}\,\EE^{\QQ_T}\left[ \int_0^\beta \phi(y_t)\,L_T^+(y_t)\,\mathrm{d}t\right],
\end{align*}
and
\begin{align*}
&\EE^{\QQ_T}\left[\int_0^\beta \phi(y_t)\,L_T^-(y_t)\,\mathrm{d}t\right]=\EE^{\QQ}\left[\int_{\iota(L,T)}^{\beta\vee\iota(L,T)} \phi(y_t)\,L_T^-(y_t)\,\mathrm{d}t\right]\\
&=\EE^{\QQ} \left[\int_{\iota(L,T)}^{\beta\vee\iota(L,T)} \phi(y_t)\,L^-(y_t)\,\mathrm{d}t\right]
=\EE^{\QQ}\left[\int_0^\beta \phi(y_t)\,\bone_{\{w^L\leq T\}}(y_t)\,L^-(y_t)\,\mathrm{d}t\right]=0.   
\end{align*}
\end{proof}

\begin{remark}
It is easy to see that, for any equilibrium $(\kappa,\QQ,L)\in\mathcal{E}$ and any $0<t<T$, the $t$-level projection of $(\kappa,\QQ_T,L_T)$ coincides with $(\kappa,\QQ_{t},L_t)$.
It is also clear that, if $(\kappa,\QQ,L)\in\mathcal{E}_T$, then $(\QQ,L)=(\QQ_T,L_T)$.
\end{remark}

\medskip

Notice that the $T$-projection corresponds to stopping the moving boundary at the physical time $T$.
Continuing this line of thought, we notice that any equilibrium $(\kappa,\QQ,L)\in\mathcal{E}$ generates a non-decreasing set-valued mapping $t\mapsto D^{\QQ,L}_t$, for $t>0$, where $D^{\QQ,L}_t\subset\RR^d$ is the union of $D^{\QQ,L}_0:=\Gamma_{s-}$ and the support of $\overline{\QQ}_t$, where the latter is the occupation measure of $\QQ_t$:
\begin{align*}
&\overline{\QQ}_t(\mathrm{d}x):=\EE^{\QQ_t}\left[\int_0^\beta \bone_{\{y_z\in\, \mathrm{d}x\}}\,\mathrm{d}z\right].
\end{align*}
We interpret $D^{\QQ,L}_t$ as the cascade aggregate at time $t$.

Now, we can formally introduce the arrival time function $w^{\QQ,L}:\RR^d\rightarrow[0,\infty]$ in the equilibrium formulation:
\begin{align*}
& w^{\QQ,L}(x):= \inf\{t\geq0:\,x\in D^{\QQ,L}_t\}.
\end{align*}
The value of $w^{\QQ,L}(x)$ measures the physical time it takes the cascade to reach $x$.

\smallskip

\begin{remark}
As the set-valued map $t\mapsto D_t^{\QQ,L}$ is non-decreasing (in the sense that $s<t$ implies $D_s^{\QQ,L}\subset D_t^{\QQ,L}$), it is not hard to see that
\begin{align}
\{x:\,w^{\QQ,L}\leq t\} = \bigcap_{\varepsilon>0}D_{t+\varepsilon}^{\QQ,L}=:D_{t+}^{\QQ,L}.\label{eq.sublevSetW.Gamma}
\end{align}
%We refer to $\Gamma^{\QQ}_{s+}$ as the aggregate generated by the associated equilibrium at time $s$.
\end{remark}

\smallskip

\begin{remark}
For any $(\kappa,\QQ,L)\in\mathcal{E}_T$, we have $D^{\QQ,L}_T = \Gamma_{s-}\cup\{L>0\}$.
\end{remark}

\medskip

The next lemma shows that $w^{\QQ,L}$ is obtained by setting $w^L$ to infinity outside of the terminal cascade aggregate.

\begin{lemma}\label{lem:w=v}
Let $(\kappa,\QQ,L)\in\mathcal{E}_T$ and $D_t:=D_t^{\QQ,L}$. Then
\begin{align*}
w^{\QQ,L}(x)=w^L(x)\,\bone_{D_T}(x)+\infty\,\bone_{\RR^d\setminus D_T}(x),\quad x\in\RR^d.    
\end{align*}
\end{lemma}
\begin{proof}
The desired equality clearly holds for $x\in\Gamma_{s-}$.
Let $x\in D_T\setminus\Gamma_{s-}$ and consider $t:=w^{\QQ,L}(x)\leq T$. It holds that $\overline{\QQ}_{t+\varepsilon}(B(x,r))>0$ for any $r>0$ and any $\varepsilon>0$. In particular, Lemma \ref{lem:UsefulProperties} implies the existence of a sequence $(x_n)_{n\ge0}$ converging to $x$ with the property that $w^L(x_n)\leq t+\varepsilon$, which shows that $w^L(x)\leq t$ by the continuity of $w^L$ and by letting $\varepsilon\downarrow0$. On the other hand, if $t>0$ then for any $\varepsilon\in(0,t)$, $\overline\QQ_{t-\varepsilon}(B(x,r))=0$ for all sufficiently small $r>0$. Lemma \ref{lem:UsefulProperties} implies that for $\QQ$-a.e. $(y,\beta)$ and a.e. $z\in[0,\beta]$ such that $y_z\in B(x,r)$ it must be that $w^L(y_z)>t-\varepsilon$. This, together with the first half of the proof, implies the existence of a sequence $(x_n)_{n\ge0}$ converging to $x$ with the property that $w^L(x_n)>t-\varepsilon$, which shows that $w^L(x)\ge t$ by the continuity of $w^L$ and by letting $\varepsilon\downarrow0$.

As $D_t=D_T$ for all $t>T$, we see from the definition that $w^{\QQ,L}(x)=\infty$ if $x\in\RR^d\setminus D_T$.
\end{proof}

\smallskip

\begin{corollary}
For any $(\kappa,\QQ,L)\in\mathcal{E}$, we have $w^{\QQ,L}=w^L\,\bone_{\{L>0\}\cup\Gamma_{s-}} + \infty\,\bone_{\{L=-1/\gamma\}\setminus\Gamma_{s-}}$.
\end{corollary}

\smallskip

The next lemma shows how the arrival time function is affected by a projection.

\begin{lemma}\label{le:projectionOfw}
Let $(\kappa,\QQ,L)\in\mathcal{E}$ and $D_t:=D_t^{\QQ,L}$. For any $T\in(0,\infty)$, we have:
\begin{align*}
w^{\QQ,L}\wedge T=w^{\QQ_T,L_T}\wedge T,
\quad w^{\QQ_T,L_T}= w^{L_T}\,\bone_{D_T} + \infty\,\bone_{\RR^d\setminus D_T}.
\end{align*}
\end{lemma}
\begin{proof}
Thanks to the tower property of projections, it holds that $D_t^{\QQ,L}=D_t^{\QQ_T,L_T}$ for any $t\leq T$ and thus $w^{\QQ,L}\wedge T=w^{\QQ_T,L_T}\wedge T$. The second equality follows easily from the same observation: $D_T^{\QQ,L}=D_T^{\QQ_T,L_T}$.
\end{proof}

\medskip

The following proposition connects the value function $w^{L}$ back to the cascade PDE \eqref{PDE:wave}.
 
\begin{proposition}\label{prop:fromEquil.toPDE}
Let $(\kappa,\QQ,L)\in\cE$ be such that $\EE^\QQ[\beta]<\infty$ (or, equivalently, let $\overline\QQ$ be a finite measure) and let $v:=w^{L}$, $\mu:=\QQ\circ y_0^{-1}$ and $\nu:=\QQ\circ y_\beta^{-1}$. Then,
\begin{align*}
\mathrm{div}\left\{\left(\frac{\nabla v(x)}{|\nabla v(x)|^2}
+\gamma\frac{\nabla v(x)}{|\nabla v(x)|}\right)\bone_{\{|\nabla v|>0\}}(x)\right\}= \kappa\nu - \kappa\mu 
\end{align*}
holds in the sense of distributions, or, more specifically,
\begin{align*}
&\int_{\{|\nabla v|>0\}}\bigg(\frac{\nabla v(x)}{|\nabla v(x)|^2}+\gamma\frac{\nabla v(x)}{|\nabla v(x)|}\bigg)\cdot\nabla\varphi(x)\,\mathrm{d}x = \kappa\,\int_{\RR^d}\varphi\,\mathrm{d}\mu-\kappa\,\int_{\RR^d}\varphi\,\mathrm{d}\nu, 
\end{align*}
for any $\varphi\in C_c^\infty(\RR^d)$.
\end{proposition}
\begin{proof}
Items (2) and (5) of Lemma \ref{lem:UsefulProperties} imply that $\{|\nabla v|>0\}\stackrel{\mathrm{a.e.}}{=}\{L>0\}\stackrel{\mathrm{a.e.}}{=}\{q>0\}\stackrel{\mathrm{a.e.}}{=}\{q>\frac{\gamma}{\kappa}\}$ and that $|\nabla v|=L^+$ a.e.
Taking $\varphi:=\frac{1}{|\nabla v|}\bone_{\{|\nabla v|>0\}}$ in the fixed point equation (the last line of equations \eqref{equi}), we see that
\begin{align*}
\int_{\{|\nabla v|>0\}}\left(\gamma+\frac{1}{|\nabla v|}\right)\,\mathrm{d}x=\kappa\,\EE^\QQ[\beta].    
\end{align*}
As a result, $\frac{1}{|\nabla v|}\bone_{\{|\nabla v|>0\}}$ is integrable under the assumptions of the proposition. For any $\varphi\in C_c^\infty(\RR^d)$, $\nabla\varphi\cdot\frac{\nabla v}{|\nabla v|^2}\bone_{\{|\nabla v|>0\}}$ is an integrable function with a compact support and thus we can use the fixed point equation (the last line of equations \eqref{equi}), together with Lemma \ref{lem:UsefulProperties} (5), to deduce that
\begin{align*}
\kappa\int_{\RR^d}\varphi\,\mathrm{d}\mu-\kappa\int_{\RR^d}\varphi\,\mathrm{d}\nu&=\kappa\EE^\QQ[\varphi(y_0)-\varphi(y_\beta)]\\
&=\kappa\EE^\QQ\left[-\int_0^\beta\frac{\mathrm{d}}{\mathrm{d}s}\varphi(y_s)\,\mathrm{d}s\right]\\
&=\kappa\EE^\QQ\left[\int_0^\beta\nabla\varphi(y_s)\cdot\frac{\nabla v}{|\nabla v|}(y_s)\,\mathrm{d}s\right]\\
&=\kappa\EE^\QQ\left[\int_0^\beta\nabla\varphi(y_s)\cdot\frac{\nabla v}{|\nabla v|^2}(y_s)\bone_{\{|\nabla v|>0\}}(y_s)L^+(y_s)\,\mathrm{d}s\right]\\
&=\int_{\{|\nabla v|>0\}}\nabla\varphi\cdot\frac{\nabla v}{|\nabla v|^2}(\gamma L+1)\,\mathrm{d}x\quad\\
&=\int_{\{|\nabla v|>0\}}\nabla\varphi\cdot\left(\gamma\frac{\nabla v}{|\nabla v|}+\frac{\nabla v}{|\nabla v|^2}\right)\,\mathrm{d}x,
\end{align*}
which proves the claim.
\end{proof}

The above proposition has the following obvious corollary, which provides a rigorous connection between the arrival time function $w^{\QQ,L}$ and the cascade PDE \eqref{PDE:wave}.
 
\begin{corollary}\label{cor:fromEquil.toPDE}
Using the assumption and notation of Proposition \ref{prop:fromEquil.toPDE}, the equation
\begin{align}\label{eq.cascadePDE.wQL}
\mathrm{div}\left(\frac{\nabla w^{\QQ,L}(x)}{|\nabla w^{\QQ,L}(x)|^2}
+\gamma\frac{\nabla w^{\QQ,L}(x)}{|\nabla w^{\QQ,L}(x)|}\right)=\kappa\nu-\kappa\mu 
\end{align}
holds the sense of distributions, for $x$ in the interior of $\{L>0\}$.
\end{corollary}

\smallskip

%The presence of the additional negative term `$- \kappa\mu(\mathrm{d}x)$' in the right hand side of \eqref{eq.cascadePDE.wQL} is consistent with the fact that the arrival time function solves a relaxed version of the cascade PDE \eqref{PDE:wave}, in which the equality is replaced by `$\leq$', as discussed in the introduction and confirmed by simple examples (see ??).
Corollary \ref{cor:fromEquil.toPDE} shows that any arrival time function $w^{\QQ,L}$ satisfying \eqref{eq.intro.defEquil.u} does solve the relaxed version of \eqref{PDE:wave} in the interior of $\{L>0\}$ (which contains the interior of $\{0<w^{\QQ,L}\leq t\}$ for any $t>0$). Moreover, we obtain the precise characterization of the ``relaxation term", which describes the distribution of the excess loss of boundary energy, via $\mathbf{e}^{\kappa,\QQ,L}$ (cf. \eqref{eq.intro.def.e}).

%%%%%%%%%%%%%%%%%%%%%%%%
\subsection{Perimeter of equilibrium aggregates}
\label{subse:perim.equil}
%%%%%%%%%%%%%%%%%%%%%%%%

Herein, we estimate the perimeter of the equilibrium cascade aggregates $D^{\QQ,L}_T$. 
Throughout the subsection, we assume that $d=2$ and fix an arbitrary finite-perimeter set $\Gamma_{s-}$, an arbitrary equilibrium $(\kappa,\QQ,L)\in\mathcal{E}$, any $T\in(0,\infty)$, and any Lebesgue-measurable $u(s-,\cdot)$ with $\int_{\RR^2\setminus\overline{\Gamma}_{s-}}(1+u(s-,x))^-\,\mathrm{d}x<\infty$.
Recall that $D^{\QQ,L}_T=D^{\QQ_T,L_T}_T$ denotes the union of $\Gamma_{s-}$ and the support of the occupation measure $\overline{\QQ}_T$.
We also recall 
\begin{equation}\label{eq.e.def}
\begin{split}
\mathbf{e}^{\kappa,\QQ,L}(\mathrm{d}x)= &\;\kappa\,\mathbb{Q}(y_0\in \mathrm{d}x,\theta=0) \\
&+ \bone_{\{w^{\QQ,L}<\infty\}\setminus\overline{\Gamma}_{s-}}(x)\,\big(\kappa\,\QQ(y_0\in\,\mathrm{d}x) - \kappa\,\QQ(y_\beta\in\,\mathrm{d}x)
- (1+u(s-,x))\,\mathrm{d}x\big),
\end{split}
\end{equation}
which describes the distribution of the excess loss of energy. 
The main goal of this subsection is to prove the following theorem.

\begin{theorem}\label{thm:perimeter.mainThm}
Assume that $\mathbf{e}^{\kappa,\QQ_T,L_T}\geq0$ and $\kappa\,\bone_{\partial_*\Gamma_{s-}}(x)\,\QQ(y_\theta\in \mathrm{d}x,\beta=\theta)\geq \gamma\,\bone_{\partial_*\Gamma_{s-}}(x)\,\mathcal{H}^1(\mathrm{d}x)$.
Then, $D^{\QQ,L}_T$ is a set of finite perimeter, and %its perimeter is bounded from above by $\kappa/\gamma$.
\begin{align*}
\big\|\partial D^{\QQ,L}_T\big\|(\RR^2) \leq \frac{\kappa}{\gamma}\,\QQ_T(\beta=\theta)
- \frac{1}{\gamma}\int_{\{w^{\QQ_T,L_T}<\infty\}\setminus\overline{\Gamma}_{s-}} (1+u(s-,x))\,\mathrm{d}x.
\end{align*}
\end{theorem}

%It is shown in Subsection \ref{subse:perim.minSol} how to deduce Theorem \ref{thm:peri}, along with other useful results, from Theorem \ref{thm:perimeter.mainThm}.
To prove Theorem \ref{thm:perimeter.mainThm}, we first let $O:=\RR^2\setminus\mathrm{supp}\,\overline{\QQ}_T$, and notice that $O$ is an open set and that $\RR^2\setminus D^{\QQ,L}_T = O\setminus\Gamma_{s-}$.
Next, we find a convenient decomposition of $O$.
Throughout this subsection, to ease the notation, we let $v:=w^{L_T}$.
We start with several auxiliary lemmas.

\begin{lemma}
For any open connected component of $O$, there exists a $t\in[0,T]$ such that $\{v=t\}$ contains this component.
\end{lemma}
\begin{proof}
By Lemma \ref{lem:UsefulProperties} (3), $L_T=-1/\gamma$ in any open connected component of $O$, which yields that $v$ is constant therein.
\end{proof}

\begin{lemma}\label{le:perim.mon}
For any random times $0\leq\tau,\sigma\leq\beta$, we have $\{\tau<\sigma\}\subset\{v(y_\tau)>v(y_\sigma)\}$, $\QQ_T$-a.s.
\end{lemma}
\begin{proof}
This is a simple consequence of items (4) and (5) of Lemma \ref{lem:UsefulProperties}.
\end{proof}

\begin{lemma}\label{le:perimEst.levelSet.in.O}
For any $t\in[0,T]$, the (topological) interior of $\{v=t\}$ is a subset of $O$.
\end{lemma}
\begin{proof}
We argue by contradiction and assume that there exists a point $x\in\{v=t\}\setminus O$ and an open neighborhood $A$ of this point in which $v=t$.
Since $x$ belongs to the support of $\overline{\QQ}_T$, we have $\overline{\QQ}_T(A)>0$, hence, with positive $\QQ_T$-probability, the canonical path hits an open ball inside $A$ before the killing time $\beta$. On the other hand, the function $v$ evaluated along any such path cannot be strictly decreasing between its hitting time $\tau$ of $A$ and its hitting time $\sigma$ of the smaller ball, which yields a contradiction with Lemma \ref{le:perim.mon}.
\end{proof}

\smallskip

Since $v$ is continuous and takes values in $[0,T]$, the above lemmas yield the decomposition: $O=\bigcup_{i=1}^\infty O_i$, where $O_i$ is the interior of $\{v=t_i\}$, for some $t_i\in[0,T]$.
Clearly, $O\setminus\Gamma_{s-} = \bigcup_{i} (O_i\setminus\Gamma_{s-})$. Then,
\begin{align*}
& \lim_{n\rightarrow\infty}\sum_{i=1}^n \bone_{O_i\setminus\Gamma_{s-}} = \bone_{O\setminus\Gamma_{s-}}\quad \text{in }L^1_{loc},
\end{align*}
hence, the perimeter of $O\setminus\Gamma_{s-}$ is bounded from above by the limit inferior (see, e.g., \cite[Subsection 5.2.1]{EvGa}), as $n\rightarrow\infty$, of the perimeter of $\bigcup_{i=1}^n (O_i\setminus\Gamma_{s-})$. 
Thus, to prove Theorem \ref{thm:perimeter.mainThm}, it suffices to estimate the perimeter of $\bigcup_{i=1}^n (O_i\setminus\Gamma_{s-})$.
The desired estimate is deduced from the following proposition.

\begin{proposition}\label{prop:PerimEst}
Assume $\mathbf{e}^{\kappa,\QQ_T,L_T}\geq0$ and $\kappa\,\bone_{\partial_*\Gamma_{s-}}(x)\,\QQ(y_\theta\in \mathrm{d}x,\beta=\theta)\geq \gamma\,\bone_{\partial_*\Gamma_{s-}}(x)\,\mathcal{H}^1(\mathrm{d}x)$. Then, for any $i\geq1$, $O_i\setminus\Gamma_{s-}$ is a set of finite perimeter, and there exists an $\mathcal{H}^1$-zero $\mathcal{N}_i\subset\partial_*(O_i\setminus\Gamma_{s-})$ such that
\begin{align}
\gamma\,\bone_{\partial_*(O_i\setminus\Gamma_{s-})\setminus\mathcal{N}_i}\,\mathrm{d}\mathcal{H}^1
\leq \kappa\,\bone_{\partial_*(O_i\setminus\Gamma_{s-})\setminus\mathcal{N}_i}\,\mathrm{d}(\QQ_T\circ y^{-1}_0)
- \kappa\,\bone_{\partial_*(O_i\setminus\Gamma_{s-})\setminus(\mathcal{N}_i\cup\partial_*\Gamma_{s-})}\,\mathrm{d}(\QQ_T\circ y^{-1}_\beta).
\label{eq.PerimEst.eq}
\end{align}
\end{proposition}

To deduce Theorem \ref{thm:perimeter.mainThm} from the above proposition, we first recall that, up to an $\mathcal{H}^1$-zero set, the measure-theoretic boundary $\partial_* E$ of a finite-perimeter set $E$ coincides with its \textit{reduced boundary} $\partial^* E$ (cf.~\cite{EvGa}), which in turn coincides with its \textit{essential boundary} $\partial^M E$ (cf.~\cite{ACMM}). Thus, we can replace $\partial_*(O_i\setminus\Gamma_{s-})$ by $\partial^M(O_i\setminus\Gamma_{s-})$ on the left-hand side of \eqref{eq.PerimEst.eq}.
Next, we recall (cf. \cite{ACMM}) that $\partial^M(\bigcup_{i=1}^n (O_i\setminus\Gamma_{s-})$ is contained in $\bigcup_{i=1}^n \partial^M(O_i\setminus\Gamma_{s-})$. Then, the former can be decomposed into a union of at most $n$ disjoint sets $\{E_i\}$ such that $E_i\subset \partial^M(O_i\setminus\Gamma_{s-})$. Thus, according to Proposition \ref{prop:PerimEst},
\begin{align*}
& \gamma\,\bigg\|\partial \bigcup_{i=1}^n (O_i\setminus\Gamma_{s-})\bigg\|(\RR^2)
= \gamma\,\int_{\partial_*\bigcup_{i=1}^n(O_i\setminus\Gamma_{s-})} \mathrm{d}\mathcal{H}^1
= \sum_{i=1}^n \gamma\,\int_{E_i\setminus\mathcal{N}_i} \mathrm{d}\mathcal{H}^1\\
&\leq \kappa\,\sum_{i=1}^n \big(\QQ_T\left(y_0\in E_i\setminus\mathcal{N}_i\right) 
- \QQ_T(y_\beta\in E_i\setminus(\mathcal{N}_i\cup\partial_*\Gamma_{s-}))\big)\\
&= \kappa\,\QQ_T\Big(y_0\in\bigcup_{i=1}^n E_i\setminus\mathcal{N}_i\Big)
- \kappa\,\QQ_T\Big(y_\beta\in\Big(\bigcup_{i=1}^n E_i\setminus\mathcal{N}_i\Big)\setminus\partial_*\Gamma_{s-}\Big).
\end{align*}
Next, we notice that, $\QQ_T$-a.s.: $y_0\in\{L_T<0\}\setminus\overline{\Gamma}_{s-}$ or $y_\beta\in\{L_T<0\}\setminus\overline{\Gamma}_{s-}$ implies $0=\beta<\theta$; $y_0\in\overline{\Gamma}_{s-}$ implies $0=\beta=\theta$ and $y_0\in\partial_*\Gamma_{s-}$; $y_\beta\in\overline{\Gamma}_{s-}$ implies $y_\beta\in\partial_*\Gamma_{s-}$. Then,
\begin{align*}
&\QQ_T(y_0\in\{L_T>0\}\setminus\overline{\Gamma}_{s-},\,\beta<\theta)\leq \QQ_T\left(y_\beta\in\{L_T>0\}\setminus\overline{\Gamma}_{s-}\right),\\
& \kappa\,\QQ_T(\beta=\theta) = \kappa\,\QQ_T(y_0\in\partial_*\Gamma_{s-}) + \kappa\,\QQ_T(y_0\in\{L_T>0\}\setminus\overline{\Gamma}_{s-},\,\beta=\theta)\\ 
&\geq \kappa\,\QQ_T\left(y_0\in\{L_T>0\}\cup\overline{\Gamma}_{s-}\right)
- \kappa\,\QQ_T\left(y_\beta\in\{L_T>0\}\setminus\overline{\Gamma}_{s-}\right).
\end{align*}
Due to upper semicontinuity of $L$, we have $\bigcup_{i=1}^n E_i\setminus\mathcal{N}_i\subset\bigcup_{i=1}^n\partial^M(O_i\setminus\Gamma_{s-})\subset\{L_T>0\}\cup\overline{\Gamma}_{s-}$.
Then, using the above and $\mathbf{e}^{\kappa,\QQ_T,L_T}\geq0$, we obtain:
\begin{align*}
& \kappa\,\QQ_T(\beta=\theta) \geq \kappa\,\QQ_T\Big(y_0\in \bigcup_{i=1}^n E_i\setminus\mathcal{N}_i\Big) -\kappa\,\QQ_T\Big(y_\beta\in\Big(\bigcup_{i=1}^n E_i\setminus\mathcal{N}_i\Big)\setminus\overline{\Gamma}_{s-}\Big)\\
&\phantom{?????????????}\;
+ \kappa\,\int_{\{L_T>0\}\setminus(\overline{\Gamma}_{s-}\cup\bigcup_{i=1}^n E_i\setminus\mathcal{N}_i)} \QQ_T(y_0\in \mathrm{d}x) - \QQ_T(y_\beta\in \mathrm{d}x)\\
&\geq\kappa\,\QQ_T\Big(y_0\in \bigcup_{i=1}^n E_i\setminus\mathcal{N}_i\Big) - \kappa\,\QQ_T\Big(y_\beta\in \Big(\bigcup_{i=1}^n E_i\setminus\mathcal{N}_i\Big)\setminus\partial_*\Gamma_{s-}\Big) + \int_{\{w^{\QQ_T,L_T}<\infty\}\setminus\overline{\Gamma}_{s-}} (1+u(s-,x))\,\mathrm{d}x,
\end{align*}
where we also used the fact that $\mathrm{Leb}(\bigcup_{i=1}^n E_i\setminus\mathcal{N}_i)=0$.
Collecting the above, we obtain
\begin{align*}
& \bigg\|\partial \bigcup_{i=1}^n (O_i\setminus\Gamma_{s-})\bigg\|(\RR^2)
\leq \frac{\kappa}{\gamma}\,\QQ_T(\beta=\theta)
- \frac{1}{\gamma}\int_{\{w^{\QQ_T,L_T}<\infty\}\setminus\overline{\Gamma}_{s-}} (1+u(s-,x))\,\mathrm{d}x,
\end{align*}
for all $n\geq1$. Taking $\liminf$ as $n\rightarrow\infty$ and using the lower semicontinuity of the perimeter (see, e.g., \cite[Subsection 5.2.1]{EvGa}), we deduce the statement of Theorem \ref{thm:perimeter.mainThm}.

\medskip

The remainder of this subsection is devoted to the proof of Proposition \ref{prop:PerimEst}.

\medskip

\noindent{\bf Step 1.} In this step, we show that, for any $t\in(0,T]$, we have: 
\begin{align}
\QQ_T(y_{\iota(L,t)}\in \mathrm{d}x,\beta>\iota(L,t))\geq\frac{\gamma}{\kappa}\,\|\partial\{v<t\}\|(\mathrm{d}x),
\label{eq.Step1.mainResult}
\end{align}
where we recall and slightly modify (so that it works for $t=0$ as well) the definition
\begin{align*}
&\iota(L,t)=\inf\left\{r\in[0,\theta]:\, \int_{r}^\theta L^+(y_z)\,\mathrm{d}z<t\right\}\wedge\theta.
\end{align*}

\smallskip

Consider a function $0\leq\phi\in C^1_c(\RR^2)$ and $\epsilon\in(0,t)$. 
Using the fixed-point condition (the last line of \eqref{equi}) with the test function $\phi\,\bone_{\{t-\epsilon<v<t\}}$ and making use of Lemma \ref{lem:UsefulProperties} (5) and the fact that $L^+\leq L+1/\gamma$, we obtain successively:
\begin{align*}
& \int_{\RR^2} \phi(x)\,\bone_{\{t-\epsilon<v<t\}}(x)\,(L(x)+1/\gamma)\,\mathrm{d}x = \frac{\kappa}{\gamma}\,\EE^{\QQ_T} \left[\int_0^\beta \phi(y_r)\,\bone_{\{t-\epsilon<v<t\}}(y_r)\,L^+(y_r)\,\mathrm{d}r\right],\\
& \int_{\RR^2} \phi(x)\,\bone_{\{t-\epsilon<v<t\}}(x)\,L^+(x)\,\mathrm{d}x \leq \frac{\kappa}{\gamma}\,\EE^{\QQ_T}\left[\int_0^\beta \phi(y_r)\,\bone_{\{t-\epsilon<v<t\}}(y_r)\,L^+(y_r)\,\mathrm{d}r\right],\\
& \int_{\{t-\epsilon<v<t\}} \phi(x)\,|\nabla v(x)|\,\mathrm{d}x \leq \frac{\kappa}{\gamma}\,\EE^{\QQ_T}\left[\int_0^\beta \phi(y_r)\,\bone_{\{t-\epsilon<v<t\}}(y_r)\,L^+(y_r)\,\mathrm{d}r\right].
\end{align*}
Using the coarea formula, we obtain
\begin{align*}
& \int_{\{t-\epsilon<v<t\}} \phi(x)\,|\nabla v(x)|\,\mathrm{d}x = \int_{(t-\epsilon,t)} \int_{\{v=r\}} \phi(x)\,\mathrm{d}\mathcal{H}^{1}(x)\,\mathrm{d}r.
\end{align*}
On the other hand,
\begin{align*}
& \int_0^\beta \phi(y_r)\,\bone_{\{t-\epsilon<v<t\}}(y_r)\,L^+(y_r)\,\mathrm{d}r
\leq \bone_{\{\beta>\iota(L,t)\}}\,\int_{\iota(L,t)}^{\iota(L,t-\epsilon)} \phi(y_r)\,L^+(y_r)\,\mathrm{d}r.
\end{align*}
The definition of $\iota$ and the positivity of $L(y)$ imply that $\lim_{\epsilon\downarrow0}\iota(L,t-\epsilon)=\iota(L,t)$ in $\{\beta>\iota(L,t)\}$, $\QQ_T$-a.s.
Moreover, $\phi(y_r)=\phi(y_{\iota(L,t)}) + R_r$, where 
\begin{align}
&|R_r| \leq \sup_{z\in[\iota(L,t),\iota(L,t-\epsilon)]}|\phi(y_z)-\phi(y_{\iota(L,t)})| =:\eta.\label{eq.perim.o1}
\end{align}
Recalling, in addition, that
\begin{align*}
& \int_{\iota(L,t)}^{\iota(L,t-\epsilon)} L^+(y_r)\,\mathrm{d}r \leq \epsilon,
\end{align*}
we deduce
\begin{align*}
& \int_0^\beta \phi(y_r)\,\bone_{\{t-\epsilon<v<t\}}(y_r)\,L^+(y_r)\,\mathrm{d}r
\leq \bone_{\{\beta>\iota(L,t)\}}\,\left(\phi(y_{\iota(L,t)})\,\epsilon + \epsilon\,\eta\right).
\end{align*}
All in all, we obtain
\begin{align*}
& \int_{(t-\epsilon,t)} \int_{\{v=r\}} \phi(x)\,\mathrm{d}\mathcal{H}^{1}(x)\,\mathrm{d}r
\leq \epsilon\,\frac{\kappa}{\gamma}\,\EE^{\QQ_T}\left[\bone_{\{\beta>\iota(L,t)\}}\,(\phi(y_{\iota(L,t)}) + \eta)\right].
\end{align*}
Using \eqref{eq.perim.o1}, the Bounded Convergence Theorem, and $\lim_{\epsilon\downarrow0}\iota(L,t-\epsilon)=\iota(L,t)$, we deduce that
\begin{align*}
& \lim_{\epsilon\downarrow0}\EE^{\QQ_T}[\eta] = 0.
\end{align*}
The above implies that, for any $\varepsilon>0$, there exists an increasing sequence $t_n\uparrow t$ such that, for every $n$,
\begin{align*}
& \int_{\{v=t_n\}} \phi(x)\,\mathrm{d}\mathcal{H}^{1}(x)
\leq \frac{\kappa}{\gamma}\,\EE^{\QQ_T}\left[\bone_{\{\beta>\iota(L,t)\}}\,\phi(y_{\iota(L,t)})\right] + \varepsilon.
\end{align*}

\smallskip

Next, we notice that $\bone_{\{v<t_n\}}\rightarrow \bone_{\{v<t\}}$ a.e. as $n\rightarrow\infty$, and the Bounded Convergence Theorem implies that these indicators converge in $L^1_{loc}$.
Then, $\phi(x)\,\bone_{\{v<t_n\}},\,\phi(x)\,\bone_{\{v<t\}}\in \mathrm{BV}(\RR^2)$ and the lower semicontinuity of the perimeter (see, e.g., \cite[Subsection 5.2.1]{EvGa}) yields:
\begin{align*}
\liminf_{n\rightarrow\infty}\int_{\RR^2} \mathrm{d}|\nabla (\phi\,\bone_{\{v<t_n\}})| \geq \int_{\RR^2} \mathrm{d}|\nabla (\phi\,\bone_{\{v<t\}})|,
\end{align*}
where we denote by $\mathrm{d}|\nabla f|$ the \textit{variation measure} of a $\mathrm{BV}_{loc}(\RR^2)$ function $f$ (cf. \cite{EvGa}).
Notice also that
\begin{align*}
& \int_{\RR^2} \mathrm{d}|\nabla (\phi\,\bone_{\{v<t_n\}})| = \int_{\RR^2} |\nabla \phi(x)|\,\bone_{\{v<t_n\}}(x)\,\mathrm{d}x
+ \int_{\RR^2} \phi\,\mathrm{d}\|\partial\{v<t_n\}\|,\\
& \int_{\RR^2} \mathrm{d}|\nabla (\phi\,\bone_{\{v<t\}})| = \int_{\RR^2} |\nabla \phi(x)|\,\bone_{\{v<t\}}(x)\,\mathrm{d}x
+ \int_{\RR^2} \phi\,\mathrm{d}\|\partial\{v<t\}\|,
\end{align*}
where we made use of the mutual singularity of the Lebesgue measure and $\mathrm{d}\|\partial\{v<\cdot\}\|=\mathrm{d}|\nabla\bone_{\{v<\cdot\}}|$.
Since the first integrals on the right-hand sides in the above display become equal in the $n\rightarrow\infty$ limit, we conclude that 
\begin{align*}
\liminf_{n\rightarrow\infty}\int_{\RR^2} \phi\,\mathrm{d}\|\partial\{v<t_n\}\| \geq \int_{\RR^2} \phi\,\mathrm{d}\|\partial\{v<t\}\|.
\end{align*}

\smallskip

Next, we notice
\begin{align*}
\int_{\RR^2} \phi\,\mathrm{d}\|\partial\{v<t_n\}\| 
= \int_{\partial_*\{v<t_n\}} \phi\,\mathrm{d}\mathcal{H}^1
\leq \int_{\{v=t_n\}} \phi\,\mathrm{d}\mathcal{H}^1.
\end{align*}
Recalling that $\varepsilon>0$ is arbitrary, we deduce from the above the desired
\begin{align*}
\int_{\RR^2} \phi\,\mathrm{d}\|\partial\{v<t\}\| \leq \frac{\kappa}{\gamma}\,\EE^{\QQ_T}\left[\bone_{\{\beta>\iota(L,t)\}}\,\phi(y_{\iota(L,t)})\right].
\end{align*}

\medskip

\noindent{\bf Step 2.} In this step, we first repeat the arguments in Step 1 to show that, for any $t\in[0,T)$, 
\begin{align}\label{eq.Step2.mainResult.n}
\QQ_T(y_{\iota(L,t+)}\in \mathrm{d}x,\,\beta\geq\iota(L,t+):=\lim_{r\downarrow t}\iota(L,r)>0)\geq\frac{\gamma}{\kappa}\,\|\partial\{v>t\}\|(\mathrm{d}x).
\end{align}
The above, in particular, proves that $\{v>t\}$ is a finite-perimeter set.

\smallskip

Let us show that we can replace $\iota(L,t+)$ by $\iota(L,t)$ in the above.
To this end, we notice that $\iota(L,t+)\leq \iota(L,t)$ always holds, and that $\{\iota(L,t+)<\iota(L,t)\}$ implies $\beta=\iota(L,t+)<\theta$, $\QQ_T$-a.s. Next, we notice that $\mathbf{e}^{\kappa,\QQ_T,L_T}\geq0$ implies $\QQ_T(y_\beta\in \mathrm{d}x) - \QQ_T(y_0\in \mathrm{d}x)\leq (1+u(s-,x))^-\,\mathrm{d}x$ in $\partial_*\{v>t\}\setminus\overline{\Gamma}_{s-}$, and that the Lebesgue measure of $\partial_*\{v>t\}$ is zero (since $\{v>t\}$ is a finite-perimeter set). Then in $\partial_*\{v>t\}\setminus\overline{\Gamma}_{s-}$,
\begin{align}
& \QQ_T(y_\beta\in \mathrm{d}x,\,\beta>0)\leq \EE^{\QQ_T}\left[(\bone_{\{y_\beta\in\,\mathrm{d}x\}} - \bone_{\{y_0\in \,\mathrm{d}x\}})\,\bone_{\{\beta>0\}}\right]\nonumber\\
&\leq \EE^{\QQ_T}[\bone_{\{y_\beta\in\, \mathrm{d}x\}} - \bone_{\{y_0\in\, \mathrm{d}x\}}]\leq (1+u(s-,x))^-\,\mathrm{d}x,\label{eq.perimEst.Step2.aux}\\
& \QQ_T\left(y_{\iota(L,t+)}\in \partial_*\{v>t\}\setminus\overline{\Gamma}_{s-},\,\beta\geq\iota(L,t+)>0,\,\iota(L,t+)<\iota(L,t)\right)\nonumber\\
&\leq \QQ_T\left(y_\beta\in \partial_*\{v>t\}\setminus\overline{\Gamma}_{s-},\,\beta>0\right)
\leq \int_{\partial_*\{v>t\}\setminus\overline{\Gamma}_{s-}}(1+u(s-,x))^-\,\mathrm{d}x=0.\nonumber
\end{align}
Thus, we deduce:
\begin{align}\label{eq.Step2.mainResult}
\QQ_T(y_{\iota(L,t)}\in \mathrm{d}x,\,\beta\geq\iota(L,t)>0)\geq\frac{\gamma}{\kappa}\,\|\partial\{v>t\}\|(\mathrm{d}x).
\end{align}

\smallskip

To conclude this step, we notice that $\iota(L,0)=\theta$ and that $\QQ_T(y_\theta\in\partial_*\Gamma_{s-}\,\vert\,\beta=\theta)=1$, hence, in view of the above domination, we conclude that $\partial_*\{v>0\}\subset\partial_*\Gamma_{s-}$, and in turn $\partial^M\{v>0\}\subset\partial^M\Gamma_{s-}$, up to an $\mathcal{H}^1$-zero set.

%\smallskip

%{\color{red}We need to add the requirement $\QQ(y_\theta\in\partial^*\Gamma_0)=1$ to the definition of equilibrium! And this makes physical sense as well.}

\medskip

\noindent{\bf Step 3}. In this step, we ease the notation by dropping the (fixed) index $i$ and replacing $O_i$ by $O$ and $t_i$ by $t$ (with a slight abuse of notation), and we complete the proof of Proposition \ref{prop:PerimEst}.
% show that $d(\QQ_T\circ y^{-1}_0)\geq\frac{\gamma}{\kappa}\,\mathrm{d}\|\partial (O\setminus\Gamma_{s-})\|$.

\smallskip

We begin with the following auxiliary result which, in particular, shows that $O\setminus\Gamma_{s-}$ has finite perimeter.

\begin{lemma}\label{le:PerimEst.bdryInclusions}
If $t\in(0,T]$, we have
\begin{align*}
\partial^M(O\setminus\Gamma_{s-})=\partial^MO=\partial^M(\RR^2\setminus(\{v<t\}\cup\{v>t\}))=\partial^M(\{v<t\}\cup\{v>t\})\subset \partial^M\{v<t\}\cup\partial^M\{v>t\}.
\end{align*}
If $t=0$, we have
\begin{align}
\partial^M(O\setminus\Gamma_{s-})\subset \partial^M\Gamma_{s-}\cup\partial^M\{v>0\}=\partial^M\Gamma_{s-}.
\label{eq.essbdryO.t0}
\end{align}
\end{lemma}
\begin{proof}
First, we prove the statement for  $t\in(0,T]$. In this case, $O\setminus\Gamma_{s-}=O$. Suppose that $\mathrm{Leb}(\partial\{v=t\})=0$. Then,
\begin{align*}
&\partial^MO=\partial^M\overline{O}=\partial^M\{v=t\}=\partial^M(\RR^2\setminus(\{v<t\}\cup\{v>t\}))\\
&=\partial^M(\{v<t\}\cup\{v>t\}) \subset \partial^M\{v<t\}\cup\partial^M\{v>t\},
\end{align*}
as stated in the lemma. Thus, it suffices to show that $\mathrm{Leb}(\partial\{v=t\})=0$. To this end, we argue by contradiction and assume that $\mathrm{Leb}(\partial\{v=t\})>0$.
Recall also that, according to Lemma \ref{lem:UsefulProperties} (5), $v$ is differentiable a.e. and $|\nabla v|=L^+$ a.e.
Then, due to the Lebesgue Differentiation Theorem, there exists $x\in\partial\{v=t\}$ such that $v$ is differentiable at $x$, $|\nabla v(x)|=L^+(x)$, and 
\begin{align}
&\lim_{r\downarrow0}\frac{\mathrm{Leb}(B(x,r)\cap\partial\{v=t\})}{\mathrm{Leb}(B(x,r))} = 1.\label{eq.perimEst.partialO.small.density}
\end{align}
On the other hand, any open neighborhood of $x$ contains points in the support of $\overline{\QQ}_T$, at which $L>0$, according to Lemma \ref{lem:UsefulProperties}. As $L$ is upper semicontinuous, we deduce $|\nabla v(x)|=L^+(x)>0$.
Next, we notice that, in view of \eqref{eq.perimEst.partialO.small.density}, for any small enough $r>0$, there exists $x\neq y_r\in B(x,r)\cap\partial\{v=t\}$ such that
\begin{align*}
& \frac{(y_r-x)\cdot\nabla v(x)}{|y_r-x| |\nabla v(x)|} \geq\frac12.
\end{align*}
Notice also that, as $r\downarrow0$,
\begin{align*}
&0=v(y_r) - v(x) = (y_r-x)\cdot\nabla v(x) + o(|y_r-x|) \\
&= |y_r-x| |\nabla v(x)| \left(\frac{(y_r-x)\cdot\nabla v(x)}{|y_r-x| |\nabla v(x)|} + o(1) \right)
\geq \frac{1}{4} |y_r-x| |\nabla v(x)|,
\end{align*}
which yields the desired contradiction.

\smallskip

For $t=0$, we repeat the above argument to show that $\mathrm{Leb}(\partial O)=\mathrm{Leb}(\partial\{v=0\})=\mathrm{Leb}(\partial \{v>0\})=0$. 
Next, we notice
\begin{align*}
O\setminus\Gamma_{s-} = (\RR^2\setminus\overline{\{v>0\}})\cap(\RR^2\setminus\Gamma_{s-}).
\end{align*}
The above yields
\begin{align*}
\partial^M(O\setminus\Gamma_{s-}) \subset \partial^M\overline{\{v>0\}}\cup\partial^M\Gamma_{s-}
= \partial^M\{v>0\}\cup\partial^M\Gamma_{s-}.
\end{align*}
The equality $\partial^M\Gamma_{s-}\cup\partial^M\{v>0\}=\partial^M\Gamma_{s-}$ follows from Step 2.
\end{proof}

\smallskip

Recalling that $\kappa\,\QQ_T\circ y^{-1}_\theta\geq\gamma\,\mathrm{d}\|\partial\Gamma_{s-}\|$, and using the above lemma and the results of Steps 1 and 2, we conclude that, whichever $t\in[0,T]$ is associated with $O$, we have
\begin{align*}
& \QQ_T\circ y^{-1}_{\iota(L,t)}\geq\frac{\gamma}{\kappa}\,\mathrm{d}\|\partial(O\setminus\Gamma_{s-})\|,
\end{align*}
which, in particular, implies that $O\setminus\Gamma_{s-}$ has a finite perimeter.

\smallskip

The following proposition completes this step of the proof, and thereby the proof of Proposition \ref{prop:PerimEst}.

\begin{proposition}\label{prop:PerimEst.prop2}
There exists an $\mathcal{H}^1$-zero set $\mathcal{N}\subset\partial^M(O\setminus\Gamma_{s-})$ such that, $\QQ_T$-a.s., the event $y_{\iota(L,t)}\in\partial^M(O\setminus\Gamma_{s-})\setminus\mathcal{N}$ implies $\iota(L,t)=0$.
\end{proposition}
\begin{proof}
We begin with an auxiliary general result, which explains why we make the assumption $d=2$.

\begin{lemma}\label{le:2dim.key}
Assume that $A,E\subset\RR^2$ are disjoint open sets, that $A$ has a finite perimeter, and that $E$ is connected. Then, there exists an $\mathcal{H}^1$-zero set $\mathcal{N}\subset\partial A\cap\partial E$ such that, for any $x\in(\partial A\cap\partial E)\setminus \mathcal{N}$ admitting a continuous curve $\theta:\,[0,1]\rightarrow\RR^2$ with $\theta(0)=x$ and $\theta((0,1])\subset A$, we have $x\in \partial^MA$. 
\end{lemma}
\begin{proof}
Corollary 1 in \cite{ACMM} implies the existence of rectifiable Jordan curves $\{\gamma^{+,k}\}_{k=1}^\infty$ and $\{\gamma^{-,n}\}_{n=1}^\infty$ such that any two elements of $\{\mathrm{int}(\gamma^{+,k}),\mathrm{int}(\gamma^{-,n})\}_{k,n}$ are either disjoint or one of them contains the other, and
\begin{align*}
A = \bigcup_{k=1}^\infty Y_k,\quad Y_k := \mathrm{int}(\gamma^{+,k})\setminus \bigcup_{n:\, \mathrm{int}(\gamma^{-,n})\subset\mathrm{int}(\gamma^{+,k})} \mathrm{int}(\gamma^{-,n}),
\end{align*}
up to a Lebesgue-zero set.
Moreover, $\partial^M A=\bigcup_{k}\gamma^{+,k}\cup\bigcup_n\gamma^{-,n}$, up to an $\mathcal{H}^1$-zero set, and for any $\mathrm{int}(\gamma^{+,k'})\subset\mathrm{int}(\gamma^{+,k})$ there exists an $n$ such that $\mathrm{int}(\gamma^{+,k'})\subset\mathrm{int}(\gamma^{-,n})\subset\mathrm{int}(\gamma^{+,k})$.

\smallskip

As $\partial A\supset\overline{\partial^MA}$, we have $\bigcup_{k}\gamma^{+,k}\cup\bigcup_n\gamma^{-,n}\subset\partial A$.
Since $A$ is open, it cannot contain any of the points in $\bigcup_{k}\gamma^{+,k}\cup\bigcup_n\gamma^{-,n}$. In addition, by Jordan's Theorem, any continuous curve connecting a point in $Y_k$ to a point in $\RR^2\setminus Y_k$ must intersect $\bigcup_{k}\gamma^{+,k}\cup\bigcup_n\gamma^{-,n}$. These observations are used below without being cited explicitly.

\smallskip

Consider any $x\in\partial A\cap\partial E$ that admits a continuous curve $\theta:\,[0,1]\rightarrow\RR^2$ with $\theta(0)=x$ and $\theta((0,1])\subset A$. It suffices to show that $x\in\bigcup_{k}\gamma^{+,k}\cup\bigcup_n\gamma^{-,n}$.

Assume that $\theta((0,1])$ does not intersect any $Y_k$. Then, by covering $\theta([\epsilon,1])$ with small open balls contained in $A$ and recalling that $A$ cannot intersect $\bigcup_{k}\gamma^{+,k}\cup\bigcup_n\gamma^{-,n}$ (hence our covering cannot intersect $\bigcup_{k=1}^\infty Y_k$), we obtain a set of positive Lebesgue measure that is contained in $A$ but is not contained in $\bigcup_{k=1}^\infty Y_k$, which is a contradiction. Thus, there exists a $k$ such that $\theta((0,1])\subset Y_k$. 

Since $E$ is open, connected, and disjoint with $A$, we must have either $E\subset \RR^2\setminus \overline{\mathrm{int}(\gamma^{+,k})}$ or $E\subset \mathrm{int}(\gamma^{-,n})\subset \mathrm{int}(\gamma^{+,k})$ for some $n$. Then, either $x\in\gamma^{+,k}$ or $x\in\gamma^{-,n}$, which completes the proof of the lemma.
\end{proof}

The above lemma has a useful corollary.

\begin{corollary}\label{cor:2dim.key}
Assume that $F,G\subset\RR^2$ are disjoint open sets with finite perimeters. Then, there exists an $\mathcal{H}^1$-zero set $\mathcal{N}\subset\partial F$ such that, for any $x\in\partial^M F\setminus \mathcal{N}$ admitting a continuous curve $\theta:\,[0,1]\rightarrow\RR^2$ with $\theta(0)=x$ and $\theta((0,1])\subset G$, there exists an open finite-perimeter set $A\subset G$ such that $x\in \partial^MA$. 
\end{corollary}
\begin{proof}
Applying the decomposition of $\partial^M F$ into Jordan curves $\{\gamma_i\}$, as in the proof of Lemma \ref{le:2dim.key}, we choose $\mathcal{N}_0$ so that $x\in\partial^M F$ implies that $x$ belongs to one of these Jordan curves. For each $i$, we define the open sets $A_i^+$ and $A_i^-$ as the intersections of $G$ and, respectively, the unbounded and the bounded open connected components produced by $\gamma_i$. Notice that $A_i^+$ and $A_i^-$ have finite perimeters and are subsets of $G$. We denote by $E_i^+$ and $E_i^-$, respectively, the bounded and the unbounded open connected components produced by $\gamma_i$. Then, Lemma \ref{le:2dim.key} implies the existence of $\mathcal{H}^1$-zero sets $\mathcal{N}_i^{+}$ and $\mathcal{N}_i^{-}$ such that, for any $x\in(\partial E_i^{\pm}\cap\partial A_i^{\pm})\setminus \mathcal{N}_i^{\pm}$, admitting a curve $\theta:\,[0,1]\rightarrow\RR^2$ with $\theta(0)=x$ and $\theta((0,1])\subset A_i^{\pm}$, implies $x\in \partial^MA_i^{\pm}$. We denote by $\mathcal{N}$ the union of $\mathcal{N}_0$ and all $\{\mathcal{N}_i^{\pm}\}$. Then, for any $x\in\partial^M F\setminus \mathcal{N}$, admitting a continuous curve $\theta:\,[0,1]\rightarrow\RR^2$ with $\theta(0)=x$ and $\theta((0,1])\subset G$, there exists an $i$ such that $x\in\gamma_i$ and $\theta((0,1])$ is contained in one of the open connected components produced by $\gamma$ -- w.l.o.g. we assume it is the bounded component. Then, $\theta((0,1])\subset A_i^-$ and $x\in(\partial E_i^{-}\cap\partial A_i^{-})\setminus \mathcal{N}_i^{-}$, and hence, $x\in \partial^MA_i^{-}$.
\end{proof}

\medskip

Next, we go back to the proof of Proposition \ref{prop:PerimEst.prop2}.
Recall that $O$ is the interior of $\{v=t\}$, for some $t\in[0,T]$. Assume that $t\in(0,T)$ and notice that, in this case, $O\setminus\Gamma_{s-}=O$. Consider a path $y\in\mathcal{Y}$ with $x:=y_{\iota(L,t)}\in\partial^MO$ and $\iota(L,t)>0$. Lemma \ref{le:perim.mon} implies that, $\QQ_T$-a.s., the continuous path $y_{[0,\iota(L,t))}$ is inside the open set $G:=\{v>t\}$. 
Recalling that $G$ and $F:=O$ are disjoint and have finite perimeters, we apply Corollary \ref{cor:2dim.key} to conclude that there exists an $\mathcal{H}^1$-zero set $\mathcal{N}\subset\partial O$, such that, whenever $x\in\partial^M O\setminus \mathcal{N}$, there exists an open finite-perimeter set $A\subset \{v>t\}$ such that $x\in \partial^MA$.
Applying the same argument and using $\iota(L,t)<\theta$, we conclude that there exists an open finite-perimeter set $A'\subset \{v<t\}$ such that $x\in \partial^MA'$.
All in all, we conclude that there exists an $\mathcal{H}^1$-zero set $\mathcal{N}$, such that $y_{\iota(L,t)}\in\partial^MO\setminus\mathcal{N}$ and $\iota(L,t)>0$ imply the existence of two open finite-perimeter sets $A\subset \{v>t\}$ and $A'\subset \{v<t\}$ such that $y_{\iota(L,t)}\in \partial^MA\cap\partial^MA'\cap\partial^MO$, $\QQ_T$-a.s.
Let us show that the latter is impossible.
Using \cite[Theorem 5.13]{EvGa} and increasing $\mathcal{N}$ if needed (so that it remains an $\mathcal{H}^1$-zero set), we deduce that $x=y_{\iota(L,t)}\in(\partial^MO\setminus\mathcal{N})\cap\partial^MA\cap\partial^MA'$ yields: 
\begin{align*}
\bone_{O_r(x)}&\to\bone_{H^-(O,x)}\quad\text{in }L_{\mathrm{loc}}^1(\RR^2),\\
\bone_{A_r'(x)}&\to\bone_{H^-(A',x)}\quad\text{in }L_{\mathrm{loc}}^1(\RR^2),\\
\bone_{A_r(x)}&\to\bone_{H^-(A,x)}\quad\text{in }L_{\mathrm{loc}}^1(\RR^2),
\end{align*}
as $r\downarrow0$, where the blow-up set $E_r(x)$ is defined by
\begin{align*}
E_r(x):=\{y\in\RR^2:\,r(y-x)+x\in E\}, 
\end{align*}
and the half space $H^-(E,x)$ is defined by
\begin{align*}
H^-(E,x):=\{y\in\RR^2:\,\nu_E(x)\cdot(y-x)\leq0\},   
\end{align*}
for $E=O,A,A'$ and $\nu_E(x)$ is the measure theoretic unit outer normal to $E$ at $x$.
Next, we notice that
\begin{align*}
1\ge\bone_{O_r(x)}(z)+\bone_{A_r(x)}(z)+\bone_{A_r'(x)}(z),   
\end{align*}
for any $z\in\RR^2$ as the three sets $O,A,A'$ are disjoint.
On the other hand, 
\begin{align*}
\frac32\mathrm{Leb}(B(x,R))&=\int_{B(x,R)}(\bone_{H^-(O,x)}(z)+\bone_{H^-(A,x)}(z)+\bone_{H^-(A',x)}(z))\,\mathrm{d}z\\
&\leq\liminf_{r\downarrow0}\int_{B(x,R)}(\bone_{O_r(x)}(z)+\bone_{A_r(x)}(z)+\bone_{A_r'(x)}(z))\,\mathrm{d}z\\
&\leq\int_{B(x,R)}1\,\mathrm{d}z=\mathrm{Leb}(B(x,R)),
\end{align*}
for any $R>0$, which yields a contradiction. Hence, such an $x$ does not exist, which means that $y_{\iota(L,t)}\in\partial^MO\setminus\mathcal{N}$ implies $\iota(L,t)=0$, $\QQ_T$-a.s.

\medskip

It remains to analyze the cases $t=0,T$. The case $t=T$ is trivial as $\QQ_T(\iota(L,T)=0)=1$ by the definition of the $T$-projection and the definition of $\iota(L,T)$.
In the case $t=0$, we recall \eqref{eq.essbdryO.t0} to deduce that any $x\in\partial^M(O\setminus\Gamma_{s-})$ also belongs to $\partial^M\Gamma_{s-}$. On the other hand, if $x=y_{\iota(L,0)}$ and $\iota(L,0)=\theta>0$, then $x$ also belongs to the boundary of $\{v>0\}$. Then, Corollary \ref{cor:2dim.key} implies the existence of an $\mathcal{H}^1$-zero set $\mathcal{N}$ such that $y_{\iota(L,0)}\in\partial^M(O\setminus\Gamma_{s-})\setminus\mathcal{N}$ and $\iota(L,0)>0$ imply the existence of an open finite-perimeter set $A\subset \{v>0\}$ such that $x\in \partial^MA$.
Using \cite[Theorem 5.13]{EvGa} and repeating the estimates in the preceding paragraph, we conclude that $(\partial^M(O\setminus\Gamma_{s-})\setminus\mathcal{N})\cap\partial^M\Gamma_{s-}\cap\partial^MA=\emptyset$. The latter means that $y_{\iota(L,0)}\in\partial^M(O\setminus\Gamma_{s-})\setminus\mathcal{N}$ implies $\iota(L,0)=0$, $\QQ_T$-a.s.. 
\end{proof}

\smallskip

Let us complete the proof of Proposition \ref{prop:PerimEst}, with $\partial_*(O\setminus\Gamma_{s-})$ replaced by $\partial^M(O\setminus\Gamma_{s-})$ in \eqref{eq.PerimEst.eq}, w.l.o.g. First, Proposition \ref{prop:PerimEst.prop2} and \eqref{eq.Step2.mainResult} imply that $\mathcal{H}^1(\partial^M(O\setminus\Gamma_{s-})\cap\partial^M\{v>t\})=0$ for any $t\in[0,T]$. Then, for $t\in(0,T]$, Lemma \ref{le:PerimEst.bdryInclusions} implies that 
\begin{align*}
&\bone_{\partial^M(O\setminus\Gamma_{s-})\setminus\mathcal{N}}(x)\,\mathcal{H}^1(\mathrm{d}x)
= \bone_{(\partial^M(O\setminus\Gamma_{s-})\cap\partial^M\{v<t\})\setminus\mathcal{N}}(x)\,\mathcal{H}^1(\mathrm{d}x).
\end{align*}
On the other hand, for $t\in(0,T]$, \eqref{eq.Step1.mainResult} and Proposition \ref{prop:PerimEst.prop2} yield the existence of an $\mathcal{H}^1$-zero set $\mathcal{N}$ such that
\begin{align*}
&\frac{\gamma}{\kappa}\,\bone_{(\partial^M(O\setminus\Gamma_{s-})\cap\partial^M\{v<t\})\setminus\mathcal{N}}(x)\,\mathcal{H}^1(\mathrm{d}x)
\leq \bone_{(\partial^M(O\setminus\Gamma_{s-})\cap\partial^M\{v<t\})\setminus\mathcal{N}}(x)\,\QQ_T(y_0\in \mathrm{d}x,\beta>0)\\
&\leq \bone_{\partial^M(O\setminus\Gamma_{s-})\setminus\mathcal{N}}(x)\,\QQ_T(y_0\in \mathrm{d}x,\beta>0)\\
&=\bone_{\partial^M(O\setminus\Gamma_{s-})\setminus\mathcal{N}}(x)\,\QQ_T(y_0\in \mathrm{d}x)
- \bone_{\partial^M(O\setminus\Gamma_{s-})\setminus\mathcal{N}}(x)\,\QQ_T(y_0\in \mathrm{d}x,\beta=0)\\
&=\bone_{\partial^M(O\setminus\Gamma_{s-})\setminus\mathcal{N}}(x)\,\QQ_T(y_0\in \mathrm{d}x)
- \bone_{\partial^M(O\setminus\Gamma_{s-})\setminus\mathcal{N}}(x)\,\QQ_T(y_\beta\in \mathrm{d}x),
\end{align*}
where the last equality is due to the fact that $y_\beta\in\partial^M(O\setminus\Gamma_{s-})\setminus\mathcal{N}$ and $\beta>0$ yield $\iota(L,t)>0$ and $y_\beta=y_{\iota(L,t)}$, $\QQ_T$-a.s., by Lemma \ref{le:perim.mon}, which implies $\QQ_T(y_\beta\in\partial^M(O\setminus\Gamma_{s-})\setminus\mathcal{N},\,\beta>0)=0$ via Proposition \ref{prop:PerimEst.prop2}.
The above displays imply the statement of Proposition \ref{prop:PerimEst} for $t\in(0,T]$, since, for such $t$, we have $\partial^M(O\setminus\Gamma_{s-})\cap\partial_*\Gamma_{s-}=\emptyset$. If $t=0$, then $\partial^M(O\setminus\Gamma_{s-})\subset\partial^M\Gamma_{s-}$, due to Lemma \ref{le:PerimEst.bdryInclusions}, and it suffices to prove the existence of an $\mathcal{H}^1$-zero set $\mathcal{N}$ such that
\begin{align*}
\frac{\gamma}{\kappa}\,\bone_{\partial^M(O\setminus\Gamma_{s-})\setminus\mathcal{N}}(x)\,\mathcal{H}^1(\mathrm{d}x)
\leq \bone_{\partial^M(O\setminus\Gamma_{s-})\setminus\mathcal{N}}(x)\,\QQ_T(y_0\in \mathrm{d}x).
\end{align*}
The above follows by applying Proposition \ref{prop:PerimEst.prop2} to deduce that the right-hand side of the above display is equal to $\bone_{\partial^M(O\setminus\Gamma_{s-})\setminus\mathcal{N}}(x)\,\QQ_T(y_\beta\in \mathrm{d}x,\beta=\theta)$, and then applying the assumption of Proposition \ref{prop:PerimEst}.

\smallskip

Theorem \ref{thm:perimeter.mainThm} yields the following bound on the total variation of the arrival time function of an equilibrium.

\begin{corollary}\label{cor:perim.uniformBdOnTV}
Consider any equilibrium $(\kappa,\QQ,L)\in\mathcal{E}$. Assume that $d=2$, that $\Gamma_{s-}$ is a set of finite perimeter, that $u(s-,\cdot)$ is Lebesgue-measurable with $\int_{\RR^2\setminus\overline{\Gamma}_{s-}}(1+u(s-,x))^-\,\mathrm{d}x<\infty$, that $\mathbf{e}^{\kappa,\QQ_T,L_T}\geq0$, and that $\kappa\,\bone_{\partial_*\Gamma_{s-}}(x)\,\QQ(y_\theta\in \mathrm{d}x,\beta=\theta)\geq \gamma\,\bone_{\partial_*\Gamma_{s-}}(x)\,\mathcal{H}^1(\mathrm{d}x)$.
Then, for any $T\in(0,\infty)$, we have $w^{\QQ,L}\wedge T\in \mathrm{BV}_{loc}(\RR^2)$ and
\begin{align*}
\int_{\RR^2}\mathrm{d}|\nabla(w^{\QQ,L}\wedge T)| \leq \frac{T}{\gamma}\bigg(2\kappa + \int_{\RR^2\setminus\overline{\Gamma}_{s-}} (1+u(s-,x))^-\,\mathrm{d}x \bigg).
\end{align*}
\end{corollary}
\begin{proof}
Recalling Lemma \ref{le:projectionOfw}, we obtain:
\begin{align*}
&w^{\QQ,L}\wedge T=(w^{L_T}\wedge T)\,\bone_{D^{\QQ,L}_T}+T\,\bone_{\RR^2\setminus D^{\QQ,L}_T}
=w^{L_T}\wedge T + (T-w^{L_T})^+\,\bone_{O\setminus\Gamma_{s-}},
%&=w^L\wedge T + (T-w^L)^+\,\lim_{n\rightarrow\infty}\sum_{i=1}^n \bone_{O_i}    
\end{align*}
where $O=\RR^2\setminus\mathrm{supp}\,\overline\QQ_T$. Decomposing $O$ into a countable union of open connected components, we apply Lemma \ref{lem:UsefulProperties} to deduce that $\nabla w^{L_T}=0$ a.e. in every such component, and hence, in $O$.
Proposition \ref{prop:PerimEst} states that $O$ has finite perimeter, which implies that the variation measure $\|\partial O\|$ is mutually singular with the Lebesgue measure. Then, we deduce from Theorem \ref{thm:perimeter.mainThm}:
\begin{align*}
&\mathrm{d}|\nabla(w^{\QQ,L}(x)\wedge T)| \leq |\nabla(w^{L_T}(x)\wedge T)|\,\mathrm{d}x + |\nabla(T-w^{L_T}(x))^+|\,\bone_{O\setminus\Gamma_{s-}}(x)\,\mathrm{d}x\\ 
&\phantom{????????????????????????????????????????}
+ (T-w^{L_T}(x))^+\,\|\partial D^{\QQ,L}_T\|(\mathrm{d}x),\\
&\int_{\RR^2} \mathrm{d}|\nabla(w^{\QQ,L}(x)\wedge T)| \leq \int_{\RR^2} |\nabla(w^{L_T}(x)\wedge T)|\,\mathrm{d}x
+ \frac{T}{\gamma}\bigg(\kappa - \int_{\{w^{\QQ_T,L_T}<\infty\}\setminus\overline{\Gamma}_{s-}} (1+u(s-,x))\,\mathrm{d}x \bigg)\\
& \leq \frac{T}{\gamma}\bigg(2\kappa + \int_{\RR^2\setminus\overline{\Gamma}_{s-}} (1+u(s-,x))^-\,\mathrm{d}x \bigg),
\end{align*}
where we used Lemma \ref{lem:UsefulProperties} (5) to obtain the last inequality.
\end{proof}

%!TEX root =Main.tex

\section{Minimal solution to the cascade equation} \label{se:perimeter}

%%%%%%%%%%%%%%%%%%%%%%
\subsection{Existence of a minimal solution}
\label{subse:minSol}
%%%%%%%%%%%%%%%%%%%%%%

Recall that we would like to select a solution to the relaxed version of \eqref{PDE:wave} that corresponds to a moving boundary which ``loses as little extra energy as possible for as long as possible". Since the total excess loss of energy by an equilibrium $(\kappa,\QQ,L)$ up to the physical time $t\geq0$ is given by
\begin{align*}
\mathbf{E}^{\kappa,\QQ,L}_t=\mathbf{e}^{\kappa,\QQ,L}(\{w^{\QQ,L}\leq t\}),
\end{align*}
with $\mathbf{e}^{\kappa,\QQ,L}$ defined in \eqref{eq.intro.def.e}, we search for an equilibrium $(\kappa,\QQ,L)$ that is \textit{minimal} in the sense that there exists no other equilibrium $(\hat\kappa,\hat\QQ,\hat L)$ with the property that $\mathbf{E}^{\hat\kappa,\hat\QQ,\hat L}_t \leq \mathbf{E}^{\kappa,\QQ,L}_t$, $t\in[0,T]$, with some $T>0$, and the inequality being strict at some $t\in[0,T]$.%\footnote{Example ?? shows why it is not enough to require that there exists no other equilibrium $(\hat\kappa,\hat\QQ,\hat L)$ with the property that $\mathbf{E}^{\hat\kappa,\hat\QQ,\hat L}_t \leq \mathbf{E}^{\kappa,\QQ,L}_t$ for some $t>0$.}

Of course, the above notion of minimality only makes sense if we compare equilibria with the same initial velocity (i.e., with the same initial boundary energy) and with non-negative excess loss of energy (cf. \eqref{eq.intro.defEquil.u}). Recall from \eqref{eq.intro.defEquil.boundaryCond} that the boundary condition $\partial_{\mathbf{n}_{s-}} w\,\vert_{\partial\Gamma_{s-}}=1/V_0$ translates into the equilibrium setting as
\begin{align}
&\kappa\,\mathbb{Q}(y_\theta\in \mathrm{d}x,\,\beta=\theta)=\big(\gamma+V_0(x)\big)\,\bone_{\partial_*\Gamma_{s-}}(x)\,\mathcal{H}^{d-1}(\mathrm{d}x)=:\pi(\mathrm{d}x),\label{eq.nu.constraint.equil}
\end{align}
with the same function $V_0\geq0$.
%For convenience, we denote the right hand side of the above by $\pi$ and treat it as a generic element of $\mathcal{M}(\partial_*\Gamma_{s-})$, the space of all finite Borel measures on $\partial_*\Gamma_{s-}$.
Thus, given $\Gamma_{s-}$, $u(s-,\cdot)$ and $V_0\geq0$, we search for a minimal equilibrium across all elements $(\kappa,\QQ,L)\in\mathcal{E}$ satisfying \eqref{eq.nu.constraint.equil} and $\mathbf{e}^{\kappa,\QQ,L}\geq0$.
%Note, however, that only those $\pi$ that can be written in the form of the right-hand side in \eqref{eq.nu.constraint.equil} are physically relevant.

\smallskip

Unfortunately, simple examples (see Subsection \ref{ex:1D2} and the remarks therein) indicate that a desired minimal element may not be attainable within the space of all admissible equilibria, and one needs to consider a closure of the family $\{w^{\QQ,L}\}$, parameterized by all $(\kappa,\QQ,L)\in\mathcal{E}$ that satisfy $\kappa\,\mathbb{Q}(y_\theta\in \mathrm{d}x,\beta=\theta)>\pi(\mathrm{d}x)$ and $\mathbf{e}^{\kappa,\QQ,L}\geq0$ (where we use the notation $\PP>\QQ$ for two finite measures $\PP$ and $\QQ$, if $\PP(E)>\QQ(E)$ for any measurable $E$ with $\PP(E)>0$). To ensure consistency with the input data $(u(s-,\cdot),V_0)$ and minimality, we require that a minimal solution $w$ has an approximating sequence $(w^{\QQ^n,L^n})_{n\ge1}$ that is \textit{asymptotically minimal} (property (ii) of Definition \ref{def:min}) and satisfies $\kappa^n\,\mathbb{Q}^n(y_\theta\in \mathrm{d}x,\beta=\theta)\downarrow\pi(\mathrm{d}x)$ and $\mathbf{e}^{\kappa^n,\QQ^n,L^n}\geq0$ (property (i) of Definition \ref{def:min}).

\smallskip

\begin{remark}
Note that $\kappa^n\,\mathbb{Q}^n(y_\theta\in \mathrm{d}x,\beta=\theta)>\pi(\mathrm{d}x)$ and $\kappa^n\,\QQ^n(\beta=\theta)\rightarrow  \pi(\partial_*\Gamma_{s-})$, in Definition \ref{def:min}, yield $\kappa^n\,\mathbb{Q}^n(y_\theta\in \mathrm{d}x,\beta=\theta)\rightarrow\pi(\mathrm{d}x)$ in total variation.

Note also that the condition $\liminf_{n\rightarrow\infty}\left(\mathbf{E}^{\hat{\kappa}^n,\hat{\QQ}^n,\hat{L}^n}_t-\mathbf{E}^{\kappa^n,\QQ^n,L^n}_r\right)\geq0$ for any $0\leq r<t<T$, combined with $\mathbf{E}^{\hat{\kappa}^n,\hat{\QQ}^n,\hat{L}^n}\leq\mathbf{E}^{\kappa^n,\QQ^n,L^n}$, stated in Definition \ref{def:min}, is equivalent to saying that $\mathbf{E}^{\hat{\kappa}^n,\hat{\QQ}^n,\hat{L}^n}_\cdot - \mathbf{E}^{\kappa^n,\QQ^n,L^n}_\cdot$ converges to zero in Lebesgue measure on $[0,T]$.
\end{remark}

\medskip

The rest of this subsection is devoted to the proof of Theorem \ref{thm:exist}, which states the existence of a minimal solution.
We recall the assumptions that $d=2$, the set $\Gamma_{s-}$ has finite perimeter, $\int_{\partial_*\Gamma_{s-}} (\gamma+V_0)\,\mathrm{d}\mathcal{H}^{1} < \infty$, and $\int_{\RR^2\setminus\overline{\Gamma}_{s-}} (1+u(s-,x))^-\,\mathrm{d}x<\infty$.
%We also denote $\pi(\mathrm{d}x):=\big(\gamma+V_0(x)\big)\,\bone_{\partial_*\Gamma_{s-}}(x)\,\mathcal{H}^{d-1}(\mathrm{d}x)$ to ease the notation.
%\begin{theorem}\label{prop:maxsol.exists}
%Assume that $d=2$, that $\Gamma_{s-}$ is a set of finite perimeter, and that $\int_{\RR^2\setminus\overline{\Gamma}_{s-}} (1+u(s-,x))^-\,\mathrm{d}x<\infty$. Then, for any given $\pi\in\mathcal{M}(\partial_*\Gamma_{s-})$, with $\mathrm{d}\pi\geq\gamma\,\mathrm{d}\|\partial\Gamma_{s-}\|(=\gamma\,\mathrm{d}(\mathcal{H}^1\resmes\partial_*\Gamma_{s-}))$, there exists an associated minimal solution to the cascade equation \eqref{PDE:wave}. 
%\end{theorem}
%\begin{remark}
%Note that the assumption $\mathrm{d}\pi\geq\gamma\,\mathrm{d}\|\partial\Gamma_{s-}\|$ is always satisfied by the measure $\pi$ given by the right-hand side of \eqref{eq.nu.constraint.equil}, provided $V_0\geq0$.
%\end{remark}
%\begin{proof}
%First we recall that, as stated in Remark \ref{rem:nonempty}, $\mathcal{E}_T$ is always non-empty.
%Thus, 
%\begin{align*}
%\sup_{(\mu,\QQ,L)\in \mathcal{E}_T} \EE^{\QQ} \int_0^\theta L^+(y_t)\,\mathrm{d}t\in[0,\infty].
%\end{align*}
%\smallskip
%{\color{red}For simplicity, the proof below assumes that $\RR^d\setminus\Gamma_0$ is compact. This assumption should be easy to relax by localizing the integral of $w^{\QQ,L}\wedge \frac{k+1}{n}$.}
%\smallskip
Throughout the proof, we use the following partial orders on the space $\mathcal{E}$: $(\hat\kappa,\hat\QQ,\hat L)\gtrsim_T (\kappa,\QQ,L)$ if $\mathbf{E}^{\hat{\kappa},\hat{\QQ},\hat{L}}_t \leq \mathbf{E}^{\kappa,\QQ,L}_t$ for a.e. $t\in[0,T]$ and $\hat{\kappa}^n\,\hat{\QQ}^n(y_\theta\in \mathrm{d}x,\beta=\theta) \geq \kappa^n\,\QQ(y_\theta\in \mathrm{d}x,\beta=\theta)$. It is clear that this relation is transitive. Moreover,
%due to Lemmas \ref{le:projectionOfw} and \ref{lem:EquilibriumTimeConsistency}, this partial order is preserved under the associated projection: namely, $(\hat\kappa,\hat\QQ,\hat L)\gtrsim_T (\kappa,\QQ,L)$ implies $(\hat\kappa,\hat\QQ_T,\hat L_T)\gtrsim_T (\kappa,\QQ_T,L_T)$.
%Finally,
it is easy to see that $(\hat\kappa,\hat\QQ,\hat L)\gtrsim_T (\kappa,\QQ,L)$ implies $(\hat\kappa,\hat\QQ,\hat L)\gtrsim_{T'} (\kappa,\QQ,L)$ whenever $0<T'<T$.
In addition, since $\mathbf{E}^{\kappa,\QQ_T,L_T}_t = \mathbf{E}^{\kappa,\QQ,L}_t$ for all $t\in[0,T]$, we deduce that $(\hat\kappa,\hat\QQ,\hat L)\gtrsim_{T'} (\kappa,\QQ,L)$ if and only if $(\hat\kappa,\hat\QQ_T,\hat L_T)\gtrsim_{T'} (\kappa,\QQ_T,L_T)$, for any $0<T'<T$.

\smallskip

Next, we choose an arbitrary sequence $\epsilon^n\downarrow0$ and, for any fixed $n=0,1,2,\ldots\,$, we define the sequence of equilibria $(\tilde\kappa^k,\tilde\QQ^k,\tilde L^k)_{0\leq k\leq n^2}$ recursively, as follows. We set $(\tilde\kappa^0,\tilde\QQ^0,\tilde L^0)=(\kappa_0+\epsilon^n/2,\QQ^0,L^0)$, where $\kappa_0:=\pi(\partial_*\Gamma_{s-})$, $\QQ^0$ is the image of $\kappa^{-1}_0\,\pi$ under the map that takes any $x\in\RR^d$ to the constant path $y\equiv x$, and $L^0:=-\frac1\gamma$. We choose $(\tilde\kappa^{k+1},\tilde\QQ^{k+1},\tilde L^{k+1})\gtrsim_{\frac{k}{n}} (\tilde\kappa^{k},\tilde\QQ^{k},\tilde L^{k})$, for $k=0,1,\ldots,n^2-1$, so that $(\tilde\kappa^{k+1},\tilde\QQ^{k+1},\tilde L^{k+1})\in\mathcal{C}^k_{n}$ and
\begin{align*}
&\int_0^{\frac{k+1}{n}} \mathbf{E}^{\tilde\kappa^{k+1},\tilde\QQ^{k+1},\tilde L^{k+1}}_t\,\mathrm{d}t \leq \inf_{(\kappa,\QQ,L)\in \mathcal{C}^k_{n}} \int_0^{\frac{k+1}{n}} \mathbf{E}^{\kappa,\QQ,L}_t\,\mathrm{d}t + \frac{1}{n},
\end{align*}
where
\begin{align*}
&\mathcal{C}^k_{n}:=\left\{(\kappa,\QQ,L)\in \mathcal{E}_{\frac{k+1}{n}}:\,(\kappa,\QQ,L)\gtrsim_{\frac{k}{n}}(\tilde\kappa^k,\tilde\QQ^k,\tilde L^k), \,\kappa\,\QQ(\beta=\theta)<\kappa_0+\epsilon^n,\,\mathbf{e}^{\kappa,\QQ,L}\geq0\right\}.
\end{align*}
Then, we define $(\kappa^n,\QQ^n,L^n):=(\tilde\kappa^{n^2},\tilde\QQ^{n^2},\tilde L^{n^2})$, for $n=0,1,\ldots\,$.

\smallskip

For any $k\leq n^2-1$, we have $(\kappa^n,\QQ^n,L^n)\gtrsim_{\frac{k+1}{n}} (\tilde\kappa^{k+1},\tilde\QQ^{k+1},\tilde L^{k+1})\gtrsim_{\frac{k}{n}} (\tilde\kappa^k,\tilde\QQ^k,\tilde L^k)$,
%and, hence, $(\kappa^n,\QQ^n_{(k+1)/n},L^n_{(k+1)/n})\gtrsim_{\frac{k+1}{n}} (\tilde\kappa^{k+1},\tilde\QQ^{k+1},\tilde L^{k+1})\gtrsim_{\frac{k}{n}} (\tilde\kappa^{k},\tilde\QQ^k,\tilde L^k)$.
$\kappa^n\,\QQ^n(\beta=\theta)<\kappa_0+\epsilon^n$ and $\mathbf{e}^{\kappa^n,\QQ^n,L^n}\geq0$, which implies $(\kappa^n,\QQ^n_{(k+1)/n},L^n_{(k+1)/n})\in\mathcal{C}^k_{n}$. 
Then, since it holds  $\mathbf{E}^{\kappa^n,\QQ^n_{(k+1)/n},L^n_{(k+1)/n}}_t=\mathbf{E}^{\kappa^n,\QQ^n,L^n}_t$ for all $t\in[0,(k+1)/n]$, we obtain:
\begin{align*}
&\inf_{(\kappa,\QQ,L)\in \mathcal{C}^k_{n}} \int_0^{\frac{k+1}{n}} \mathbf{E}^{\kappa,\QQ,L}_t\,\mathrm{d}t
\leq \int_0^{\frac{k+1}{n}} \mathbf{E}^{\kappa^n,\QQ^n,L^n}_t\,\mathrm{d}t\\
&\leq \int_0^{\frac{k+1}{n}} \mathbf{E}^{\tilde\kappa^{k+1},\tilde{\QQ}^{k+1},\tilde{L}^{k+1}}_t\,\mathrm{d}t
\leq \inf_{(\kappa,\QQ,L)\in \mathcal{C}^k_{n}} \int_0^{\frac{k+1}{n}} \mathbf{E}^{\kappa,\QQ,L}_t\,\mathrm{d}t + \frac{1}{n}.
\end{align*}

\smallskip

Similarly, for any $k\leq n^2-1$ and any $(\hat\kappa^n,\hat\QQ^n,\hat L^n)\gtrsim_{\frac{k+1}{n}} (\kappa^n,\QQ^n,L^n)$ with $\hat\kappa^n\,\hat\QQ^n(\beta=\theta)<\kappa_0+\epsilon^n$ and $\mathbf{e}^{\hat\kappa^n,\hat\QQ^n,\hat L^n}\geq0$, we have $(\hat\kappa^n,\hat\QQ^n,\hat L^n)\gtrsim_{\frac{k+1}{n}}(\kappa^n,\QQ^n,L^n)\gtrsim_{\frac{k+1}{n}} (\tilde\kappa^{k+1},\tilde\QQ^{k+1},\tilde L^{k+1})\gtrsim_{\frac{k}{n}} (\tilde\kappa^{k},\tilde\QQ^{k},\tilde L^{k})$. Hence,
\begin{align*}
&\inf_{(\kappa,\QQ,L)\in \mathcal{C}^k_{n}} \int_0^{\frac{k+1}{n}} \mathbf{E}^{\kappa,\QQ,L}_t\,\mathrm{d}t
\leq \int_0^{\frac{k+1}{n}} \mathbf{E}^{\hat\kappa^n,\hat{\QQ}^n,\hat{L}^n}_t\,\mathrm{d}t\\
&\leq \int_0^{\frac{k+1}{n}} \mathbf{E}^{\tilde\kappa^{k+1},\tilde{\QQ}^{k+1},\tilde{L}^{k+1}}_t\,\mathrm{d}t
\leq \inf_{(\kappa,\QQ,L)\in \mathcal{C}^k_{n}} \int_0^{\frac{k+1}{n}} \mathbf{E}^{\kappa,\QQ,L}_t\,\mathrm{d}t + \frac{1}{n}.
\end{align*}

\smallskip

Thus, assuming that the sequence $(n)$ and the time $R$ are such that for all large enough $n$ in the sequence there exists a $k\leq n^2-1$ with $R=(k+1)/n$, we obtain:
\begin{align}
& 0\leq \int_0^R \left(\mathbf{E}^{\kappa^n,\QQ^n,L^n}_r - \mathbf{E}^{\hat\kappa^n,\hat{\QQ}^n,\hat{L}^n}_r\right)\,\mathrm{d}r \leq \frac1n,\label{eq.existence.proof.keyEst}
\end{align}
for all large enough $n$.

\smallskip

To complete the proof, we apply Corollary \ref{cor:perim.uniformBdOnTV} and use the pre-compactness of $BV_{\mathrm{loc}}(\RR^2)$ in $L_{\mathrm{loc}}^1(\RR^2)$ (see e.g. \cite[Theorem 5.5]{EvGa}) to deduce the existence of a subsequence of $(n)$ along which $(w^{\QQ^n,L^n}\wedge t)$ converges in $L^1_{\mathrm{loc}}$, for any $t\geq0$. It is then clear that there exists a Borel function $w:\RR^2\rightarrow[0,\infty]$ such that these limits are given by $w\wedge t$.
In addition, by choosing a further subsequence of $(n)$ we ensure that the set $\{k/n\}_{k=1}^{n^2}$ is non-decreasing in $n$ and converges to a dense set in $\RR_+$, denoted by $S$.

Then, we notice that $\kappa^n\,\QQ^n(y_\theta\in \mathrm{d}x,\beta=\theta) \geq \tilde\kappa^0\,\tilde\QQ^0(y_\theta\in \mathrm{d}x,\beta=\theta) > \pi(\mathrm{d}x)$, $\mathbf{e}^{\kappa^n,\QQ^n,L^n}\geq0$ and $\kappa^n\,\QQ^n(\beta=\theta)\rightarrow\kappa_0=\pi(\partial_*\Gamma_{s-})$. Thus, $(\kappa^n,\QQ^n,L^n)$ satisfies the property (i) of Definition \ref{def:min}.

To verify the property (ii) of Definition \ref{def:min}, we consider an arbitrary $T>0$ and a dominating sequence $(\hat\kappa^n,\hat{\QQ}^n,\hat{L}^n) \in \mathcal{E}$, with $(\hat\kappa^n,\hat{\QQ}^n,\hat{L}^n)\gtrsim_T (\kappa^n,\QQ^n,L^n)$ and $\mathbf{e}^{\hat\kappa^n,\hat\QQ^n,\hat L^n}\geq0$.
We fix arbitrary $0\leq r<t<T$ and consider a subsequence along which $\liminf_{n\rightarrow\infty}\left(\mathbf{E}^{\hat{\kappa}^n,\hat{\QQ}^n,\hat{L}^n}_t-\mathbf{E}^{\kappa^n,\QQ^n,L^n}_r\right)$ is attained.
Then, applying \eqref{eq.existence.proof.keyEst} to any $R\in S\cap(t,T)$ and selecting a further subsequence, we conclude that $\mathbf{E}^{\hat\kappa^n,\hat\QQ^n,\hat L^n}_z-\mathbf{E}^{\kappa^n,\QQ^n,L^n}_z$ converges to zero for a.e. $z\in[0,R]$ along this subsequence. In particular, there exists $z\in(r,t)$ at which the latter convergence holds.
Then, 
\begin{align*}
& \mathbf{E}^{\hat{\kappa}^n,\hat{\QQ}^n,\hat{L}^n}_t-\mathbf{E}^{\kappa^n,\QQ^n,L^n}_r
\geq \mathbf{E}^{\hat{\kappa}^n,\hat{\QQ}^n,\hat{L}^n}_z-\mathbf{E}^{\kappa^n,\QQ^n,L^n}_z,
\end{align*}
and the right-hand side of the above converges to zero along the chosen subsequence.
%\end{proof}

%%%%%%%%%%%%%%%%%%%%%%%%%%%%
\subsection{Perimeter of minimal aggregates}
\label{subse:perim.minSol}
%%%%%%%%%%%%%%%%%%%%%%%%%%%%
Corollary \ref{cor:perim.uniformBdOnTV} allows us to estimate the total variation of a minimal solution.

\begin{corollary}\label{cor.perim.cor2}
Assuming that $d=2$ and that $\Gamma_{s-}$ is a set of finite perimeter, consider a minimal solution $w$ corresponding to Lebesgue-measurable $V_0\geq0$, with $\int_{\partial_*\Gamma_{s-}} (\gamma+V_0)\,\mathrm{d}\mathcal{H}^{1} < \infty$, and $u(s-,\cdot)$, with $\int_{\RR^2\setminus\overline{\Gamma}_{s-}}(1+u(s-,x))^-\,\mathrm{d}x<\infty$.
Then, for any $T\in(0,\infty)$, we have $w\wedge T\in \mathrm{BV}_{\mathrm{loc}}(\RR^2)$ and
\begin{align*}
\int_{\RR^2} \mathrm{d} |\nabla(w\wedge T)| \leq \frac{T}{\gamma}\bigg(2\kappa - \int_{\RR^2\setminus\overline{\Gamma}_{s-}} (1+u(s-,x))^-\,\mathrm{d}x \bigg).
\end{align*}
\end{corollary}
\begin{proof}
Since $w\wedge T$ is obtained as the $L^1_{\mathrm{loc}}$ limit of a sequence $(w^{\QQ^n,L^n}\wedge T)$, with $(\kappa^n,\QQ^n,L^n)\in\mathcal{E}$, we deduce the desired statement from Corollary \ref{cor:perim.uniformBdOnTV} and from the lower semicontinuity of the total variation.
\end{proof}

\medskip

Next, we show that the perimeter bound for equilibrium aggregates translates into the same bound for the aggregates in a minimal solution.

\begin{proposition}\label{prop:perimEst.perimLevSet}
Using the notation and assumptions of Corollary \ref{cor.perim.cor2}, for all but countably many $T\in(0,\infty)$, we have
\begin{align*}
\|\partial\{w\leq T\}\|(\RR^2) \leq \frac{1}{\gamma}\int_{\partial_*\Gamma_{s-}} (\gamma+V_0)\,\mathrm{d}\mathcal{H}^{1} - \frac{1}{\gamma}\int_{\{w\leq T\}\setminus\overline{\Gamma}_{s-}} (1+u(s-,x))\,\mathrm{d}x.
\end{align*}
\end{proposition}
\begin{proof}
First, we recall \eqref{eq.sublevSetW.Gamma} to see that, for any $(\kappa,\QQ,L)\in\mathcal{E}$, we have $\bone_{\{w^{\QQ,L}\leq T\}}=\lim_{n\rightarrow\infty}\bone_{D^{\QQ,L}_{T+1/n}}$, with the convergence being Lebesgue-a.e. and in $L^1_{\mathrm{loc}}$.
We also have: $\QQ_t(\beta=\theta)\leq\QQ(\beta=\theta)$ and $\{w^{\QQ_t,L_t}<\infty\}\setminus\overline{\Gamma}_{s-}=D^{\QQ,L}_t\setminus\overline{\Gamma}_{s-}$ for all $t>0$.
In addition, Fatou's Lemma yields:
\begin{align*}
&\liminf_{n\rightarrow\infty}\int_{\RR^2} \bone_{D^{\QQ,L}_{T+1/n}\setminus\overline{\Gamma}_{s-}}(x)\,(1+u(s-,x))\,\mathrm{d}x
\geq \int_{\{w^{\QQ,L}\leq T\}\setminus\overline{\Gamma}_{s-}} (1+u(s-,x))\,\mathrm{d}x.
\end{align*}
Then, Theorem \ref{thm:perimeter.mainThm} and the lower semicontinuity of the perimeter yield
\begin{align*}
\big\|\partial\{w^{\QQ,L}\leq T\}\big\|(\RR^2) \leq \frac{\kappa}{\gamma}\QQ(\beta=\theta) - \frac{1}{\gamma}\int_{\{w^{\QQ,L}\leq T\}\setminus\overline{\Gamma}_{s-}} (1+u(s-,x))\,\mathrm{d}x,\quad T>0.
\end{align*}

\smallskip

We now recall that there exists a sequence $(\kappa^n,\QQ^n,L^n)\in\mathcal{E}$ such that $w^{\QQ^n,L^n}\wedge T$ converges to $w\wedge T$ in $L^1_{\mathrm{loc}}$ as $n\rightarrow\infty$, for any $T>0$. Fixing a $T>0$, the sequence $(w^{\QQ^n,L^n}\wedge T)_{n\ge1}$ has a subsequence converging to $w\wedge T$ Lebesgue-a.e., denoted as the same sequence for simplicity. We claim that, for all but countably many $t\in(0,T)$, it holds that $\mathbf{1}_{\{w^{\QQ^n,L^n}\le t\}}\to\mathbf{1}_{\{w\le t\}}$ Lebesgue-a.e. and in $L^1_{\mathrm{loc}}$. Indeed, Lebesgue-a.e., 
\[
\limsup_{n\to\infty}\,\big|\mathbf{1}_{\{w^{\QQ^n,L^n}\le t\}}-\mathbf{1}_{\{w\le t\}}\big|
\le \limsup_{n\to\infty}\,\mathbf{1}_{\{w^{\QQ^n,L^n}\le t,\,w>t\}}
+  \limsup_{n\to\infty}\,\mathbf{1}_{\{w^{\QQ^n,L^n}>t,\,w\le t\}} 
\le \mathbf{1}_{\{w=t\}}.
\]
Moreover, since for any bounded open $U\subset\RR^2$ the function $t\mapsto\mathrm{Leb}(\{w\le t\}\cap U)$ is right-continuous with left limits, $\{w=t\}\cap U$ is a Lebesgue null set for all but countably many $t\in(0,T)$. Thus, $\mathbf{1}_{\{w=t\}}=0$ Lebesgue-a.e.~for all but countably many $t\in(0,T)$, and the claim readily follows.

Then, for any such $t$, we apply Fatou's Lemma to obtain
\begin{align*}
&\liminf_{n\rightarrow\infty} \int_{\RR^2}\bone_{\{w^{\QQ^n,L^n}\leq t\}\setminus\overline{\Gamma}_{s-}}(x) (1+u(s-,x))\,\mathrm{d}x
\geq \int_{\RR^2} \liminf_{n\rightarrow\infty} \bone_{\{w^{\QQ^n,L^n}\leq t\}\setminus\overline{\Gamma}_{s-}}(x) (1+u(s-,x))\,\mathrm{d}x\\
& = \int_{\RR^2} \bone_{\{w\leq t\}\setminus\overline{\Gamma}_{s-}}(x) (1+u(s-,x))\,\mathrm{d}x.
\end{align*}

\smallskip

It remains to recall that $\lim_{n\rightarrow\infty}\kappa^n\,\QQ^n(\beta=\theta) = \int_{\partial_*\Gamma_{s-}} (\gamma+V_0)\,\mathrm{d}\mathcal{H}^{1}$ and to apply the lower semicontinuity of the perimeter to infer, for all but countably many $t\in(0,T)$:
\[
\big\|\partial\{w\le t\}\big\|(\RR^2) \le 
\liminf_{n\to\infty}\,\big\|\partial\{w^{\QQ^n,L^n}\le t\}\big\|(\RR^2)
\le \frac{1}{\gamma}\int_{\partial_*\Gamma_{s-}} (\gamma+V_0)\,\mathrm{d}\mathcal{H}^{1} - \frac{1}{\gamma}\int_{\{w\leq T\}\setminus\overline{\Gamma}_{s-}} (1+u(s-,x))\,\mathrm{d}x.
\] 
%This yields \eqref{eq:per fin T}. Finally, \eqref{eq:per infin T} follows by another application of the lower semicontinuity of the variation measure.
\end{proof}

\medskip

The following corollary of Proposition \ref{prop:perimEst.perimLevSet} yields Theorem \ref{thm:peri}.

\begin{corollary}
Using the notation and assumptions of Corollary \ref{cor.perim.cor2}, we have
\begin{align*}
\big\|\partial\{w<\infty\}\big\|(\RR^2) \leq \frac{1}{\gamma}\int_{\partial_*\Gamma_{s-}} (\gamma+V_0)\,\mathrm{d}\mathcal{H}^{1} - \frac{1}{\gamma}\int_{\{w<\infty\}\setminus\overline{\Gamma}_{s-}} (1+u(s-,x))\,\mathrm{d}x.
\end{align*}
\end{corollary}
\begin{proof}
We notice that $\bone_{\{w<\infty\}}=\lim_{t\rightarrow\infty}\bone_{\{w\leq t\}}$ Lebesgue-a.e. and in $L^1_{\mathrm{loc}}$. Then, the desired statement follows from Proposition \ref{prop:perimEst.perimLevSet}, the lower semicontinuity of the perimeter, and Fatou's Lemma.
\end{proof}

\section{Examples} \label{se:example}
Throughout this section, we use $u(x)$ to denote $u(s-,x)$ and assume that $u$ is a bounded function. We start with the following technical lemma which shows that the net distribution of the birth/death of the particles outside of the aggregate is zero.  %Recall that the proposed equation for the arrival time function $w$ is given by the cascade PDE 
%\begin{equation}\label{eq: w div example}
%\mathrm{div}\bigg(\frac{\nabla w(x)}{|\nabla w(x)|^2}+\gamma\frac{\nabla w(x)}{|\nabla w(x)|}\bigg)=-(u(x)+1).
%\end{equation} 
% The following path-connectedness property of the aggregates follows easily from the definition of an aggregate $D_t^{\QQ,L}$ and is useful when analyzing the equilibria.
% \begin{lemma}\label{lem:GammaConnected1}
% Let $(\kappa,\QQ,L)\in\cE$ and $D_t:=D_t^{\QQ,L}$. Then for any $\varepsilon>0$ and any $x\in D_t$, there exists $y\in\cY$ with $|y_0-x|<\varepsilon$ and $y_{[0,\theta]}\subset D_t$.    
% \end{lemma}

\begin{lemma}\label{lem:SupportMuMinusNu}
Consider any equilibrium $(\kappa,\QQ,L)\in\cE_T$ and set $\mu:=\QQ\circ y_0^{-1}$, $\nu:=\QQ\circ y_\beta^{-1}$, $D_T:=D_T^{\QQ,L}$. Then, $\mu-\nu$ vanishes outside $D_T$. As a result, $\mu(D_T)=\nu(D_T)$.
\end{lemma}
\begin{proof}
As the support of $\overline\QQ$ (which is a closed set) is a subset of $D_T$, we immediately see that
\begin{align*}
\QQ(y_0\in \RR^d\setminus D_T,\,\beta>0)=\QQ(y_\beta\in \RR^d\setminus D_T,\,\beta>0)=0.    
\end{align*}
As a result, for any Lebesgue-measurable $A\subset \RR^d\setminus D_T$,
\begin{align*}
\mu(A)-\nu(A)=\QQ(y_0\in A)-\QQ(y_\beta\in A)=\QQ(y_0\in A,\,\beta=0)-\QQ(y_\beta\in A,\,\beta=0)=0.    
\end{align*}
The second statement follows from the above and the fact that both $\mu$ and $\nu$ are probability measures.
\end{proof}

\subsection{Example: One Dimension, One Interface}\label{ex:1D1}
Let $\Gamma_{s-}:=(-\infty,0]$ and consider any $V_0\ge0$. Note that due to the special form of $\Gamma_{s-}$, the value function of the optimal control problem \eqref{eq:Eikonal.OC.1} must be given by
\begin{align*}
w^L(x)=\int_0^xL^+(z)\,\mathrm{d}z,\quad x\in[0,\infty).
\end{align*}
\begin{lemma}\label{lem:Example11}
Consider any $T>0$ and any equilibrium $(\kappa,\QQ,L)\in\cE_T$ that satisfies the admissibility condition $\mathbf{e}^{\kappa,\QQ,L}\ge0$. 
\begin{enumerate}
    \item There exists a non-decreasing continuous function $t\mapsto\Lambda_t$ with $\Lambda_0=0$ such that $D_t^{\QQ,L}=(-\infty,\Lambda_t]=\{w^{\QQ,L}\leq t\}$ for any $t\ge0$.
    \item $\QQ$-a.s., $\beta=0$ or
    \begin{align*}
    y_t=(y_0-t)\,\bone_{[0,y_0]}(t).   
    \end{align*}
    \item Let $\mu:=\QQ\circ y_0^{-1}$, $\nu:=\QQ\circ y_\beta^{-1}$, and let $q$ be the Lebesgue density of the occupation measure $\overline\QQ$. Then,
    \begin{align*}
    q(x)=\nu([0,x])-\mu([0,x]),\quad\text{a.e. }\,x\in[0,\infty).   
    \end{align*}
    \item The following holds for any $x\ge0$:
    \begin{align*}
    \kappa\big(\nu([0,x])-\mu([0,x])\big)=\kappa\QQ(\beta=\theta)-\int_0^x(1+u(z))\,\mathrm{d}z-\mathbf{e}^{\kappa,\QQ,L}([0,x]).  
    \end{align*}
    \item The following holds for any $x\in[0,\Lambda_T)$, provided that $\Lambda_T>0$:
    \begin{align*}    \int_0^x(1+u(z))\,\mathrm{d}z\leq\kappa\QQ(\beta=\theta)-\gamma-\mathbf{e}^{\kappa,\QQ,L}([0,x])\leq\kappa\QQ(\beta=\theta)-\gamma.   
    \end{align*}
    \item If either $\kappa\QQ(\beta=\theta)\ge\gamma$ or $\Lambda_T>0$, then 
    \begin{align*}    \mathbf{E}_T^{\kappa,\QQ,L}=\mathbf{e}^{\kappa,\QQ,L}([0,\Lambda_T])\ge\gamma.    
    \end{align*}
    \item If $\kappa\QQ(\beta=\theta)\ge\gamma$ and $\Lambda_t=\Lambda_{t_0}$ for some $t_0<t\leq T$, then necessarily $\Lambda_{t_0}=\Lambda_T$ and $\mathbf{E}_t^{\kappa,\QQ,L}\ge\gamma$. 
\end{enumerate}
\end{lemma}

\begin{proof}
(1) We first prove $D_T^{\QQ,L}$ is of the form $(-\infty,\Lambda_T]$ by contradiction. Suppose there exists an $x_1>0$ such that $x_1\notin D_T^{\QQ,L}$ but $[x_1,\infty)\cap D_T^{\QQ,L}\neq\emptyset$. Consider (note that $D_T^{\QQ,L}$ is a closed set)
\begin{align*}
x^*:=\sup\{x>x_1:[x_1,x)\subset\RR\setminus D_T^{\QQ,L}\}\in(x_1,\infty),
\end{align*}
which satisfies $[x_1,x^*)\subset\RR\setminus D_T^{\QQ,L}$. It is easy to see from the definition of $x^*$ that there exists a $t\in[0,T)$ such that $\{x^*\}=\partial\{w^L>t\}=\partial_*\{w^L>t\}$. Then, applying \eqref{eq.Step2.mainResult},\footnote{Although $d=2$ is assumed throughout Subsection \ref{subse:perim.equil}, it is easy to check that only Lemma \ref{le:2dim.key} and the subsequent results rely on this assumption -- all the results preceding Lemma \ref{le:2dim.key} hold for any $d\geq1$.} we deduce that $\QQ_T(y_{\iota(L,t)}=x^*,\,\beta\geq\iota(L,t)>0)>0$.
On the other hand, the last inequality in \eqref{eq.perimEst.Step2.aux} implies $\QQ(y_{\beta}=x^*,\,\beta\geq\iota(L,t)>0)=0$, which yields $\QQ(y_{\iota(L,t)}=x^*,\,\beta>\iota(L,t))>0$. Notice also that $\beta>\iota(L,t)$ implies $y_\beta<y_{\iota(L,t)}$, $\QQ$-a.s., in view of the monotonicity of $w^L$ and Lemma \ref{le:perim.mon}. Thus, we conclude that $\overline{\QQ}([x_1,x^*))>0$ and obtain the desired contradiction. The monotonicity of $\Lambda$ is clear from its definition. To show the continuity of $\Lambda$, we notice from Lemma \ref{lem:EquilibriumTimeConsistency} that $w^L$ must remain constant in $(\Lambda_{t-},\Lambda_{t+})$, for any $t\in[0,T)$ for which the latter interval is non-empty (with the convention $0-=0$). The latter implies that, for $x^*=\Lambda_{t+}>0$, there exists $x_1\in(0,x^*)$ such that $[x_1,x^*)\subset\RR\setminus D_T^{\QQ,L}$ and $\{x^*\}=\partial\{w^L>t\}=\partial_*\{w^L>t\}$. Repeating the above argument, we obtain the desired contradiction with the existence of a jump time $t\in[0,T)$ (in $[T,\infty)$, the function $\Lambda$ remains constant due to $w^L\leq T$). The equality $D_t^{\QQ,L}=\{w^{\QQ,L}\leq t\}$ is due to the continuity of $\Lambda$.

\smallskip

\noindent (2) It is an easy consequence of (1) and the optimality of $\QQ$ in the sense of the first line of \eqref{equi}.

\smallskip

\noindent (3) Making use of the explicit form of $\QQ\circ y^{-1}$, we can evaluate the occupation measure for any Borel set $A\subset[0,\infty)$ by
\begin{align*}
\overline{\QQ}(A)&=\EE^\QQ\left[\int_0^\beta\bone_A(y_t)\,\mathrm{d}t\right]=\int\mathrm{Leb}(A\cap[y_\beta,y_0))\mathop{\QQ(\mathrm{d}y\,\mathrm{d}\beta)}\\
&=\int\mathrm{Leb}(A\cap[0,y_0))\mathop{\QQ(\mathrm{d}y\,\mathrm{d}\beta)}-\int\mathrm{Leb}(A\cap[0,y_\beta))\mathop{\QQ(\mathrm{d}y\,\mathrm{d}\beta)}\\
&=\int_0^\infty\int\bone_A(x)\bone_{\{y_0> x\}}\mathop{\QQ(\mathrm{d}y\,\mathrm{d}\beta)}\mathop{\mathrm{d}x}-\int_0^\infty\int\bone_A(x)\bone_{\{y_\beta> x\}}\mathop{\QQ(\mathrm{d}y\,\mathrm{d}\beta)}\mathop{\mathrm{d}x}\\
&=\int_A\mu((x,\Lambda_T])\mathop{\mathrm{d}x}-\int_A\nu((x,\Lambda_T])\mathop{\mathrm{d}x},
\end{align*}
implying that its Lebesgue density is given by 
\begin{align*}
q(x)=\mu((x,\Lambda_T])-\nu((x,\Lambda_T])=\nu([0,x])-\mu([0,x])
\end{align*}
for Lebesgue-a.e. $x\in[0,\infty)$. Note that the last equality is a simple corollary of Lemma \ref{lem:SupportMuMinusNu}.

\smallskip

\noindent (4) Using the definition of $\mathbf{e}^{\kappa,\QQ,L}$ and the fact that $\{\beta=\theta\}=\{y_\beta=0\}$, we obtain
\begin{align*}
\mathbf{e}^{\kappa,\QQ,L}([0,x])&=\kappa\QQ(y_0=0)+\kappa\QQ(y_0\in(0,x])-\kappa\QQ(y_\beta\in(0,x])
-\int_0^x(1+u(z))\,\mathrm{d}z\\
&=\kappa\big(\mu([0,x])-\nu((0,x])\big)-\int_0^x(1+u(z))\,\mathrm{d}z\\
&=\kappa\big(\mu([0,x])-\nu([0,x])\big)+\kappa\QQ(y_\beta=0)-\int_0^x(1+u(z))\,\mathrm{d}z\\
&=\kappa\big(\mu([0,x])-\nu([0,x])\big)+\kappa\QQ(\beta=\theta)-\int_0^x(1+u(z))\,\mathrm{d}z.
\end{align*}
The desired identity is just a rearrangement of the terms in the equation above.

\smallskip

\noindent (5) If $\Lambda_T>0$, then it follows from items (3), (4) and Lemma \ref{lem:UsefulProperties} (2) that
\begin{align*}
0\leq\mathbf{e}^{\kappa,\QQ,L}([0,x])=\kappa\QQ(\beta=\theta)-\int_0^x(1+u(z))\,\mathrm{d}z-\kappa q(x)\leq \kappa\QQ(\beta=\theta)-\int_0^x(1+u(z))\,\mathrm{d}z-\gamma   
\end{align*}
for a.e. $x\in(0,\Lambda_T)$, which yields the desired inequality for any $x\in[0,\Lambda_T)$ by the continuity of $x\mapsto\int_0^x(1+u(z))\,\mathrm{d}z$ and the right continuity of $x\mapsto\mathbf{e}^{\kappa,\QQ,L}([0,x])$.

\smallskip

\noindent (6) If $\Lambda_T=0$, then it follows from Lemma \ref{lem:SupportMuMinusNu} and item (4) that
\begin{align*}
\mathbf{E}_T^{\kappa,\QQ,L}=\mathbf{e}^{\kappa,\QQ,L}(\{0\})=\kappa\QQ(\beta=\theta)\ge\gamma.    
\end{align*}
If $\Lambda_T>0$, then taking $x\uparrow\Lambda_T$ in item (5) implies in particular
\begin{align*}
\int_0^{\Lambda_T}(1+u(z))\,\mathrm{d}z\leq\kappa\QQ(\beta=\theta)-\gamma.
\end{align*}
We can then apply Lemma \ref{lem:SupportMuMinusNu} and item (4) to obtain
\begin{align*}
\mathbf{E}_T^{\kappa,\QQ,L}=\mathbf{e}^{\kappa,\QQ,L}([0,\Lambda_T])=\kappa\QQ(\beta=\theta)-\int_0^{\Lambda_T}(1+u(z))\,\mathrm{d}z\ge\gamma.  
\end{align*}

\smallskip

\noindent (7) Recall that $\QQ$-a.s., $y_t=y_0-t$ for $t\in[0,\beta]$. If $\Lambda_t=\Lambda_{t_0}$ for some $t_0<t\leq T$, then necessarily $\QQ(y_{\iota(L,t)}>\Lambda_{t_0},\beta>0)=0$ (as otherwise $D_{t_0}^{\QQ,L}$ will be a strict subset of $D_t^{\QQ,L}$), which forces $\QQ(y_0>\Lambda_{t_0},\beta>0)=0$, and hence $\QQ(\exists t\in(0,\beta):\,y_t>\Lambda_{t_0})=0$. As a result, $D_T^{\QQ,L}=D_{t_0}^{\QQ,L}$, and thus $\Lambda_T=\Lambda_{t_0}$. We then apply item (6) to deduce $\mathbf{E}^{\kappa,\QQ,L}_t\ge\gamma$.
\end{proof}

% \begin{corollary}
% Consider any equilibrium $(\kappa,\QQ,L)\in\cE_T$. Then $\mathbf{e}^{\kappa,\QQ,L}([0,\Lambda_T])\ge\gamma$.    
% \end{corollary}
% \begin{proof}
% According to the calculation conducted in the proof of Lemma \ref{lem:Example14}:
% \begin{align*}
% \mathbf{e}^{\kappa,\QQ,L}([0,\Lambda_T])&= \kappa\left(\mu([0,\Lambda_T])-\nu([0,\Lambda_T])\right)+\kappa\QQ(\beta=\theta)-\int_0^{\Lambda_T}(1+u(z))\,\mathrm{d}z\\
% &=0+
% \end{align*}
% \end{proof}

\medskip

\begin{proposition}\label{prop:Example1Minimal}
The unique minimal solution $w$ to \eqref{PDE:wave}, with initial domain $\Gamma_{s-}$ and initial velocity $V_0$, is given by
\begin{align*}
w(x)=\int_0^x\frac{1}{V_0-\int_0^
z(u(r)+1)\mathop{\mathrm{d}r}}\mathop{\mathrm{d}z}\bone_{(0,x^*]}(x)+\infty\,\bone_{(x^*,\infty)}(x), 
\end{align*}
where
\begin{align*}
x^*:=\inf\left\{x>0:\,\int_0^x(u(z)+1)\,\mathrm{d}z>V_0\right\}\in[0,\infty].    
\end{align*}
\end{proposition}
\begin{remark}
In the special case $V_0=0$, the above expression for $w(x)$ evaluates to $w(x)=\infty\,\bone_{(0,\infty)}(x)$ by making use of the assumption that $u$ is bounded in $(0,\infty)$.    
\end{remark}
\begin{proof}
Take any sequence $\varepsilon_n>0$ with $\varepsilon_n\to0$, and define
\begin{align*}
w^n(x):=\int_0^x\frac{1}{V_0+\varepsilon_n-\int_0^z (u(r)+1)\mathop{\mathrm{d}r}}\mathop{\mathrm{d}z}\bone_{(0,x^{n*}]}(x)+\infty\,\bone_{(x^{n*},\infty)}(x), 
\end{align*}
where
\begin{align*}
x^{n*}:=\inf\left\{x>0:\,\int_0^x(u(z)+1)\,\mathrm{d}z>V_0+\varepsilon_n\right\}\in[0,\infty].    
\end{align*}
As it has been assumed that $u$ is bounded and that $(1+u)^-$ is integrable, it holds that $w^n(x^{n*}-)=\infty$. The sequence of functions $w^n$ satisfies the convergence
\begin{align*}
w^n\wedge T\to w\wedge T    
\end{align*}
pointwise and in $L_{\mathrm{loc}}^1$ for any $T>0$. Moreover, for each $n$, $w^n\in C^1([0,x^{n*}))$ with $(w^n)'(x)=\frac{1}{V_0+\varepsilon_n-\int_0^x(u(z)+1)\mathop{\mathrm{d}z}}$, and thus 
\begin{align*}
\left(\frac{1}{(w^n)'(x)}\right)'=-(u(x)+1)    
\end{align*}
holds a.e. and weakly in $[0,x^{n*})$. Consider any sequence $T_n\uparrow\infty$. By Proposition \ref{prop:equi} (with $V_0$ replaced by $V_0+\varepsilon_n$ and $T$ replaced by $T_n$ for each $n\ge1$), there exists a sequence of equilibria $(\kappa^n,\QQ^n,L^n)$ with the properties that $w^n=w^{\QQ^n,L^n}$ in $(-\infty,(w^n)^{-1}(T_n))$, that $\mathbf{E}^{\kappa^n,\QQ^n,L^n}_t=0$ for any $t\in[0,T_n)$, and that $\kappa^n\QQ^n(\beta=\theta)=\gamma+V_0+\varepsilon_n$. To justify the minimality of $w$, we take any $T\in(0,\infty)$ and any sequence $(\hat{\kappa}^n,\hat{\QQ}^n,\hat{L}^n)_{n
\ge1}$ satisfying $\hat{\kappa}^n\,\hat{\QQ}^n(y_\theta\in \mathrm{d}x,\beta=\theta)=\kappa^n\,\QQ^n(y_\theta\in \mathrm{d}x,\beta=\theta)$, $\mathbf{e}^{\hat{\kappa}^n,\hat{\QQ}^n,\hat{L}^n}\geq0$, and $\mathbf{E}^{\hat{\kappa}^n,\hat{\QQ}^n,\hat{L}^n}_t\leq\mathbf{E}^{\kappa^n,\QQ^n,L^n}_t$ for all $t\in[0,T]$, and we need to show that
\begin{align}\label{eq:MinimalityCondition2Example1.1}
&\liminf_{n\rightarrow\infty}\left(\mathbf{E}^{\hat{\kappa}^n,\hat{\QQ}^n,\hat{L}^n}_t-\mathbf{E}^{\kappa^n,\QQ^n,L^n}_r\right)\geq0\quad \text{for any}\,\,0\leq r<t<T.
\end{align}
For all sufficiently large $n$, $T_n>T$ and thus the condition \eqref{eq:MinimalityCondition2Example1.1} holds trivially, showing that $w$ is a minimal solution.

\smallskip

Now we turn to the uniqueness of minimal solutions. Suppose $\tilde w$ is possibly another minimal solution, and let $(\tilde\kappa^n,\tilde\QQ^n,\tilde L^n)\in\cE_{\tilde T_n}$ be the approximating sequence for $\tilde w$ as in Definition \ref{def:min}. Consider $\varepsilon_n:=\tilde\kappa^n\tilde\QQ^n(\beta=\theta)-\gamma-V_0$ which satisfies $\varepsilon_n>0$, $\varepsilon_n\to0$. We take any $T\in(0,\infty)$, and let the sequences of functions $(w^n)$ and of equilibria $(\kappa^n,\QQ^n,L^n)$ be constructed as at the beginning of this proof based on this sequence $(\varepsilon_n)$ and $T_n:=T$. It is clear that $\kappa^n\QQ^n(\beta=\theta)=\tilde\kappa^n\tilde\QQ^n(\beta=\theta)=\gamma+V_0+\varepsilon_n$, and we know that $\mathbf{E}^{\kappa^n,\QQ^n,L^n}_t=0$ for all $t\in[0,T)$. Therefore, it trivially holds true that $\mathbf{E}^{\kappa^n,\QQ^n,L^n}_t\leq\mathbf{E}^{\tilde{\kappa}^n,\tilde{\QQ}^n,\tilde{L}^n}_t$ for any $t\in[0,T)$. Now, condition (ii) of Definition \ref{def:min} forces that $\lim_{n\to\infty}\mathbf{E}^{\tilde\kappa^n,\tilde\QQ^n,\tilde L^n}_t=0$ for any $t\in[0,T)$. But since the choice of $T$ is arbitrary, we actually obtain that  
\begin{align}\label{eq:Example1.1}
\lim_{n\to\infty}\mathbf{E}^{\tilde\kappa^n,\tilde\QQ^n,\tilde L^n}_t=0,\quad t\in[0,\infty).  
\end{align}
The equation above actually implies that, for each fixed $t>0$, $w^{\tilde\QQ^n,\tilde L^n}(\tilde\Lambda_t^n)=t$ for all sufficiently large $n$, as $w^{\tilde\QQ^n,\tilde L^n}(\tilde\Lambda_t^n)<t$ would entail $\mathbf{E}^{\tilde\kappa^n,\tilde\QQ^n,\tilde L^n}_t\ge\gamma$ by Lemma \ref{lem:Example11} (7). In particular, $\tilde\Lambda_t^n>0$ for all sufficiently large $n$. To obtain a more accurate lower bound of $\tilde\Lambda_t^n$, we can apply Lemma \ref{lem:Example11} (3), (4) to see that 
\begin{align*}
\tilde\kappa^n\tilde q^n(x)&=\tilde\kappa^n\big(\tilde\nu^n([0,x])-\tilde\mu^n([0,x])\big)
=\tilde\kappa^n\tilde\QQ^n(\beta=\theta)-\int_0^x(1+u(z))\,\mathrm{d}z-\mathbf{e}^{\tilde\kappa^n,\tilde\QQ^n,\tilde L^n}([0,x])\\
&=\gamma+V_0+\varepsilon_n-\int_0^x(1+u(z))\,\mathrm{d}z-\mathbf{e}^{\tilde\kappa^n,\tilde\QQ^n,\tilde L^n}([0,x]),\quad\text{a.e.}\quad x\in[0,\tilde\Lambda_\infty^n).
\end{align*}
Now we apply Lemma \ref{lem:UsefulProperties} (2) and Lemma \ref{lem:w=v} to see that
\begin{equation}\label{eq:Example1.2}
\begin{split}
t&=w^{\tilde\QQ^n,\tilde L^n}(\tilde\Lambda_t^n)=w^{\tilde L^n}(\tilde\Lambda_t^n)=\int_0^{\tilde\Lambda_t^n}(\tilde L^n)^+(x)\,\mathrm{d}x=\int_0^{\tilde\Lambda_t^n}\frac{1}{\tilde\kappa^n\tilde q^n(x)-\gamma}\,\mathrm{d}x\\&=\int_0^{\tilde\Lambda_t^n}\frac{1}{V_0+\varepsilon_n-\int_0^x(1+u(z))\,\mathrm{d}z-\mathbf{e}^{\tilde\kappa^n,\tilde\QQ^n,\tilde L^n}([0,x])}\,\mathrm{d}x.   
\end{split}
\end{equation}
If $V_0>0$, then necessarily $w^{-1}(t)<x^*$, or equivalently $\sup_{x\in[0,w^{-1}(t)]}\int_0^x(1+u(z))\,\mathrm{d}z<V_0$. We claim that $\liminf_{n\to\infty}\tilde\Lambda_t^n\ge w^{-1}(t)$. Suppose, on the contrary, there exist a sufficiently small $\epsilon>0$ and a subsequence $(\tilde\Lambda_t^{n_k})_{k\ge1}$ such that $\tilde\Lambda_t^{n_k}\leq w^{-1}(t)-\epsilon$, $k\ge1$. Making use of the fact that $0\leq\mathbf{e}^{\tilde\kappa^n,\tilde\QQ^n,\tilde L^n}([0,x])\leq\mathbf{E}_t^{\tilde\kappa^n,\tilde\QQ^n,\tilde L^n}\to0$ for $x\in[0,\tilde\Lambda_t^n]$, we can let $k\to\infty$ in the equation \eqref{eq:Example1.2} with $n=n_k$ to see that
\begin{align*}
t&=\limsup_{k\to\infty}\int_0^{\tilde\Lambda_t^{n_k}}\frac{1}{V_0+\varepsilon_{n_k}-\int_0^x(1+u(z))\,\mathrm{d}z-\mathbf{e}^{\tilde\kappa^{n_k},\tilde\QQ^{n_k},\tilde L^{n_k}}([0,x])}\,\mathrm{d}x\\
&\leq\int_0^{w^{-1}(t)-\epsilon}\frac{1}{V_0-\int_0^x(1+u(z))\,\mathrm{d}z}\,\mathrm{d}x=w\big(w^{-1}(t)-\epsilon\big)<t,    
\end{align*}
a contradiction. Hence, $\liminf_{n\to\infty}\tilde\Lambda_t^n\ge w^{-1}(t)$ for any $t>0$ and therefore $\liminf_{n\to\infty}\tilde\Lambda_\infty^n\ge x^*$ by taking $t\uparrow\infty$. Now, we can go back to the above expression of $\tilde q^n(x)$ and get the convergence 
\begin{align*}
\tilde\kappa^n\tilde q^n(x)=
\gamma+V_0+\varepsilon_n-\int_0^x(1+u(z))\,\mathrm{d}z-\mathbf{e}^{\tilde\kappa^n,\tilde\QQ^n,\tilde L^n}([0,x])
\to \gamma+V_0-\int_0^x(1+u(z))\,\mathrm{d}z,
\end{align*}
as $n\to\infty$, for a.e. $x\in[0,x^*)$ with the limit being locally uniform in $x\in[0,x^*)$ as $x\mapsto\mathbf{e}^{\tilde\kappa^n,\tilde\QQ^n,\tilde L^n}([0,x])$ is non-decreasing for each $n$ and is such that $0\leq\mathbf{e}^{\tilde\kappa^n,\tilde\QQ^n,\tilde L^n}([0,x])\leq\mathbf{E}_t^{\tilde\kappa^n,\tilde\QQ^n,\tilde L^n}\to0$ for $x\in[0,\tilde\Lambda_t^n]$. Moreover, the right-hand side of the above display stays strictly above $\gamma$ on each compact subinterval of $[0,x^*)$. Using these observations, Lemma \ref{lem:UsefulProperties} (2), Lemma \ref{lem:w=v} and the Bounded Convergence Theorem, we obtain:
\begin{align*}
w^{\tilde\QQ^n,\tilde L^n}(x)=\int_0^x(\tilde L^n)^+(z)\,\mathrm{d}z=\int_0^x\frac{1}{\tilde\kappa^n\tilde q^n(z)-\gamma}\,\mathrm{d}z\to\int_0^x\frac{1}{V_0-\int_0^z(u(r)+1)\,\mathrm{d}r}\,\mathrm{d}z=w(x),
\end{align*}
as $n\to\infty$, for any $x\in[0,x^*)$. As $x\mapsto w^{\tilde\QQ^n,\tilde L^n}(x)$ must be non-decreasing, it shall hold that $\lim_{n\to\infty}w^{\tilde\QQ^n,\tilde L^n}(x)=\infty$ for any $x\ge x^*$. 

If $V_0=0$, then equation \eqref{eq:Example1.2} implies that 
\begin{align*}
t\ge\int_0^{\tilde\Lambda_t^n}\frac{1}{\varepsilon_n-\int_0^x(1+u(z))\,\mathrm{d}z}\,\mathrm{d}x.
\end{align*}
As $u$ is assumed to be bounded, the above forces $\lim_{n\to\infty}\tilde\Lambda_t^n=0$, and thus $\lim_{n\to\infty}w^{\tilde\QQ^n,\tilde L^n}(x)=\infty$, $x>0$.

Therefore, in both cases, 
\begin{align*}
\tilde w\wedge T=\lim_{n\to\infty}w^{\tilde\QQ^n,\tilde L^n}\wedge T=w\wedge T    
\end{align*}
holds a.e. and in $L_{\mathrm{loc}}^1$ for any $T>0$, which then implies $\tilde w=w$, hence concluding the proof of the uniqueness of minimal solutions.
\end{proof}

\begin{remark}\label{rem:min.lim.equil.1}
(a) If $V_0=0$, the value of $x^*$, defined above, coincides precisely with the physical jump size in \cite{DIRT2} (after a change of variables) and \cite{dns}.

\noindent(b) It is not unusual that a minimal solution $w$ does not correspond to the arrival time function of any single equilibrium. Indeed, for any $T$-capped equilibrium $(\kappa,\QQ,L)\in\mathcal{E}_T$, its arrival time function $w^{\QQ,L}$ takes its values in $[0,T]\cup\{\infty\}$, whereas the minimal solution $w$ may have its image being $[0,\infty)$ (as in this example with $u\equiv-1$) or $[0,\infty]$ (as in this example with $u\equiv0$). Thus, only the capped function $w\wedge T$ can be expected to coincide with the arrival time function of an equilibrium $(\kappa,\QQ,L)\in\mathcal{E}_T$, and one can recover $w$ as $\lim_{T_n\rightarrow\infty}w\wedge T_n$. This is one of the reasons why it is necessary to consider limits of equilibria (as opposed to the equilibria themselves) in the definition of a minimal solution.    
\end{remark}

\subsection{Example: One Dimension, Two Interfaces}\label{ex:1D2}
Let $\Gamma_{s-}:=(-\infty,0]\cup[1,\infty)$ and consider any $V_0(0),V_0(1)\ge0$.\\

Note that due to the special form of $\Gamma_{s-}$, for any $x\in(0,1)$, the value of the optimal control problem \eqref{eq:Eikonal.OC.1} can be obtained by either the path $y_t=(x-t)\,\bone_{[0,x]}(t)$ or the path $y_t=(x+t)\,\bone_{[0,1-x]}(t)+\bone_{(1-x,\infty)}(t)$. Therefore, the value function of the optimal control problem \eqref{eq:Eikonal.OC.1} must be given by
\begin{align*}
w^L(x)=\int_0^xL^+(z)\,\mathrm{d}z\wedge\int_x^1L^+(z)\,\mathrm{d}z   
\end{align*}
for $x\in[0,1]$. Similar to Lemma \ref{lem:Example11}, we list a few properties of admissible equilibria in the following lemma, but omit the proof for brevity. 

\begin{lemma}\label{lem:Example21}
Consider any $T>0$ and any equilibrium $(\kappa,\QQ,L)\in\cE_T$ that satisfies the admissibility condition $\mathbf{e}^{\kappa,\QQ,L}\ge0$. 

\begin{enumerate}
    \item There exists a non-decreasing continuous function $t\mapsto\Lambda_t^0$ with $\Lambda_0^0=0$, and a non-increasing continuous function $t\mapsto\Lambda_t^1$ with $\Lambda_0^1=1$, such that $\Lambda_t^0\leq\Lambda_t^1$ and $D_t^{\QQ,L}=(-\infty,\Lambda_t^0]\cup[\Lambda_t^1,\infty)=\{w^{\QQ,L}\leq t\}$ for any $t\ge0$.
    \item If $\Lambda_T^0<\Lambda_T^1$, then we have: $\int_0^{\Lambda_T^0}L^+(z)\,\mathrm{d}z=\int_{\Lambda_T^1}^1L^+(z)\,\mathrm{d}z$ and, $\QQ$-a.s., $\beta=0$ or
    \begin{align*}
    y_0\leq\Lambda_t^0&\implies y_t=(y_0-t)\,\bone_{[0,y_0]}(t),\\
    y_0\ge\Lambda_t^1&\implies y_t=(y_0+t)\,\bone_{[0,1-y_0]}(t)+\bone_{(1-y_0,\infty)}(t).
    \end{align*}
    If $\Lambda_T^0=\Lambda_T^1=:x^*$, then we have: $\int_0^{x^*}L^+(z)\,\mathrm{d}z=\int_{x^*}^1L^+(z)\,\mathrm{d}z$ and, $\QQ$-a.s., $\beta=0$ or
    \begin{align*}
    y_0<x^*&\implies y_t=(y_0-t)\,\bone_{[0,y_0]}(t),\\
    y_0>x^*&\implies y_t=(y_0+t)\,\bone_{[0,1-y_0]}(t)+\bone_{(1-y_0,\infty)}(t),\\
    y_0=x^*&\implies y_t=(y_0-t)\,\bone_{[0,y_0]}(t)\quad\text{or }\quad y_t=(y_0+t)\,\bone_{[0,1-y_0]}(t)+\bone_{(1-y_0,\infty)}(t).
    \end{align*}
    \item Let $\mu:=\QQ\circ y_0^{-1}$, $\nu:=\QQ\circ y_\beta^{-1}$, and let $q$ be the Lebesgue density of the occupation measure $\overline\QQ$. Then,
    \begin{align*}
    q(x)=
    \begin{cases}
    \nu([0,x])-\mu([0,x]),&\quad\text{a.e. }\,x\in[0,\Lambda_T^0),\\
    \nu([x,1])-\mu([x,1]),&\quad\text{a.e. }\,x\in(\Lambda_T^1,1].
    \end{cases}  
    \end{align*}
    \item The following holds for any $x\in[0,\Lambda_T^0)$ (including $\Lambda_T^0$ if $\Lambda_T^0<\Lambda_T^1$):
\begin{align*}
\kappa\big(\nu([0,x])-\mu([0,x])\big)=\kappa\QQ(y_\theta=0,\beta=\theta)-\int_0^x(1+u(z))\,\mathrm{d}z-\mathbf{e}^{\kappa,\QQ,L}([0,x]).  
\end{align*}
The following holds for any $x\in(\Lambda_T^1,1]$ (including $\Lambda_T^1$ if $\Lambda_T^0<\Lambda_T^1$):
\begin{align*}
\kappa\big(\nu([x,1])-\mu([x,1])\big)=\kappa\QQ(y_\theta=1,\beta=\theta)-\int_x^1(1+u(z))\,\mathrm{d}z-\mathbf{e}^{\kappa,\QQ,L}([x,1]).  
\end{align*}
Finally, the following identity holds if $\Lambda_T^0=\Lambda_T^1$:
\begin{align*}
0=\kappa\QQ(\beta=\theta)-\int_0^1(1+u(z))\,\mathrm{d}z-\mathbf{e}^{\kappa,\QQ,L}([0,1]).  
\end{align*}
\end{enumerate}
\end{lemma}

\begin{proposition}\label{prop:Example2Minimal}
The unique minimal solution $w$ to \eqref{PDE:wave}, with initial domain $\Gamma_{s-}$ and initial velocity $V_0$, is given by
\begin{align*}
w(x):=w_0(x)\wedge w_1(x),    
\end{align*}
where
\begin{align*}
w_0(x)&:=\int_0^x\frac{1}{V_0(0)-\int_0^z(u(r)+1)\mathop{\mathrm{d}r}}\mathop{\mathrm{d}z}\,\bone_{(0,x_0^*]}(x)+\infty\,\bone_{(x_0^*,1)}(x),\\
x_0^*&:=\inf\left\{x\in(0,1):\,\int_0^x(u(z)+1)\,\mathrm{d}z>V_0(0)\right\}\wedge1,\\  
w_1(x)&:=\int_x^1\frac{1}{V_0(1)-\int_z^1(u(r)+1)\mathop{\mathrm{d}r}}\mathop{\mathrm{d}z}\bone_{[x_1^*,1)}(x)+\infty\,\bone_{(0,x_1^*)}(x),\\
x_1^*&:=\sup\left\{x\in(0,1):\,\int_x^1(u(z)+1)\,\mathrm{d}z>V_0(1)\right\}\vee0.  
\end{align*}
\end{proposition}
\begin{remark}
Similar to Proposition \ref{prop:Example1Minimal}, for $i=0,1$, if $V_0(i)=0$, then the above expression for $w_i(x)$ evaluates to $w_i(x)=\infty\,\bone_{(0,1)}(x)$.  
\end{remark}
\begin{proof}
We only outline a sketch of the proof, as it is similar to the proof of Proposition \ref{prop:Example1Minimal}. Consider any two positive sequences $\varepsilon_n^0,\varepsilon_n^1\to0$ and define
\begin{align*}
w^n(x):=w_0^n(x)\wedge w_1^n(x),    
\end{align*}
where
\begin{align*}
w_0^n(x)&:=\int_0^x\frac{1}{V_0(0)+\varepsilon_n^0-\int_0^z(u(r)+1)\mathop{\mathrm{d}r}}\mathop{\mathrm{d}z}\bone_{(0,x_0^{n*}]}(x)+\infty\,\bone_{(x_0^{n*},1)}(x),\\
x_0^{n*}&:=\inf\left\{x\in(0,1):\,\int_0^x(u(z)+1)\,\mathrm{d}z>V_0(0)+\varepsilon_n^0\right\}\wedge1,\\  
w_1^n(x)&:=\int_x^1\frac{1}{V_0(1)+\varepsilon_n^1-\int_z^1(u(r)+1)\mathop{\mathrm{d}r}}\mathop{\mathrm{d}z}\bone_{[x_1^{n*},1)}(x)+\infty\,\bone_{(0,x_1^{n*})}(x),\\
x_1^{n*}&:=\sup\left\{x\in(0,1):\,\int_x^1(u(z)+1)\,\mathrm{d}z>V_0(1)+\varepsilon_n^1\right\}\vee0.  
\end{align*}
The sequence of functions $w^n$ satisfies
\begin{align*}
w^n\wedge T\to w\wedge T    
\end{align*}
pointwise and in $L_{\mathrm{loc}}^1$, for any $T>0$. Moreover, if we let $T_n:=n\wedge\sup_{x\in[0,1]}w^n(x)$ for each $n\ge1$, then we have: $w^n\in C^1(\{0\leq w^n<T_n\})$ and
\begin{align*}
\left(\frac{1}{(w^n)'(x)}\right)'=-(u(x)+1)    
\end{align*}
holds a.e. and weakly in $\{0\leq w^n<T_n\}$. By Proposition \ref{prop:equi} (with $V_0(0)$, $V_0(1)$ replaced by $V_0(0)+\varepsilon_n^0$, $V_0(1)+\varepsilon_n^1$ and $T$ replaced by $T_n$ for each $n\ge1$), there exists a sequence of equilibria $(\kappa^n,\QQ^n,L^n)$ with the properties that $w^n=w^{\QQ^n,L^n}$ in $\{0\leq w^n<T_n\}$, that $\mathbf{E}^{\kappa^n,\QQ^n,L^n}_t=0$ for any $t\in[0,T_n)$, and that $\kappa^n\QQ^n(y_\theta=0,\beta=\theta)=\gamma+V_0(0)+\varepsilon_n^0$, $\kappa^n\QQ^n(y_\theta=1,\beta=\theta)=\gamma+V_0(1)+\varepsilon_n^1$.

\smallskip

To justify the minimality of $w$, we take any $T\in(0,\infty)$ and any sequence $(\hat{\kappa}^n,\hat{\QQ}^n,\hat{L}^n)_{n\ge1}$ satisfying $\hat{\kappa}^n\,\hat{\QQ}^n(y_\theta\in \mathrm{d}x,\beta=\theta)=\kappa^n\,\QQ^n(y_\theta\in \mathrm{d}x,\beta=\theta)$, $\mathbf{e}^{\hat{\kappa}^n,\hat{\QQ}^n,\hat{L}^n}\geq0$, and $\mathbf{E}^{\hat{\kappa}^n,\hat{\QQ}^n,\hat{L}^n}_t\leq\mathbf{E}^{\kappa^n,\QQ^n,L^n}_t$ for all $t\in[0,T]$, and we need to show that
\begin{align}\label{eq:MinimalityCondition2Example2.1}
&\liminf_{n\rightarrow\infty}\left(\mathbf{E}^{\hat{\kappa}^n,\hat{\QQ}^n,\hat{L}^n}_s-\mathbf{E}^{\kappa^n,\QQ^n,L^n}_r\right)\geq0\quad \text{for any}\,\,0\leq r<s<T.
\end{align}
If
\begin{align*}
T^*:=\lim_{n\to\infty}T_n=\lim_{n\to\infty}\sup_{x\in[0,1]}w^n(x)=\sup_{x\in[0,1]}w(x)    
\end{align*}
is infinite, then $T_n>T$ for all sufficiently large $n$, and thus the condition \eqref{eq:MinimalityCondition2Example2.1} holds trivially, showing that $w$ is a minimal solution. 

If, on the contrary, $T^*<\infty$, then necessarily $x_0^{n*}>x_1^{n*}$ and $\Lambda_{T_n}^{n0}=\Lambda_{T_n}^{n1}$ for all sufficiently large $n$. In that case, $\mathbf{E}^{\kappa^n,\QQ^n,L^n}_t=0$ for any $t\in[0,T_n)$, and, according to Lemma \ref{lem:Example21} (4),
\begin{align*}
\mathbf{E}^{\kappa^n,\QQ^n,L^n}_t=\mathbf{E}^{\kappa^n,\QQ^n,L^n}_{T_n}=\mathbf{e}^{\kappa^n,\QQ^n,L^n}([0,1])=\kappa^n\QQ^n(\beta=\theta)-\int_0^1(1+u(z))\,\mathrm{d}z     
\end{align*}
holds for any $t\in[T_n,\infty)$. If $T\leq T^*$, then, for any $r<T$, we have $T_n>r$ for all sufficiently large $n$, and thus the condition \eqref{eq:MinimalityCondition2Example2.1} holds trivially. It remains to address the case $T>T^*$ and $r\ge T^*$. Note that the condition $\mathbf{E}^{\hat{\kappa}^n,\hat{\QQ}^n,\hat{L}^n}_t\leq\mathbf{E}^{\kappa^n,\QQ^n,L^n}_t$ for all $t\in[0,T]$ forces $\mathbf{E}^{\hat{\kappa}^n,\hat{\QQ}^n,\hat{L}^n}_t=0$ for any $t\in[0,T_n)$, which, combined with Lemma \ref{lem:Example21} (3), (4), Lemma \ref{lem:UsefulProperties} (2) and Lemma \ref{lem:w=v}, implies
\begin{align*}
w^{\hat\QQ^n,\hat L^n}(x)=\int_0^x(\hat L^n)^+(z)\,\mathrm{d}z=\int_0^x\frac{1}{\hat\kappa^n\,\hat q^n(z)-\gamma}\,\mathrm{d}z=\int_0^x\frac{1}{V_0(0)+\varepsilon_n^0-\int_0^z(u(r)+1)\mathop{\mathrm{d}r}}\mathop{\mathrm{d}z}=w^n(x)    
\end{align*}
for $x\in[0,\hat\Lambda_{T_n}^{n0})$, and, similarly, $w^{\hat\QQ^n,\hat L^n}(x)=w^n(x)$ for $x\in(\hat\Lambda_{T_n}^{n1},1]$. Then, we conclude that $\hat\Lambda_t^{n0}=\Lambda_t^{n0}$ and $\hat\Lambda_t^{n1}=\Lambda_t^{n1}$ for a.e. $t\in[0,T_n)$, which then forces $\hat\Lambda_{T_n}^{n0}=\Lambda_{T_n}^{n0}=\Lambda_{T_n}^{n1}=\hat\Lambda_{T_n}^{n1}$. Next, we invoke Lemma \ref{lem:Example21} (4) again to obtain
\begin{align*}
\mathbf{E}^{\hat\kappa^n,\hat\QQ^n,\hat L^n}_t=\mathbf{E}^{\hat\kappa^n,\hat\QQ^n,\hat L^n}_{T_n}=\mathbf{e}^{\hat\kappa^n,\hat\QQ^n,\hat L^n}([0,1])&=\hat\kappa^n\hat\QQ^n(\beta=\theta)-\int_0^1(1+u(z))\,\mathrm{d}z\\
&=\kappa^n\QQ^n(\beta=\theta)-\int_0^1(1+u(z))\,\mathrm{d}z=\mathbf{E}^{\kappa^n,\QQ^n,L^n}_t,  
\end{align*}
for any $t\in[T_n,\infty)$. In conclusion, we have established that $\mathbf{E}^{\hat{\kappa}^n,\hat{\QQ}^n,\hat{L}^n}_t=\mathbf{E}^{\kappa^n,\QQ^n,L^n}_t$ for all $t\ge0$, thus proving the minimality of $w$.

\smallskip

Now we turn to the uniqueness of the minimal solution, and we focus on the case $V_0(0),V_0(1)>0$ to demonstrate the idea. Let $\tilde w$ be another minimal solution, and let $(\tilde\kappa^n,\tilde\QQ^n,\tilde L^n)\in\cE_{\tilde T_n}$ be the approximating sequence for $\tilde w$, as in Definition \ref{def:min}. For any fixed $T\in[0,\tilde T^*)$ (where $\tilde T^*:=\sup_{x\in[0,1]}\tilde w(x)=\lim_{n\to\infty}\sup_{x\in[0,1]}w^{\tilde\QQ^n,\tilde L^n}(x)$), we have $\sup_{x\in[0,1]}w^{\tilde\QQ^n,\tilde L^n}(x)>T$ and $\tilde\Lambda_T^{n0}<\tilde\Lambda_T^{n1}$ for all sufficiently large $n$. That means we can work with $\tilde\Lambda_t^{n0}$ and $\tilde\Lambda_t^{n1}$ separately and proceed in a similar manner as in the proof of Proposition \ref{prop:Example1Minimal} to obtain successively $\liminf_{n\to\infty}\tilde\Lambda_\infty^{n0}\ge\inf\,\{w=T^*\}$, $\limsup_{n\to\infty}\tilde\Lambda_\infty^{n1}\leq\sup\,\{w=T^*\}$, $\tilde T^*\ge T^*$, $\tilde w(x)=\lim_{n\to\infty}w^{\tilde\QQ^n,\tilde L^n}(x)=w(x)$ for a.e. $x\in[0,\inf\,\{w=T^*\})\cup(\sup\,\{w=T^*\},1]=\{w<T^*\}$. If $T^*<\infty$, then $\{w=T^*\}=\{x^*\}$, where $x^*:=\mathrm{arg\,max}_{x\in[0,1]}\,w(x)$, and the above already shows that $\tilde w=w$ a.e. If $T^*=\infty$, then the above implies that $\tilde w=w$ a.e. in $[0,x_0^*)\cup(x_1^*,1]$. In addition, due to the shape of $D_t^{\tilde\QQ^n,\tilde L^n}$ as given in Lemma \ref{lem:Example21} (1), we must have 
\begin{align*}
w^{\tilde\QQ^n,\tilde L^n}(x)\ge\min\{w^{\tilde\QQ^n,\tilde L^n}(x_0^*-\varepsilon),w^{\tilde\QQ^n,\tilde L^n}(x_1^*+\varepsilon)\}
\end{align*}
for $x\in[x_0^*,x_1^*]$. But the right-hand side diverges to $\infty$ as $n\to\infty$ and $\varepsilon\to0$. Hence, we also obtain $\tilde w(x)=\lim_{n\to\infty}w^{\tilde\QQ^n,\tilde L^n}(x)=w(x)$, a.e. $x\in[x_0^*,x_1^*]$. We thus conclude the proof of the uniqueness of the minimal solution.

If there exists $i\in\{0,1\}$ such that $V_0(i)=0$, then similar to the proof of Proposition \ref{prop:Example1Minimal}, we obtain $\lim_{n\to\infty}(\tilde\Lambda_\infty^{ni}-\tilde\Lambda_0^{ni})=0$. The rest of the argument can be carried out similarly as above by making use of Lemma \ref{lem:Example21} (1) and is omitted.
\end{proof}

\begin{remark}\label{rem:min.lim.equil.2}
(a) If $V_0(0)=0$ and $V_0(1)>0$, even the capping $w\wedge T$ of the resulting minimal solution $w$ does not correspond to the arrival time function of any equilibrium (cf. Remark \ref{rem:min.lim.equil.1}). Indeed, with $u\equiv-1$, the minimal solution $w$ in this example has a discontinuity at $x^*=0$, where it jumps from zero to a finite positive number. On the other hand, Lemma \ref{lem:w=v} and the Lipschitz property of $w^L$ (cf. Lemma \ref{lem:UsefulProperties}) imply that the only discontinuity points of the arrival time function of any equilibrium are the ones where the function jumps to infinity. Moreover, assuming $u\equiv-1$ for simplicity and considering any equilibrium $(\kappa,\QQ,L)$ with $\mathbf{e}^{\kappa,\QQ,L}\ge0$ and $\kappa\QQ(y_\theta=0,\,\beta=\theta)\leq\gamma$, we apply Lemma \ref{lem:Example21} (3) and (4) to deduce that
\begin{align*}
\kappa q(x)=\kappa\big(\nu([0,x])-\mu([0,x])\big)\leq\kappa\QQ(y_\theta=0,\,\beta=\theta)\leq\gamma\quad\text{a.e.}\,\,x\in[0,\Lambda_T^0),
\end{align*}
which, combined with Lemma \ref{lem:UsefulProperties} (2), ensures that $\Lambda_T^0=0$ and, in turn, that $\Lambda_T^1=1$ by Lemma \ref{lem:Example21}. In other words, any equilibrium corresponding to $V_0(0)=0$ and $V_0(1)>0$ has to be trivial, in the sense that all particles are born, and immediately killed, at the initial boundary. On the other hand, we can obtain the minimal solution $w$ as a limit of the arrival time functions of equilibria with initial velocities equal to $\varepsilon_n^0\downarrow0$ at $x=0$. This example illustrates the necessity to ``increase the initial velocity" of the boundary in Definition \ref{def:min} and provides another reason why it is necessary to consider limits of equilibria (as opposed to the equilibria themselves) in the definition of a minimal solution.    

\noindent(b) If the function $u$ is chosen so that $x_0^*>x_1^*$, then necessarily $\sup_{x\in[0,1]}w(x)<\infty$, which means that the initial liquid region vanishes in a finite time with some remaining boundary energy ``locked inside" the solid domain after the freezing is over. In such a case, $w$ is not differentiable at $x^*=\mathrm{arg\,max}_{x\in[0,1]}\,w(x)$ as $w'(x^*-)>0>w'(x^*+)$, which becomes a point of singularity for the PDE \eqref{PDE:wave}. In the sense of distributions, $w$ satisfies
\begin{align*}
\mathrm{div}\bigg(\frac{\nabla w(x)}{|\nabla w(x)|^2}+\gamma\frac{\nabla w(x)}{|\nabla w(x)|}\bigg)=-(u(x)+1)-\bigg(2\gamma+V_0(0)+V_0(1)-\int_0^1(1+u(z))\mathop{\mathrm{d}z}\bigg)\cdot\delta_{x^*}
\end{align*}
in the domain $(0,1)$, where the last term in the right-hand side of the above records the energy ``locked in" at $x^*$. This example illustrates that, even in the domain $\{x:\,0<w(x)<\infty\}$, the PDE \eqref{PDE:wave} can only be solved with ``$=$" replaced by ``$\leq$", or, equivalently, we must relax \eqref{PDE:wave} by allowing for an additional non-positive measure-valued summand on its right-hand side.
\end{remark}

\subsection{Example: General Dimension with Radial Symmetry, Growing Ice Ball}\label{ex:Radial1}
Let $d\ge2$, $\Gamma_{s-}:=\overline{B}(0,R_0)$ and let $V_0(x)=V_0\ge0$ be a constant function on $\partial B(0,R_0)$. In addition, let $u(x)=u(|x|)$ be a radial function.

\begin{proposition}\label{prop:Example3Minimal}
A minimal solution $w$ to \eqref{PDE:wave}, with initial domain $\Gamma_{s-}$ and initial velocity $V_0$, is given by
\begin{align*}
w(x)=\int_{R_0}^{|x|}\frac{r^{d-1}}{(\gamma+V_0)R_0^{d-1}-\int_{R_0}^r(1+u(z))\,z^{d-1}\mathop{\mathrm{d}z}-\gamma r^{d-1}}\,\mathrm{d}r\,\bone_{\{R_0<|x|\leq R^*\}}(x)
+\infty\,\bone_{\{|x|>R^*\}}(x),
\end{align*}
where
\begin{align*}
R^*:=\inf\left\{r>R_0:\,\int_{R_0}^r(1+u(z))\,z^{d-1}\mathop{\mathrm{d}z}>(\gamma+V_0)R_0^{d-1}-\gamma r^{d-1}\right\}\in[R_0,\infty].   
\end{align*}
\end{proposition}
\begin{remark}
Similar to Proposition \ref{prop:Example1Minimal}, if $V_0=0$, then the above expression for $w(x)$ evaluates to $w(x)=\infty\,\bone_{\RR^d\setminus\Gamma_{s-}}(x)$.
\end{remark}

\begin{proof}
Take any $\varepsilon_n>0$ with $\varepsilon_n\to0$, and define
\begin{align*}
w^n(x):=\int_{R_0}^{|x|}\frac{r^{d-1}}{(\gamma+V_0+\varepsilon_n)R_0^{d-1}-\int_{R_0}^r(1+u(z))\,z^{d-1}\mathop{\mathrm{d}z}-\gamma r^{d-1}}\,\mathrm{d}r\,\bone_{\{R_0<|x|\leq R^{n*}\}}(x)
+\infty\,\bone_{\{|x|>R^{n*}\}}(x), 
\end{align*}
where
\begin{align*}
R^{n*}:=\inf\left\{r>R_0:\,\int_{R_0}^r(1+u(z))\,z^{d-1}\mathop{\mathrm{d}z}>(\gamma+V_0+\varepsilon_n)R_0^{d-1}-\gamma r^{d-1}\right\}\in[R_0,\infty].
\end{align*}
As it has been assumed that $u$ is bounded and that $(1+u)^-$ is integrable, $\lim_{|x|\uparrow R^{n*}}w^n(x)=\infty$. The sequence of functions $w^n$ satisfies the convergence
\begin{align*}
w^n\wedge T\to w\wedge T    
\end{align*}
pointwise and in $L_{\mathrm{loc}}^1$ for any $T>0$. Moreover, for each $n$, $w^n\in C^1(\{R_0\leq|x|<R^{n*}\})$ with
\begin{align*}
\nabla w^n(x)=\frac{|x|^{d-1}}{(\gamma+V_0+\varepsilon_n)R_0^{d-1}-\int_{R_0}^{|x|}(1+u(z))\,z^{d-1}\mathop{\mathrm{d}z}-\gamma |x|^{d-1}}\frac{x}{|x|},
\end{align*}
and thus 
\begin{align*}
\mathrm{div}\bigg(\frac{\nabla w^n(x)}{|\nabla w^n(x)|^2}+\gamma\frac{\nabla w^n(x)}{|\nabla w^n(x)|}\bigg)=-(u(|x|)+1)=-(u(x)+1)    
\end{align*}
holds a.e. and weakly in $\{R_0\leq|x|<R^{n*}\}$. Let $T_n$ be a sequence increasing to $\infty$. By Proposition \ref{prop:equi} (with $V_0$ replaced by $V_0+\varepsilon_n$ and $T$ replaced by $T_n$ for each $n\ge1$), there exists a sequence of equilibria $(\kappa^n,\QQ^n,L^n)$ with the properties that $w^n=w^{\QQ^n,L^n}$ in $\{0\leq w^n<T_n\}$ and that $\mathbf{E}^{\kappa^n,\QQ^n,L^n}_t=0$ for any $t\in[0,T_n)$. To justify the minimality of $w$, we take any $T\in(0,\infty)$ and any sequence $((\hat{\kappa}^n,\hat{\QQ}^n,\hat{L}^n))_{n\ge1}$ satisfying $\hat{\kappa}^n\,\hat{\QQ}^n(y_\theta\in \mathrm{d}x,\beta=\theta)=\kappa^n\,\QQ^n(y_\theta\in \mathrm{d}x,\beta=\theta)$, $\mathbf{e}^{\hat{\kappa}^n,\hat{\QQ}^n,\hat{L}^n}\geq0$, and $\mathbf{E}^{\hat{\kappa}^n,\hat{\QQ}^n,\hat{L}^n}_t\leq\mathbf{E}^{\kappa^n,\QQ^n,L^n}_t$ for all $t\in[0,T]$, and we need to show that
\begin{align}\label{eq:MinimalityCondition2Example3.1}
&\liminf_{n\rightarrow\infty}\left(\mathbf{E}^{\hat{\kappa}^n,\hat{\QQ}^n,\hat{L}^n}_t-\mathbf{E}^{\kappa^n,\QQ^n,L^n}_r\right)\geq0\quad \text{for any}\,\,0\leq r<t<T.
\end{align}
As $T_n>T$ for all sufficiently large $n$, the condition \eqref{eq:MinimalityCondition2Example3.1} holds trivially, showing that $w$ is a minimal solution.
\end{proof}

\begin{remark}\label{rem:min.lim.equil.3}
(a) Recall that in this example we have $w(x)=w(|x|)$ and $u(x)=u(|x|)$, and the PDE for $w$ becomes the ODE
\begin{align*}
rw''(r)-(d-1)w'(r)-(\gamma(d-1)+r(u(r)+1))w'(r)^2=0,    
\end{align*}
which admits a unique maximally defined solution, once we impose the initial conditions $w(R_0)=0$ and $w'(R_0)=V_0^{-1}$, for any $V_0>0$. It is easy to see that the latter solution coincides with the minimal solution $w$ constructed above.

\smallskip

\noindent(b) In this example, we cannot rule out the existence of other minimal solutions. However, we can apply the general local well-posedness result for quasi-linear second-order hyperbolic equations (e.g., \cite[Proposition 16.3.2]{TaylorIII}) to obtain that the minimal solution $w$ constructed in this example is the unique (maximally defined) smooth classical solution to the cascade PDE \eqref{PDE:wave} in $\{x\in\RR^d:|x|>R_0\}$, subject to the initial conditions (assuming $V_0>0$)
\begin{align*}
w(x)=0,\quad\partial_{\mathbf{n}}w(x)=\frac{1}{V_0}\quad\text{on}\quad\partial B(0,R_0),
\end{align*}
provided that $u$ is a smooth radial function.
%Indeed, it is straightforward to verify that $w$ is smooth and solves the equation. Suppose $\tilde w$ is also a maximally defined smooth classical solution to the same equation with the same initial condition on $\partial B(0,R_0)$. Then \cite[Proposition 16.3.2]{TaylorIII} implies the existence of a neighborhood $S$ of $\partial B(0,R_0)$ such that $\tilde w=w$ in $S$. In particular, we can find $R_1>R_0$ such that $\overline B(0,R_1)\subset S$. Now, consider 
%\begin{align*}
%R_*:=\sup\{R>R_0:\,\tilde w=w \quad\text{in}\quad B(0,R)\}\wedge R^*>R_0.  
%\end{align*}
%If $R_*=R^*$ then we are done as $\lim_{|x|\uparrow R^*}w(x)=\infty$, which means $w$ cannot be further extended beyond $B(0,R^*)$, and the same holds for $\tilde w$. If however $R_*<R^*$, then $\tilde w=w$ on $\overline B(0,R_*)$ and $\partial_{\mathbf{n}}\tilde w=\partial_{\mathbf{n}}w$ on $\partial B(0,R_*)$ by continuity and \cite[Proposition 16.3.2]{TaylorIII} can be applied again to the same equation with initial conditions given by the values of $w$, $\partial_{\mathbf{n}}w$ on $\partial B(0,R_*)$ to obtain the existence of a neighborhood $S'$ of $\partial B(0,R_*)$ such that $\tilde w=w$ in $S'$. In particular, we can find $R_1'>R_*$ such that $\tilde w=w$ in $B(0,R_1')$, contradicting the definition of $R_*$. Therefore, the uniqueness in the class of smooth functions is obtained.
\end{remark}

\subsection{Example: General Dimension with Radial Symmetry, Shrinking Water Ball}\label{ex:Radial2}
Let $d\ge2$, $\Gamma_{s-}:=\RR^d\setminus B(0,R_0)$, and let $V_0(x)=V_0\ge0$ be a constant function on $\partial B(0,R_0)$. In addition, let $u(x)=u(|x|)$ be a radial function.

\begin{proposition}\label{prop:Example4Minimal}
A minimal solution $w$ to \eqref{PDE:wave}, with initial domain $\Gamma_{s-}$ and initial velocity $V_0$, is given by
\begin{align*}
w(x)=\int_{|x|}^{R_0}\frac{r^{d-1}}{(\gamma+V_0)R_0^{d-1}-\int_r^{R_0}(1+u(z))\,z^{d-1}\mathop{\mathrm{d}z}-\gamma r^{d-1}}\,\mathrm{d}r\,\bone_{\{R_*\leq|x|<R_0\}}(x)
+\infty\,\bone_{\{|x|<R_*\}}(x),
\end{align*}
where
\begin{align*}
R_*:=\sup\left\{0<r<R_0:\,\int_r^{R_0}(1+u(z))\,z^{d-1}\mathop{\mathrm{d}z}>(\gamma+V_0)R_0^{d-1}-\gamma r^{d-1}\right\}\vee0,    
\end{align*}
provided that
\begin{align*}
\int_{R_*}^{R_0}\frac{r^{d-1}}{(\gamma+V_0)R_0^{d-1}-\int_r^{R_0}(1+u(z))\,z^{d-1}\mathop{\mathrm{d}z}-\gamma r^{d-1}}\,\mathrm{d}r=\infty.   
\end{align*}
\end{proposition}
\begin{remark}
Similar to Proposition \ref{prop:Example1Minimal}, if $V_0=0$, then the above expression for $w(x)$ evaluates to $w(x)=\infty\,\bone_{\RR^d\setminus\Gamma_{s-}}(x)$.
\end{remark}
\begin{proof}
Similar to the proof of Proposition \ref{prop:Example3Minimal}.
\end{proof}

\begin{remark}\label{rem:min.lim.equil.4}
(a) When $V_0=0$, $R_0-R_*$ above coincides exactly with the physical jump size (with a change of variable and phase) in \cite{NaShsurface}.

\smallskip

\noindent(b) If $R_*=0$ and 
\begin{align*}
\int_0^{R_0}(1+u(z))\,z^{d-1}\mathop{\mathrm{d}z}<(\gamma+V_0)R_0^{d-1},
\end{align*}
then necessarily
\begin{align*}
\int_0^{R_0}\frac{r^{d-1}}{(\gamma+V_0)R_0^{d-1}-\int_r^{R_0}(1+u(z))\,z^{d-1}\mathop{\mathrm{d}z}-\gamma r^{d-1}}\,\mathrm{d}r<\infty,    
\end{align*}
which means that the initial water ball vanishes in a finite time with some remaining boundary energy ``locked inside" the solid domain. When this is the case, $w$ is $C^1$ around $x=0$ but $\nabla w(x)=0$, which is a point of singularity for the PDE \eqref{PDE:wave}. In the sense of distributions, $w$ satisfies
\begin{small}
\begin{align*}
\mathrm{div}\bigg(\frac{\nabla w(x)}{|\nabla w(x)|^2}+\gamma\frac{\nabla w(x)}{|\nabla w(x)|}\bigg)=-(u(x)+1)-\bigg((\gamma+V_0)R_0^{d-1}-\int_0^{R_0}(1+u(z))\,z^{d-1}\mathop{\mathrm{d}z}\bigg)\cdot\delta_0\quad\text{in}\quad B(0,R_0),
\end{align*}
\end{small}
the last term on the right-hand side of which records the energy "locked in" at the origin.

\smallskip

\noindent(c) Similar to the argument in Remark \ref{rem:min.lim.equil.3}, if $u$ is a smooth radial function, then it can be shown that the minimal solution $w$ constructed in this example is the unique maximally defined smooth classical solution to the cascade PDE \eqref{PDE:wave} in $\{x\in\RR^d:R_*<|x|<R_0\}$, subject to the initial conditions (assuming $V_0>0$)
\begin{align*}
w(x)=0,\quad\partial_{\mathbf{n}}w(x)=\frac{1}{V_0}\quad\text{on}\quad\partial B(0,R_0).
\end{align*}
If, in addition, the assumption of the above part (b) is satisfied, then there is no classical solution to the PDE \eqref{PDE:wave} that is smooth everywhere in $B(0,R_0)$. This example illustrates that, even in the domain $\{x:\,0<w(x)<\infty\}$, the PDE \eqref{PDE:wave} can only be solved with ``$=$" replaced by ``$\leq$", or, equivalently, we must relax \eqref{PDE:wave} by allowing for an additional non-positive measure-valued summand in its right-hand side.
\end{remark}

\bigskip\bigskip\bigskip

%%%%%%%%%%%%%%%%%%%%%%%%%%%
\bibliographystyle{amsalpha}
\bibliography{Main}
%%%%%%%%%%%%%%%%%%%%%%%%%%%

\bigskip\bigskip\bigskip

\end{document}